\documentclass[11pt]{amsart} 

\usepackage{fullpage}
\usepackage{url,amssymb,amsmath,amsfonts,amsthm,mathrsfs}

\usepackage[all]{xy} \SelectTips{cm}{10}

\usepackage{color}

\definecolor{webcolor}{rgb}{0.8,0,0.2}
\definecolor{webbrown}{rgb}{.6,0,0}
\usepackage[
        colorlinks,
        linkcolor=webbrown,  filecolor=webcolor,  citecolor=webbrown, 
        backref,
        pdfauthor={David Zywina}, 
       pdftitle={},
]{hyperref}
\usepackage[alphabetic,backrefs,lite]{amsrefs} 

\numberwithin{equation}{section}

\renewcommand{\AA}{\mathbb A}

\newcommand{\CC}{\mathbb C}

\newcommand{\FF}{\mathbb F}

\newcommand{\PP}{\mathbb P}
\newcommand{\QQ}{\mathbb Q}
\newcommand{\RR}{\mathbb R}

\newcommand{\ZZ}{\mathbb Z} 
\newcommand{\Zhat}{\widehat\ZZ}

\newcommand{\calG}{\mathcal G}  \newcommand{\calF}{\mathcal F}
\newcommand{\calH}{\mathcal H}
\newcommand{\calI}{\mathcal I}
\newcommand{\calJ}{\mathcal J}

\newcommand{\calS}{\mathcal S}

\newcommand{\calA}{\mathcal A}

\newcommand{\calE}{\mathcal E}
\newcommand{\calX}{\mathcal X}

\newcommand{\calL}{\mathcal L}

\newcommand{\calM}{\mathcal M}
\newcommand{\calR}{\mathcal R}

\newcommand{\m}{\mathfrak m}

\newcommand{\scrA}{\mathscr A}
\newcommand{\scrB}{\mathscr B}

\newcommand{\scrC}{\mathscr C}
\newcommand{\scrE}{\mathscr E}

\newcommand{\scrL}{\mathscr L}
\newcommand{\scrS}{\mathscr S}
\newcommand{\scrF}{\mathscr F}

\def\cyc{{\operatorname{cyc}}}
\def\ab{{\operatorname{ab}}}

\def\Spec{\operatorname{Spec}}

\def\Gal{\operatorname{Gal}}
\def\ord{\operatorname{ord}} 
\def \GL {\operatorname{GL}}

\def \SL {\operatorname{SL}}

\def\Aut{\operatorname{Aut}} 

\def\Frob{\operatorname{Frob}}

\def\tr{\operatorname{tr}}

\newcommand{\brak}[1]{[\![ #1 ]\!]}

\newcommand{\defi}[1]{\textsf{#1}} 

\newcommand\blank[1]{}

\def\bbar#1{\setbox0=\hbox{$#1$}\dimen0=.2\ht0 \kern\dimen0 
\overline{\kern-\dimen0 #1}}
\newcommand{\Qbar}{{\overline{\mathbb Q}}} 
\newcommand{\Kbar}{{\bbar{K}}}

\newcommand{\FFbar}{\overline{\FF}} 

\newtheorem{thm}{Theorem}[section]
\newtheorem{lemma}[thm]{Lemma}
\newtheorem{cor}[thm]{Corollary}
\newtheorem{prop}[thm]{Proposition}
\newtheorem{conj}[thm]{Conjecture}

\newtheorem*{programb}{Mazur's Program B}

\theoremstyle{definition}
\newtheorem{definition}[thm]{Definition}
\newtheorem{algorithm}[thm]{Algorithm}

\theoremstyle{remark}
\newtheorem{remark}[thm]{Remark}
\newtheorem{example}[thm]{Example}

\newenvironment{romanenum}{\hfill \begin{enumerate} }{\end{enumerate}}
\newenvironment{alphenum}{\hfill \begin{enumerate} }{\end{enumerate}}

\begin{document}

\title{Explicit open images for elliptic curves over $\QQ$}
\author{David Zywina}
\address{Department of Mathematics, Cornell University, Ithaca, NY 14853, USA}
\email{zywina@math.cornell.edu}

\begin{abstract}
For a non-CM elliptic curve $E$ defined over $\QQ$, the Galois action on its torsion points gives rise to a Galois representation $\rho_E\colon \Gal(\Qbar/\QQ)\to \GL_2(\Zhat)$ that is unique up to isomorphism.   A renowned theorem of Serre says that the image of $\rho_E$ is an open, and hence finite index, subgroup of $\GL_2(\Zhat)$.   We describe an algorithm that computes the image of $\rho_E$ up to conjugacy in $\GL_2(\Zhat)$; this algorithm is practical and has been implemented.   Up to a positive answer to a uniformity question of Serre and finding all the rational points on a finite number of explicit modular curves of genus at least $2$, we give a complete classification of the groups $\rho_E(\Gal(\Qbar/\QQ))\cap \SL_2(\Zhat)$ and the indices $[\GL_2(\Zhat):\rho_E(\Gal(\Qbar/\QQ))]$ for non-CM elliptic curves $E/\QQ$.   Much of the paper is dedicated to the efficient computation of modular curves via modular forms expressed in terms of Eisenstein series.
\end{abstract}

\maketitle

\section{Introduction}

\subsection{Serre's open image theorem}

Consider an elliptic curve $E$ defined over $\QQ$.  For each integer $N>1$, let $E[N]$ be the $N$-torsion subgroup of $E(\Qbar)$, where $\Qbar$ is a fixed algebraic closure of $\QQ$.  The group $E[N]$ is a free $\ZZ/N\ZZ$-module of rank $2$.    There is a natural action of the absolute Galois group $\Gal_\QQ:=\Gal(\Qbar/\QQ)$ on $E[N]$ that respects the group structure and which we may express in terms of a representation 
\[
\rho_{E,N}\colon \Gal_\QQ\to \Aut(E[N])\cong \GL_2(\ZZ/N\ZZ).
\]
By choosing compatible bases and taking the inverse limit, we can combine these representations into a single representation 
\[
\rho_E\colon \Gal_\QQ \to \GL_2(\Zhat)
\]
that encodes the Galois action on all the torsion points of $E$. Here the ring $\Zhat$ is the profinite completion of $\ZZ$.  The representation $\rho_E$ is uniquely determined up to isomorphism and hence the image $\rho_E(\Gal_\QQ)$ is uniquely determined up to conjugacy in $\GL_2(\Zhat)$.   

With respect to the profinite topology, we find that $\rho_E(\Gal_\QQ)$ is a closed subgroup of the compact group $\GL_2(\Zhat)$. In \cite{Serre-Inv72}, Serre proved the following theorem which says that, up to finite index, the image of $\rho_E$ is as large as possible when $E$ is non-CM  (it was actually shown for elliptic curves over a general number field, but we will restrict our attention to the rationals).  

\begin{thm}[Serre's open image theorem] \label{T:Serre 1972}
Let $E$ be a non-CM elliptic curve defined over $\QQ$.  Then $\rho_E(\Gal_\QQ)$ is an open subgroup of $\GL_2(\Zhat)$.  Equivalently, $\rho_E(\Gal_\QQ)$ is a finite index subgroup of $\GL_2(\Zhat)$.
\end{thm}

The group $\rho_E(\Gal_\QQ)$, when known, will have a simple description since it is open in $\GL_2(\Zhat)$, i.e., it is given by its level $N$ and a set of generators for its image modulo $N$ in $\GL_2(\ZZ/N\ZZ)$.  For a definition of the level and other conventions see \S\ref{SS:notation}.  Unfortunately, Serre's proof is in general ineffective.  

The goal of this work is to explain how, given a non-CM elliptic curve $E/\QQ$, we can compute the group $\rho_E(\Gal_\QQ)$ up to conjugacy in $\GL_2(\Zhat)$.    The algorithm we obtain is practical.  For example, we have used it to compute the image of $\rho_E$, up to conjugacy, for all non-CM elliptic curves $E/\QQ$ with conductor up to $500000$ (on the machine we ran it on, it took on average $0.015$ seconds per curve); these images are publicly available in the L-Functions and Modular Forms Database
(LMFDB) \cite{lmfdb}.  Our algorithms are implemented in \texttt{Magma} \cite{Magma} and the code can be found in the public repository \cite{github}. 

A large part of Serre's paper \cite{Serre-Inv72} is dedicated to showing that $\rho_{E,\ell}$ is surjective for all sufficiently large primes $\ell$.   Serre  asked whether there is a constant $C$, not depending on $E$, such that $\rho_{E,\ell}$ is surjective for all primes $\ell > C$, cf.~\cite[\S4.3]{Serre-Inv72}.  Moreover, he asks whether $\rho_{E,\ell}$ is surjective for all $\ell>37$ \cite[p.~399]{MR644559}.   We pose as a conjecture a slightly stronger version (it was conjectured independently in \cite{possibleimages} and \cite{MR3482279}). We denote the $j$-invariant of $E$ by $j_E$.

\begin{conj} \label{C:Serre question}
If $E$ is a non-CM elliptic curve over $\QQ$ and $\ell>13$ is a prime, then either $\rho_{E,\ell}(\Gal_\QQ)=\GL_2(\ZZ/\ell\ZZ)$ or 
\[
(\ell,j_E) \in \big\{\, (17, -17^2 \!\cdot\! 101^3/2), \,(17,-17\!\cdot\! 373^3/2^{17}),\, (37,-7\!\cdot\! 11^3),\, (37,-7\!\cdot\! 137^3\!\cdot\! 2083^3) \,\big\}.
\]
\end{conj}

Assuming Conjecture~\ref{C:Serre question}, one can show that the indices $[\GL_2(\Zhat):\rho_{E}(\Gal_\QQ)]$ are uniformly bounded as we vary over all non-CM elliptic curves $E/\QQ$, cf.~\cite[Theorem~1.3]{possibleindices}.    Based on the computations arising in this paper, we make the following prediction.

\begin{conj} \label{C:index bounds}
We have $[\GL_2(\Zhat):\rho_{E}(\Gal_\QQ)]\leq 2736$ for all non-CM elliptic curve $E$ over $\QQ$.
\end{conj}

\begin{remark}
An elliptic curve $E/\QQ$ with $j$-invariant $-7\!\cdot\! 11^3$ or $-7\!\cdot\! 137^3\!\cdot\! 2083^3$ satisfies $[\GL_2(\Zhat):\rho_{E}(\Gal_\QQ)]= 2736$.  Such non-CM elliptic curves $E/\QQ$ are special because they have an isogeny of degree $37$ defined over $\QQ$.  In particular, the upper bound in Conjecture~\ref{C:index bounds} would be best possible.
\end{remark}

We now make a braver conjecture on the possible values of the index $[\GL_2(\Zhat) : \rho_E(\Gal_\QQ)]$.  This conjecture holds assuming Conjecture~\ref{C:Serre question} and assuming that we have not missed any rational points on the high genus modular curves that arise in our computations.

\begin{conj} \label{C:brave}
If $E$ is a non-CM elliptic curve defined over $\QQ$, then $[\GL_2(\Zhat) : \rho_E(\Gal_\QQ)]$ lies in the set
\[
 \left\{\begin{array}{c} 2, 4, 6, 8, 10, 12, 16, 20, 24, 30, 32, 36, 40, 48, 54, 60, 72, 80, 84,
 96, 108,\\ 112, 120, 128, 144, 160, 182, 192, 200, 216, 220, 224, 240, 288, 300, 
336, \\360, 384, 480, 504, 576, 768, 864, 1152, 1200, 1296, 1536, 2736 \end{array}\right\}.
\]
\end{conj}

\begin{remark}
All of the integers in the set from Conjecture~\ref{C:brave} actually occur as an index $[\GL_2(\Zhat) : \rho_E(\Gal_\QQ)]$ for some non-CM elliptic curve $E/\QQ$.  
\end{remark}

In our arguments, it will often be convenient to work with the dual representation 
\[
\rho_E^*\colon \Gal_\QQ\to \GL_2(\Zhat)
\]
of $\rho_E$, i.e., $\rho_E^*(\sigma)$ is the transpose of $\rho_E(\sigma^{-1})$.  Similarly, we can define $\rho_{E,N}^*(\sigma)$ to be the transpose of $\rho_{E,N}(\sigma^{-1})$.  Of course, computing the images of $\rho_E^*$ and $\rho_E$ are equivalent problems and their images have the same index in $\GL_2(\Zhat)$.

For a prime $\ell$, let $\rho_{E,\ell^\infty}\colon \Gal_\QQ \to \GL_2(\ZZ_\ell)$ be the representation obtained by taking the inverse limit of the $\rho_{E,\ell^n}$; equivalently, $\rho_{E,\ell^\infty}$ is the composition of $\rho_E$ with the $\ell$-adic projection.

\subsection{The Kronecker--Weber constraint on the image}

For a fixed non-CM elliptic curve $E/\QQ$, consider the group $G_E :=\rho_E^*(\Gal_\QQ) \subseteq \GL_2(\Zhat)$.    The group $G_E$ is open in $\GL_2(\Zhat)$ by Theorem~\ref{T:Serre 1972}.   We have $\det(G_E)=\Zhat^\times$ since $\det\circ \rho_E^*=\chi_\cyc^{-1}$, where $\chi_\cyc \colon \Gal_\QQ \to \Zhat^\times$ is the \defi{cyclotomic character}, cf.~\S\ref{S:Galois first}.    
 
 We also have the following important constraint on $G_E$ that arises from the Kronecker--Weber theorem.  For a group $G$, we will denote its commutator subgroup by $[G,G]$.
 
 \begin{lemma} \label{L:KW}
We have $G_E \cap \SL_2(\Zhat) = [G_E,G_E]$. In particular, 
\[
[\GL_2(\Zhat): G_E]=[\SL_2(\Zhat): [G_E,G_E]].
\]
\end{lemma}
\begin{proof}
Let $\QQ^\ab\subseteq \Qbar$ be the maximal abelian extension of $\QQ$.   Since $\Gal(\Qbar/\QQ^\ab)$ is the commutator subgroup of $\Gal_\QQ$, we  have $\rho_E^*(\Gal(\Qbar/\QQ^\ab))= [G_E,G_E]$.   By the Kronecker--Weber theorem, $\QQ^\ab$ is the cyclotomic extension of $\QQ$.  Since $\chi_\cyc^{-1}=\det\circ \rho_E^*$, we deduce that $G_E\cap \SL_2(\Zhat)=\rho_E^*(\Gal(\Qbar/\QQ^\ab))$.  We obtain $G_E\cap \SL_2(\Zhat)=[G_E,G_E]$ by comparing our two descriptions of $\rho_E^*(\Gal(\Qbar/\QQ^\ab))$.  Since $\det(G_E)=\Zhat^\times$, we have $[\GL_2(\Zhat):G_E]=[\SL_2(\Zhat):G_E \cap \SL_2(\Zhat)]$.  The lemma is now immediate.
\end{proof}

\begin{example} \label{ex:not surjective}
The commutator subgroup of $\GL_2(\Zhat)$ is an index $2$ subgroup of $\SL_2(\Zhat)$, cf.~Lemma~\ref{L:basics commutators}, so the index of $[G_E,G_E]$ in $\SL_2(\Zhat)$ is even.   Therefore, the index of $G_E$ in $\GL_2(\Zhat)$ is even by Lemma~\ref{L:KW}.  In particular, $G_E\neq \GL_2(\Zhat)$; this was first observed by Serre, cf.~Proposition~22 of \cite{Serre-Inv72}.   Moreover, the image of $G_E$ lies in a specific index $2$ subgroup of $\GL_2(\Zhat)$, cf.~\S\ref{SS:Serre curves}.
\end{example}

\subsection{Modular curves}
 
 Let $E/\QQ$ be a non-CM elliptic curve and set $G_E :=\rho_E^*(\Gal_\QQ)$.    Our main tool for studying the group $G_E$ is the theory of \emph{modular curves}.   
 
 Consider any open subgroup $G$ of $\GL_2(\Zhat)$ that satisfies $\det(G)=\Zhat^\times$ and $-I \in G$.       Associated to $G$, we will define a \defi{modular curve} $X_G$, cf.~\S\ref{S:first modular curve}.  The modular curve $X_G$ is a smooth, projective and geometrically irreducible curve defined over $\QQ$ that comes with a morphism 
\[
\pi_G\colon X_G\to \PP_\QQ^1=\AA^1_\QQ \cup\{\infty\}.
\]   

For our applications to Serre's open image theorem, the {key property} of the curve $X_G$ is that $G_E$ is conjugate in $\GL_2(\Zhat)$ to a subgroup of $G$ if and only if the $j$-invariant $j_E$ of $E$ lies in the set $\pi_G(X_G(\QQ)) \subseteq \QQ\cup\{\infty\}$.   We say that a point $P\in X_G(\QQ)$ is \defi{non-CM} if $\pi_G(P) \in \QQ \cup\{\infty\}$ is the $j$-invariant of a non-CM elliptic curve.

The pairs $(X_G,\pi_G)$, as we vary over all $G$, will thus determine the image of $G_E$ in $\GL_2(\Zhat)/\{\pm I\}$ up to conjugacy.  However, this is an impractical approach for finding $G_E$ since there are \emph{infinitely many} groups $G$ to consider.   In fact, infinitely many open subgroups of $\GL_2(\Zhat)$ can arise as $G_E$ as we vary over all non-CM elliptic curve $E/\QQ$.

In \S\ref{S:computing modular forms}, we will describe a method for computing an explicit model for the modular curve $X_G$ given a group $G$. We are also interested in computing $\pi_G$ in terms of our model.   Our approach to computing modular curves is via related spaces of modular forms that we study in \S\ref{S:modular forms}.  Our application involves computing thousands of modular curves, so we are especially interested in finding efficient techniques.

\subsection{Agreeable closures} \label{SS:agreeable intro}

Instead of computing $G_E=\rho_E^*(\Gal_\QQ)$ directly, we first find a larger and friendlier group.  We say that a subgroup $\calG$ of $\GL_2(\Zhat)$ is  \defi{agreeable} if it is open in $\GL_2(\Zhat)$,  has full determinant, contains all the scalar matrices, and the levels of $\calG$ and $\calG\cap \SL_2(\Zhat)$ in $\GL_2(\Zhat)$ and $\SL_2(\Zhat)$, respectively,  have the same odd prime divisors.  

 There is a unique minimal agreeable subgroup $\calG_E$ satisfying $G_E\subseteq \calG_E$ which we call the \defi{agreeable closure} of $G_E$, cf.~\S\ref{S:agreeable}.  The group $G_E$ is normal in $\calG_E$ and the quotient group $\calG_E/G_E$ is finite and abelian, cf.~Proposition~\ref{P:agreeable closure}.   

We claim that $[G_E,G_E]=[\calG_E,\calG_E]$.   We have $[\calG_E,\calG_E] \subseteq G_E$ since $\calG_E/G_E$ is abelian and hence $[\calG_E,\calG_E] \subseteq G_E \cap \SL_2(\Zhat)$.   By Lemma~\ref{L:KW}, this gives the inclusion $[\calG_E,\calG_E]\subseteq [G_E,G_E]$.   The claim follows since the other inclusion is a consequence of $G_E \subseteq \calG_E$.

In particular,  the agreeable group $\calG_E$ determines $G_E\cap \SL_2(\Zhat)=[\calG_E,\calG_E]$, up to conjugation in $\GL_2(\Zhat)$, and also determines the index $[\GL_2(\Zhat): G_E]=[\SL_2(\Zhat):[\calG_E,\calG_E]]$.  The advantage of agreeable groups is that are far fewer of them to consider.   In fact, if Conjecture~\ref{C:Serre question} holds for $E$, then any prime dividing the level of $\calG_E$ in $\GL_2(\Zhat)$ must lie in the set  
\[
\calL:=\{2,3,5,7,11,13,17,37\},
\]
cf.~Lemma~\ref{L:level of agreeable and primes}.  From this, one can show that there are only \emph{finitely many} agreeable groups of the form $\calG_E$ as we vary over \emph{all} non-CM elliptic curves $E/\QQ$ for which Conjecture~\ref{C:Serre question} holds.   

The following theorem summarizes some details of our computations.

\begin{thm}  \label{T:main agreeable}
We can compute a finite set $\scrA$ of agreeable subgroups that are pairwise non-conjugate in $\GL_2(\Zhat)$ and satisfy the following conditions:
\begin{alphenum}
\item \label{I:prop of goal a}
For every group $\calG\in \scrA$, the level of $\calG$ is not divisible by any prime $\ell\notin \calL$.
\item \label{I:prop of goal b}
Let $G$ be any agreeable subgroup of $\GL_2(\Zhat)$ for which the level of $G$ is divisible only by primes in the set $\calL$ and for which $X_G(\QQ)$ has a non-CM point.
\begin{itemize}
\item
If $X_G(\QQ)$ is infinite, then $G$ is conjugate in $\GL_2(\Zhat)$ to some group $\calG\in \scrA$.
\item
If $X_G(\QQ)$ is finite, then $G$ is conjugate in $\GL_2(\Zhat)$ to a subgroup of some $\calG\in \scrA$ with $X_\calG(\QQ)$ finite.
\end{itemize}
\item  \label{I:prop of goal c}
If $\calG\in \scrA$ is a group for which $X_{\calG}(\QQ)$ is finite, then $X_G(\QQ)$ is infinite for all agreeable groups $\calG\subsetneq G \subseteq \GL_2(\Zhat)$.
\item \label{I:prop of goal d}
If $\calG\in \scrA$ is a group for which $X_{\calG}$ has genus at most $1$, then $X_\calG(\QQ)$ has a non-CM point.
\item \label{I:prop of goal e}
For any group $\calG\in \scrA$ for which  $X_\calG$ has genus at most $1$, we can compute a model for the curve $X_{\calG}$ and, with respect to this model, compute the morphism $\pi_{\calG}$ from $X_{\calG}$ to the $j$-line. 
\item \label{I:prop of goal f}
For any group $\calG\in \scrA$ and rational number $j \in \QQ-\{0,1728\}$, we can determine whether  $\pi_{\calG}(P)=j$ for some $P\in X_{\calG}(\QQ)$. 
\end{alphenum}
\end{thm}

Our set $\scrA$ contains $315$ groups $\calG$ for which $X_\calG$ has genus $0$ (and hence $X_\calG$ is isomorphic to $\PP^1_\QQ$ since it has a rational point).  Our set $\scrA$ contains {$139$} and {$17$} groups $\calG$ for which $X_\calG$ has genus $1$ and $X_\calG(\QQ)$ is infinite or finite, respectively.  

Our original set $\scrA$ constructed contained thousands of groups $\calG$ for which $X_\calG$ has genus at least $2$.    For each such group $\calG$, $X_\calG(\QQ)$ is finite by Faltings; unfortunately, $X_\calG(\QQ)$ can sometimes be extremely difficult to compute.   Observe that whenever one can show that $X_\calG(\QQ)$ has no non-CM points, then we can remove $\calG$ from $\scrA$.  In our set $\scrA$, we know of $53$ groups $\calG$ for which $X_\calG$ has genus at least $2$ and $X_\calG(\QQ)$ has a non-CM point; these give rise to $81$ exceptional $j$-invariants of non-CM elliptic curves. \\

We will now outline how to compute $\calG_E$, up to conjugacy in $\GL_2(\Zhat)$, for a fixed non-CM elliptic curve $E/\QQ$.   Again recall that once we know $\calG_E$, we can then compute $G_E \cap \SL_2(\Zhat)=[\calG_E,\calG_E]$ up to conjugacy in $\GL_2(\Zhat)$ (see \S\ref{SSS:computing commutator subgroups} for how to compute commutator subgroups).    We can thus also compute the index $[\GL_2(\Zhat): \rho_E(\Gal_\QQ)]=[\GL_2(\Zhat):G_E]= [\SL_2(\Zhat): [\calG_E,\calG_E]]$.\\

Consider the case where Conjecture~\ref{C:Serre question} holds for $E/\QQ$ and $j_E$ is not in the finite set 
\begin{align}\label{E:calJ new}
\calJ := \bigcup_{G\in \scrA,\, X_G(\QQ) \text{ finite}} \pi_G(X_G(\QQ)).
\end{align}
We can verify whether $E/\QQ$ satisfies these conditions by using the algorithm from \cite{surjectivityalgorithm} and Theorem~\ref{T:main agreeable}(\ref{I:prop of goal f}).   (For non-CM elliptic curves over $\QQ$ that do not satisfy these conditions, we will take any alternate and more direct approach later.)

Let us now explain how Theorem~\ref{T:main agreeable}  allows us to compute the group $\calG_E$ up to conjugacy; more details can be found in \S\ref{S:finding the agreeable closure}.   We have $j_E \in \pi_{\calG_E}(X_{\calG_E}(\QQ))$ since $\rho_E^*(\Gal_\QQ)\subseteq \calG_E$.  In particular, $X_{\calG_E}(\QQ)$ contains a non-CM point.  Our assumption that Conjecture~\ref{C:Serre question} holds for $E$ implies that the level of $\calG_E$ is not divisible by any prime $\ell \notin \calL$, cf.~\S\ref{SS:closure in most cases}.    If $\calG_E$ is conjugate in $\GL_2(\Zhat)$ to a subgroup of some group $\calG\in \scrA$ with $X_\calG(\QQ)$ finite, then we have $j_E \in \pi_\calG(X_\calG(\QQ))$ which contradicts $j_E\notin \calJ$.  Therefore, $\calG_E$ is not conjugate in $\GL_2(\Zhat)$ to a subgroup of any $\calG\in \scrA$ with $X_\calG(\QQ)$ finite.   Applying Theorem~\ref{T:main agreeable}(\ref{I:prop of goal b}), we deduce that $\calG_E$ is conjugate in $\GL_2(\Zhat)$ to a unique group $\calG\in \scrA$.   So let $\calG$ be a group in $\scrA$ with maximal index in $\GL_2(\Zhat)$ amongst those that satisfy $j_E \in \pi_\calG(X_\calG(\QQ))$; this can be found using Theorem~\ref{T:main agreeable}(\ref{I:prop of goal f}).   Then the explicit group $\calG$ is conjugate in $\GL_2(\Zhat)$ to the agreeable closure $\calG_E$ of $G_E$.

\subsection{Finding the image of Galois}  \label{SS:finding Galois}

Let $E$ be a non-CM elliptic curve over $\QQ$.  Suppose that we have found an agreeable subgroup $\calG$ of $\GL_2(\Zhat)$ such that, after possibly conjugating $G_E:=\rho^*_E(\Gal_\QQ)$ in $\GL_2(\Zhat)$, $G_E$ is a subgroup of $\calG$ satisfying $G_E\cap \SL_2(\Zhat)=[\calG,\calG]$.   In particular, $G_E$ is a normal subgroup of $\calG$ and $\calG/G_E$ is finite and abelian.  As noted in \S\ref{SS:agreeable intro}, the agreeable closure $\calG_E$ of $G_E$ will satisfy these properties and is computable.

Choose an open subgroup $G$ of $\calG$ satisfying $\det(G)=\Zhat^\times$ and $G\cap \SL_2(\Zhat)=[\calG,\calG]$.  Note that such a subgroup $G$ exists since the (unknown) group $G_E$ will have these properties.  In practice, we choose $G$ with minimal possible level.   Note that $G$ is a normal subgroup of $\calG$ and that $\calG/G$ is finite and abelian.  

Let $\alpha_E\colon \Gal_\QQ  \to \calG/G$ be the homomorphism that is the composition of $\rho_E^*\colon \Gal_\QQ \to G_E \subseteq \calG$ with the quotient map $\calG\to \calG/G$.  Since $\calG/G$ is abelian, there is a unique homomorphism 
\[
\gamma_E \colon \Zhat^\times \to \calG/G
\]
satisfying $\gamma_E(\chi_\cyc(\sigma)^{-1})=\alpha_E(\sigma)$ for all $\sigma\in \Gal_\QQ$.  Now define
\begin{align} \label{E:calHE intro def}
\calH_E:=\{ g \in \calG :  g\cdot G = \gamma_E(\det g)\};
\end{align}
it is a subgroup of $\GL_2(\Zhat)$ and the following lemma shows that it is the image of $\rho_E^*$ up to conjugacy.

\begin{lemma} \label{L:intro HE}
The groups $G_E=\rho_E^*(\Gal_\QQ)$ and $\calH_E$ are conjugate in $\GL_2(\Zhat)$.
\end{lemma}
\begin{proof}
For any $\sigma\in \Gal_\QQ$, we have
\[
\rho_E^*(\sigma) \cdot G =\alpha_E(\sigma)=\gamma_E(\chi_\cyc(\sigma)^{-1}) = \gamma_E( \det(\rho_E^*(\sigma)) ).
\]
In particular, $G_E$ is a subgroup of $\calH_E$.   By assumption, we have $G_E\cap \SL_2(\Zhat)=[\calG,\calG]$ and hence
\[
\calH_E \cap \SL_2(\Zhat) = G \cap \SL_2(\Zhat) = [\calG,\calG] = G_E \cap \SL_2(\Zhat).
\]    
Since $G_E$ is a subgroup of $\calH_E$ with $G_E \cap \SL_2(\Zhat)= \calH_E \cap \SL_2(\Zhat)$ and $\det(G_E)=\Zhat^\times$, we deduce that $G_E=\calH_E$.
\end{proof}

Since $G_E$ is conjugate in $\GL_2(\Zhat)$ to the group $\calH_E$, computing the image of $\rho_E^*$ reduces to the problem of finding $\gamma_E$.

\begin{remark}
The notation $\alpha_E$, $\gamma_E$ and $\calH_E$ suppresses the dependence on the choice of conjugate of $\rho_E^*$ for which $\rho_E^*(\Gal_\QQ)\subseteq \calG$.   However, the group $\calH_E$, up to conjugacy in $\GL_2(\Zhat)$, depends only on $E$.
\end{remark}

We now give an alternate description of $\gamma_E$ and an overview of how we will compute it.

Define the open subvariety $U_\calG:=X_\calG - \pi_\calG^{-1}(\{0,1728,\infty\})$ of $X_\calG$.    In \S\ref{S:finding image given agreeable closure}, we shall describe a particular \'etale cover $\phi\colon Y\to U_\calG$ that is Galois with group $\calG/G$.  When $-I\in G$, we will have $Y=U_G \subseteq X_G$ and $\phi$ will be the natural morphism.   The main task in \S\ref{S:finding image given agreeable closure} is computing models for $Y$ and $U_\calG$, with the action of $\calG/G$ on $Y$, and finding the corresponding $\phi$ with respect to these models.  Note that from \S\ref{SS:agreeable intro}, we need only consider a finite number of agreeable groups $\calG$ if we restrict to non-CM $E/\QQ$ for which Conjecture~\ref{C:Serre question} holds.

Since $G_E$ is conjugate in $\GL_2(\Zhat)$ to a subgroup of $\calG$  and we have a model for $U_\calG$, we can choose an explicit rational point $u\in U_\calG(\QQ)$ such that $\pi_\calG(u)=j_E$.   The fiber $\phi^{-1}(u)\subseteq Y(\Qbar)$ has a simply transitive $\calG/G$-action.  We also have a $\Gal_\QQ$-action on $\phi^{-1}(u)$ since $\phi$ and $u$ are defined over $\QQ$.  The actions of $\calG/G$ and $\Gal_\QQ$ on $\phi^{-1}(u)$ commute.    Fix an element $y\in Y(\Qbar)$ with $\phi(y)=u$.    For each $\sigma\in \Gal_\QQ$, we have
\[
\sigma(y)=\alpha_u(\sigma)\cdot y
\]
for a unique $\alpha_u(\sigma) \in \calG/G$.   In this manner, we obtain a homomorphism $\alpha_u\colon \Gal_\QQ \to \calG/G$.
Since $\calG/G$ is abelian, $\alpha_u$ does not depend on the choice of $y$.   Our models produce an explicit description of $\alpha_u$.    For any sufficiently large prime $p$, by reducing our models modulo $p$ we will be able to verify that $\alpha_u$ is unramified at $p$ and also compute $\alpha_u(\Frob_p)\in \calG/G$.  

After possibly replacing $\rho_E^*$ by an isomorphic representation that still satisfies $G_E:=\rho_E^*(\Gal_\QQ)\subseteq \calG$ and $G_E\cap \SL_2(\Zhat)=[\calG,\calG]$, we will show in \S\ref{SS:finding image setup} that $\alpha_E = \chi_d\cdot \alpha_u$, where $\chi_d\colon \Gal_\QQ \to \Gal(\QQ(\sqrt{d})/\QQ)\hookrightarrow \{\pm 1\}$ is the quadratic character arising from a certain squarefree integer $d$ (the $d$ is chosen so that the quadratic twist of $E$ by $d$ is isomorphic to an explicit elliptic curve over $\QQ$ with the same $j$-invariant).   With this $\alpha_E$ and any sufficiently large prime $p$, we can verify that $\chi_d$ and $\alpha_u$ are unramified at $p$ and compute $\alpha_E(\Frob_p) = \chi_d(\Frob_p) \alpha_u(\Frob_p)$.

The homomorphism $\alpha_E$ factors through $\rho_{E,N}^*$, where $N$ is the level of $G$.   Therefore, $\alpha_E$ is unramified at all primes $p\nmid M$, where $M$ is the product of $N$ and the primes $p$ for which $E$ has bad reduction.  So the corresponding $\gamma_E$ factors through a homomorphism
\[
\bbar\gamma_E \colon \ZZ_M^\times/(\ZZ_M^\times)^e \to \calG/G,
\]
where $e$ is the exponent of the group $\calG/G$ and $\ZZ_M=\prod_{\ell|M}\ZZ_\ell$.   So to compute $\gamma_E$, it suffices to find $\bbar\gamma_E(p \cdot(\ZZ_M^\times)^e )=\alpha_E(\Frob_p)^{-1} \in \calG/G$ for a finite set of primes $p\nmid M$ that generate the finite group $\ZZ_M^\times/(\ZZ_M^\times)^e$.    

So by computing $\alpha_E(\Frob_p)$ for enough primes $p\nmid M$, we obtain $\gamma_E$ and hence can compute the group $\calH_E$.     This will complete the computation of  $\rho_E^*(\Gal_\QQ)$, up to conjugacy  in $\GL_2(\Zhat)$, since it is conjugate to $\calH_E$.   

\begin{remark} \label{R:let us make clear}
Let us make clear what we mean that $\calH_E$ is computable.  That we have computed $\gamma_E$ means we have a positive integer $D\geq 1$ such that $\gamma_E$ factors through an explicit homomorphism $(\ZZ/D\ZZ)^\times \to \calG/G$.  We also know the groups $\calG$ and $G$, i.e., we have an integer $N\geq 1$ that is divisible by the levels of $\calG$ and $G$, and we have a set of generators for the image in $\GL_2(\ZZ/N\ZZ)$ of $\calG$ and $G$.    Let $N'$ be the least common multiple of $N$ and $D$.   Then $\calH_E$ is an open subgroup of $\GL_2(\Zhat)$ whose level divides $N'$ and we can find explicit generators for the image of $\calH_E$ in $\GL_2(\ZZ/N'\ZZ)$ under reduction modulo $N'$.
\end{remark}

\subsection{Example: Serre curves} \label{SS:Serre curves}

Let us consider the largest agreeable group $\calG:=\GL_2(\Zhat)$. The commutator subgroup $[\calG,\calG]$ is the unique subgroup of $\SL_2(\Zhat)$ with level $2$ and index $2$, cf.~Lemma~\ref{L:basics commutators}.   Let $G$ be the unique subgroup of $\calG=\GL_2(\Zhat)$ with level $2$ and index $2$.  We have $G\cap \SL_2(\Zhat)=[\calG,\calG]$,  $\det(G)=\Zhat^\times$, and $\calG/G$ is cyclic of order $2$.

  For a squarefree integer $d$, let $\gamma_d\colon \Zhat^\times\to  \calG/G$ be the unique homomorphism for which $\Gal_\QQ\to \calG/G$, $\sigma\mapsto \gamma_d( \chi_\cyc(\sigma)^{-1})$  factors through $\Gal(\QQ(\sqrt{d})/\QQ)\hookrightarrow \calG/G$.  After fixing an isomorphism $\calG/G\cong \{\pm1\}$, $\gamma_d$ factors through the \emph{Kronecker character} of $\QQ(\sqrt{d})$.    Define the group 
\[
G_{\gamma_d} := \{g \in \calG :  g\cdot G = \gamma_d(\det g) \}.
\]
Observe that $G_{\gamma_d}$ is an open subgroup of $\calG=\GL_2(\Zhat)$ with index $2$ that satisfies $\det(G_{\gamma_d})=\Zhat^\times$ and $G_{\gamma_d}\cap \SL_2(\Zhat)= G\cap \SL_2(\Zhat)=[\calG,\calG]$.    We have $G_{\gamma_1}=G$.\\

Now consider any non-CM elliptic curve $E/\QQ$ and set $G_E:=\rho_E^*(\Gal_\QQ)$.  
As noted in Remark~\ref{ex:not surjective}, the index of $G_E$ in $\GL_2(\Zhat)$ is always divisible by $2$. Moreover, we can show that $G_E$ is contained in an explicit index $2$ subgroup of $\GL_2(\Zhat)$.  Let $\alpha_E\colon \Gal_\QQ \to \calG/G$ be the composition of $\rho_E^*$ with the obvious quotient map.   Let $\gamma_E \colon \Zhat^\times \to \calG/G$  be the  homomorphism  satisfying $\gamma_E(\chi_\cyc(\sigma)^{-1})=\alpha_E(\sigma)$ for all $\sigma\in \Gal_\QQ$.    We have $\gamma_E = \gamma_d$ for a unique squarefree integer $d$ and hence 
\[
G_E\subseteq G_{\gamma_d}
\] 
since $\rho_E^*(\sigma)\cdot G = \alpha_E(\sigma) = \gamma_E(\chi_\cyc(\sigma)^{-1})=\gamma_d(\det \rho_E^*(\sigma))$.  Since $G$ has level $2$, the integer $d$ can be found by studying the $2$-torsion of $E$.  A direct computation shows that $d \in \Delta\cdot (\QQ^\times)^2$, where $\Delta$ is the discriminant of a Weierstrass model of $E/\QQ$.    One can  show that $\Delta\cdot (j_E-1728)$ is always a square in $\QQ$ so $d\in (j_E-1728)\cdot (\QQ^\times)^2$.

Following Lang and Trotter \cite{MR0568299}, we say that a (non-CM) $E/\QQ$ is a \defi{Serre curve} if $[\GL_2(\Zhat): G_E]=2$.   Thus Serre curves are elliptic curves $E/\QQ$ for which the image of $\rho_E$ is as ``large as possible''.   Equivalently, a non-CM elliptic curve $E/\QQ$ is a Serre curve if and only if $G_E=G_{\gamma_d}$ for the unique squarefree integer $d\in (j_E-1728)\cdot (\QQ^\times)^2$.    The first examples of such curves were given by Serre,  see the end of \S5.5 of \cite{Serre-Inv72}.   

Now consider a Serre curve $E/\QQ$.  We have $G_E=G_{\gamma_d}$ for a  unique squarefree integer $d$.    If $d$ is divisible by an odd prime, then $G_{\gamma_d}$ is not agreeable since the level of $G_E$ is divisible by an odd prime and the level of $G_E \cap \SL_2(\Zhat)=G \cap \SL_2(\Zhat)$ is $2$.   So for $d \notin \{\pm 1,\pm 2\}$, the agreeable closure of $G_E$ is $\GL_2(\Zhat)$.   For $d \in \{\pm 1, \pm 2\}$, the group $G_{\gamma_d}$ is agreeable (and so $G_E$ is its own agreeable closure).   However, we cannot have $G_E=G_{\gamma_d}$ for $d\in \{\pm 1\}$,  since in these cases $[G_{\gamma_d},G_{\gamma_d}] \subsetneq G_{\gamma_d} \cap \SL_2(\Zhat)$.

\begin{prop}
For a non-CM elliptic curve $E/\QQ$, let $\calG_E$ be the agreeable closure of $G_E$.   Then $E$ is a Serre curve if and only if $\calG_E$ is equal to $\GL_2(\Zhat)$, $G_{\gamma_2}$ or $G_{\gamma_{-2}}$.
\end{prop}
\begin{proof}
One direction we have already proved.  Now suppose that $\calG_E$ is $\GL_2(\Zhat)$, $G_{\gamma_2}$ or $G_{\gamma_{-2}}$.    In all three cases, we have $[\calG_E,\calG_E] = [\calG,\calG]$.   So $[\GL_2(\Zhat):G_E]=[\SL_2(\Zhat): [\calG,\calG]]=2$.
\end{proof}

\begin{remark}
\begin{romanenum}
\item
With notation as in \S\ref{SS:finding Galois}, there should be a related \'etale cover $\phi\colon Y\to U_\calG$ of degree $|\calG/G|=2$.   We have $U_{\calG}=\AA^1_\QQ -\{0,1728\}$ and $Y=U_G$.   

There is a model $U_G=\{t\in \AA^1_\QQ : t(t^2+1728)\neq 0\}$ with morphism $\phi\colon U_G\to U_\calG$ given by $\phi(t)=t^2+1728$.  For each $j\in U_\calG(\QQ)=\QQ-\{0,1728\}$, the fiber $\phi^{-1}(j)\subseteq U_G(\Qbar)$ consists of two points and the action of $\Gal_\QQ$ on it factors through a faithful action of $\Gal(\QQ(\sqrt{d})/\QQ)$, where $d$ is the unique squarefree integer in $(j-1728)\cdot (\QQ^\times)^2$.
\item
``Most'' elliptic curves over $\QQ$ are Serre curves, cf.~\cite{Jones}.
\item
For number fields $K\neq \QQ$ that contain no nontrivial abelian extension of $\QQ$, one can show that there are elliptic curves $E$ over $K$ with $\rho_{E}^*(\Gal_K)=\GL_2(\Zhat)$, cf.~\cite{MR2721742}.  Note that when $K\neq \QQ$, the maximal abelian extension of $K$ is strictly larger than the cyclotomic extension and hence the constraint of Lemma~\ref{L:KW} need not hold.
\end{romanenum}
\end{remark}

\subsection{An example with the largest known index} \label{SS:largest known index}

Let $E/\QQ$ be the non-CM elliptic curve defined by the Weierstrass equation $y^2 + xy + y = x^3 + x^2 - 8x + 6$; it has $j$-invariant $-7\cdot 11^3$ and conductor $5^2\cdot 7^2$.   This curve $E$ is special since it has an isogeny of degree $37$ defined over $\QQ$.   Without giving all the details, we now explain what goes into computing the group $G_E:=\rho_{E}^*(\Gal_\QQ)$.   In particular, this elliptic curve arises from a rational point on a modular curve of genus $2$, i.e., $X_0(37)$.

For every prime $\ell \neq 37$, we have $\rho_{E,\ell^\infty}(\Gal_\QQ)=\GL_2(\ZZ_\ell)$.  After choosing bases appropriately, the image of $\rho_{E,37}^*$ will lie in the group of upper triangular matrices in $\GL_2(\ZZ/37\ZZ)$.  Therefore, $G_E$ is a subgroup of 
\[
\calG:=\big\{ g\in \GL_2(\Zhat): g\equiv \left(\begin{smallmatrix}* & * \\0 & *\end{smallmatrix}\right) \pmod{37} \big\}.
\]
Making use of our explicit modular curves from Theorem~\ref{T:main agreeable}, we find that $\calG_E=\calG$, i.e., $\calG$ is the agreeable closure of $G_E$.

 The commutator subgroup $[\calG,\calG]$ is the level $2\cdot 37$ subgroup of $\SL_2(\Zhat)$ consisting of matrices whose image modulo $2$ lies in the unique index $2$ subgroup of $\SL_2(\ZZ/2\ZZ)$ and whose image modulo $37$ is of the form $\left(\begin{smallmatrix}1 & * \\0 & 1\end{smallmatrix}\right)$. 
 As noted in \S\ref{SS:agreeable intro}, since $\calG$ is the agreeable closure of $G_E$ we will have $G_E \cap \SL_2(\Zhat)=[\calG,\calG]$ and
\[
[\GL_2(\Zhat): G_E] = [\SL_2(\Zhat): [\calG,\calG]] = 2\cdot |\SL_2(\ZZ/37\ZZ)|/37 = 2736.\\
\]

Let $G$ be the open subgroup of $\GL_2(\Zhat)$ of level $2\cdot 37$ consisting of matrices whose image modulo $2$ lies in the unique index $2$ subgroup of $\GL_2(\ZZ/2\ZZ)$ and whose image modulo $37$ is of the form $\left(\begin{smallmatrix}* & * \\0 & 1\end{smallmatrix}\right)$.   Note that $G$ is an open subgroup of $\calG$ that satisfies $\det(G)=\Zhat^\times$ and $G\cap \SL_2(\Zhat)=[\calG,\calG]$.    

Let $\chi_1\colon \calG \to \{\pm 1\}$ be the homomorphism obtained by composing reduction modulo $2$ with the only nontrivial homomorphism  $\GL_2(\ZZ/2\ZZ) \twoheadrightarrow \{\pm 1\}$.  Let $\chi_2\colon \calG \to (\ZZ/37\ZZ)^\times$ be the homomorphism that takes a matrix $\left(\begin{smallmatrix}a & b \\c & d\end{smallmatrix}\right)$ to $d$ modulo $37$.    The kernel of the homomorphisms $\chi_1$ and $\chi_2$ both contain $G$ and together they induce an isomorphism
\begin{align}\label{E:chi1 and chi2}
\calG/G \xrightarrow{\sim} \{\pm 1\} \times (\ZZ/37\ZZ)^\times.
\end{align}

Let $\alpha_E\colon \Gal_\QQ  \to \calG/G$ be the homomorphism obtained by composing $\rho_E^*\colon \Gal_\QQ \to \calG$ with the quotient map $\calG/G$.  Since $\calG/G$ is abelian, there is a unique homomorphism $\gamma_E \colon \Zhat^\times \to \calG/G$ satisfying $\gamma_E(\chi_\cyc(\sigma)^{-1})=\alpha_E(\sigma)$ for all $\sigma\in \Gal_\QQ$.  Once we have found $\gamma_E$, Lemma~\ref{L:intro HE} implies that $G_E$ is conjugate in $\GL_2(\Zhat)$ to the explicit group $\calH_E:=\{ g \in \calG :  g\cdot G = \gamma_E(\det g)\}$.\\

Let $\gamma_1\colon \Zhat^\times\to \{\pm 1\}$ and $\gamma_2\colon \Zhat^\times \to (\ZZ/37\ZZ)^\times$ be the homomorphism obtained by composing $\gamma_E$ with $\chi_1$ and $\chi_2$, respectively.

We first describe $\gamma_1$.  The extension $\QQ(E[2])/\QQ$ is a Galois extension with group isomorphic to $\GL_2(\ZZ/2\ZZ)$.  So $\QQ(E[2])$ contains a unique quadratic extension; it is $\QQ(\sqrt{\Delta})$, where $\Delta=-5^3 7^2$ is the discriminant of the Weierstrass model of $E$.   Therefore, $\QQ(\sqrt{-5})$ is the unique quadratic extension of $\QQ$ in $\QQ(E[2])$.  Using this, we find  that $\gamma_1 \circ \chi_\cyc \colon \Gal_\QQ \to \{\pm 1\}$ factors through $\Gal(\QQ(\sqrt{-5})/\QQ)\hookrightarrow \{\pm 1\}$.   Therefore, $\gamma_1$ is obtained by composing the reduction modulo $20$ homomorphism $\Zhat^\times \to (\ZZ/20\ZZ)^\times$ with the unique Dirichlet character $(\ZZ/20\ZZ)^\times \to \{\pm 1\}$ of conductor $20$.  

We now describe $\gamma_2$.   There is a unique subgroup $H\subseteq E[37]$ of order $37$ that is stable under the action of $\Gal_\QQ$; the $x$-coordinates of the nonzero elements of $H$ are the roots of a degree $18$ polynomial $f(x)\in \ZZ[x]$ that can be found by factoring the $37$-th division polynomial of $E$.  Let $\beta\colon \Gal_\QQ \to (\ZZ/37\ZZ)^\times$ be the homomorphism for which $\sigma(P)=\beta(\sigma)\cdot P$ for all $P\in H$ and $\sigma\in \Gal_\QQ$.    We have 
\[
\rho_{E,37}^*(\sigma) = \left(\begin{smallmatrix}* & * \\0 & \beta(\sigma)^{-1}\end{smallmatrix}\right)
\]
for $\sigma\in \Gal_\QQ$.   From this, we find that $\gamma_2(\chi_\cyc(\sigma))=\beta(\sigma)$ for $\sigma\in \Gal_\QQ$.   The representation $\rho_{E,37}^*$, and hence also $\beta$, is unramified at all primes $p\nmid 5\cdot 7 \cdot 37$ since the conductor of $E$ is $5^2\cdot 7^2$.   Therefore, $\gamma_2$ factors through a homomorphism $\bbar{\gamma}_2 \colon (\ZZ/ (5\cdot 7\cdot 37) \ZZ)^\times = (\ZZ/ 1295 \ZZ)^\times\to (\ZZ/37\ZZ)^\times$ satisfying $\bbar{\gamma}_2(p)=\beta(\Frob_p)$ for all primes $p\nmid 5\cdot 7\cdot 37$.    We can compute $\beta(\Frob_p)$ for any prime $p\nmid 5\cdot 7\cdot 37$ by working modulo $p$ (for any point $P\in E(\FFbar_p)$ whose $x$-coordinate is a root of $f$, we have $\Frob_p(P)=\beta(\Frob_p)\cdot P$).    In particular, we find that: 
\begin{align} \label{E:beta example intro}
\beta(\Frob_{13})=6, \quad \beta(\Frob_{19})=26, \quad  \beta(\Frob_{29})=36.
\end{align}
Since $13$, $19$ and $29$ generate $(\ZZ/1295 \ZZ)^\times$, the values (\ref{E:beta example intro}) determine $\bbar\gamma_2$ and hence also $\gamma_2$.

The homomorphism $\gamma_E \colon \Zhat^\times \to \calG/G$ is thus the map $a\mapsto (\gamma_1(a),\gamma_2(a))$ composed with the inverse of (\ref{E:chi1 and chi2}).  Now that we know $\calG$ and $\gamma_E$, we can compute $\calH_E$ which gives the image of $\rho_E^*$ up to conjugacy.   A direct calculation shows that $\calH_E$ has  level $4\cdot 5\cdot 7 \cdot 37=5180$ and the image of $\calH_E$ modulo $5180$ is generated by the matrices:
\begin{align} \label{E:exceptional gen 2023}
\left(\begin{smallmatrix} 1 & 38 \\ 0 & 1\end{smallmatrix}\right), \quad \left(\begin{smallmatrix} 1 & 1 \\ 37 & 38\end{smallmatrix}\right), \quad \left(\begin{smallmatrix} 13 & 0 \\ 0 & 2391\end{smallmatrix}\right), \quad
\left(\begin{smallmatrix} 64 & 3737 \\ 37 & 2970\end{smallmatrix}\right), \quad
\left(\begin{smallmatrix} 70 & 851 \\ 37 & 5038\end{smallmatrix}\right), \quad \left(\begin{smallmatrix} 42& 1961 \\ 37 & 4318\end{smallmatrix}\right).
\end{align}
Note that the first two matrices generate the image of $[\calG,\calG]$ modulo $5180$ while the other matrices are chosen to have determinants $3$, $11$, $13$ and $19$, respectively (these primes generate the group $(\ZZ/5180\ZZ)^\times$).

\subsection{An involved example}  

We now give a more complicated example involving the \'etale morphism $\phi\colon Y\to U_\calG$ from \S\ref{SS:finding Galois}; it will arise from our computations.   

Let $\calG$ be the open subgroup of $\GL_2(\Zhat)$ of level $27$ whose image modulo $27$ is generated by the matrices: $\left(\begin{smallmatrix}1 & 1 \\0 & 1\end{smallmatrix}\right)$, $\left(\begin{smallmatrix}1 & 2 \\3 & 2\end{smallmatrix}\right)$, $\left(\begin{smallmatrix}2 & 1 \\9 & 5\end{smallmatrix}\right)$.    We have $\det(\calG)=\Zhat^\times$, $-I\in \calG$, and $[\GL_2(\Zhat):\calG]=36$.  

The curve $X_\calG$ is isomorphic to $\PP^1_\QQ$ and $\calG$ is one of the agreeable groups in our set $\scrA$ from Theorem~\ref{T:main agreeable}.   Moreover, $\QQ(X_\calG)=\QQ(t)$ for some $t$ so that the map $\pi_\calG$ to the $j$-line is given by the rational function
\[
\pi(t):= \frac{(t^3 + 3)^3  (t^9 + 9t^6 + 27t^3 + 3)^3 }{ t^3(t^6 + 9t^3 + 27)}.
\]
We have $\pi(t)-1728=(t^{18} + 18t^{15} + 135t^{12} + 504t^9 + 891t^6 + 486t^3 - 27)^2/ (t^3(t^6 + 9t^3 + 27))$.  So
\[
U_\calG=\Spec \QQ[t,1/f] \subseteq \Spec \QQ[t]=\AA_\QQ^1,
\]
where 
$f:=t (t^3+3)(t^9 + 9t^6 + 27t^3 + 3)(t^{18} + 18t^{15} + 135t^{12} + 504t^9 + 891t^6 + 486t^3 - 27)$.

Let $G$ be the open subgroup of $\GL_2(\Zhat)$ of level $54$ whose image modulo $54$ is generated by the matrices: $\left(\begin{smallmatrix}7 & 0 \\ 36 & 1\end{smallmatrix}\right)$, $\left(\begin{smallmatrix}7 & 16 \\ 0 & 25\end{smallmatrix}\right)$, $\left(\begin{smallmatrix}16 & 7 \\ 3 & 5\end{smallmatrix}\right)$.   We have $\det(G)=\Zhat^\times$, $-I\notin G$, $G\subseteq \calG$, and $[\calG:G]=36$.   The group $G$ is normal in $\calG$ and $G\cap \SL_2(\Zhat) = [\calG,\calG]$.   The quotient group $\calG/G$ is abelian of order $36$.\\

As mentioned in \S\ref{SS:finding Galois}, in \S\ref{S:finding image given agreeable closure} we describe a particular \'etale cover $\phi\colon Y\to U_\calG$ that is Galois with group $\calG/G$; it is used for computing groups $\rho^*_E(\Gal_\QQ)$ whose agreeable closure is conjugate in $\GL_2(\Zhat)$ to $\calG$.   We now state $\phi$ with respect to the explicit models that occur in our computations.

For each $1\leq i \leq 9$, define a homogeneous polynomial $F_i \in \ZZ[x_1,\ldots, x_8]$ and a polynomial $c_i \in \ZZ[t]$ as follows:
{
\small       
\begin{align*}        
& F_1:=         
x_1^2 + x_1 x_4 + x_2^2 + x_2 x_5 + x_3^2 + x_3 x_6 + x_4^2 + x_5^2 + x_6^2,
\\
&F_2:=
x_1 x_3 - x_1 x_5 + x_2 x_4 - x_2 x_6 + x_3 x_4 + x_3 x_5 + x_4 x_6,
\\
&F_3:=
x_1 x_2 - x_1 x_6 + x_2 x_3 + x_2 x_4 + x_3 x_4 + x_3 x_5 + x_4 x_5 + x_5 x_6,
\\
&F_4:=x_1^2 x_4 + x_1 x_4^2 + x_2^2 x_5 + x_2 x_5^2 + x_3^2 x_6 + x_3 x_6^2,
\\
&F_5:=x_1^2 x_3 + x_1^2 x_6 - x_1 x_2^2 + 2 x_1 x_3 x_4 + x_1 x_5^2 - x_2 x_3^2 + 2 x_2 x_4 x_5 + x_2 x_6^2 + 2 x_3 x_5 x_6 - x_4^2 x_6 + x_4 x_5^2 + x_5 x_6^2,
\\
&F_6:=x_1^3 + 3 x_1^2 x_4 + x_2^3 + 3 x_2^2 x_5 + x_3^3 + 3 x_3^2 x_6 - x_4^3 - x_5^3 - x_6^3,
\\
&F_7:=x_1^2 x_2 +	 x_1^2 x_5 + 2 x_1 x_2 x_4 - x_1 x_3^2 + x_1 x_6^2 + x_2^2 x_3 + x_2^2 x_6 + 2 x_2 x_3 x_5+ 2 x_3 x_4 x_6 - x_4^2 x_5 
+ x_4 x_6^2 - x_5^2 x_6,\\
&F_8:= x_7^2,\\
&F_9:= x_8^2,
\end{align*} 
\begin{align*} 
& c_1:= 2(t^6 + 9t^3 + 27), \\
&c_2:=-(t^6 + 9t^3 + 27), \\
& c_3:= -(t^6 + 9t^3 + 27),\\
& c_4:= -(2t^2 + 2t - 3)(t^6 + 9t^3 + 27), \\
&c_5:=3(t - 1)(t + 2)(t^6 + 9t^3 + 27), \\
&c_6:=-(2t - 3)(t^2 + 3t + 3)(t^6 + 9t^3 + 27),\\
&c_7:=(3t^2 + 4t - 3)(t^6 + 9t^3 + 27), \\
&c_8:=t(t^6 + 9t^3 + 27), \\
& c_9:=-3t(t^3 + 3)(t^6 + 9t^3 + 27)(t^9 + 9t^6 + 27t^3 + 3)(t^{18} + 18t^{15} + 135t^{12} + 504t^9 + 891t^6 + 486t^3 - 27).
\end{align*}

}       

 Let $Y$ be the closed subvariety of $\Spec \QQ[x_1,\ldots, x_8,t,1/f]$ defined by the equations 
 \[
 F_i(x_1,\ldots, x_8) = c_i(t) 
 \]
 with $1\leq i \leq 9$.   Let $\phi\colon Y \to U_\calG$ be the morphism given by $(x_1,\ldots, x_8,t)\mapsto t$.

We now describe an action of $\calG/G$ on $Y$.  Choose matrices $g_1$, $g_2$ and $g_3$ in $\calG$ that are congruent modulo $54$ to 
$\left(\begin{smallmatrix}31 & 44 \\ 36 & 25\end{smallmatrix}\right)$,  $\left(\begin{smallmatrix}28 & 27 \\ 27 & 28\end{smallmatrix}\right)$ and $\left(\begin{smallmatrix}-1 & 0 \\ 0 & -1\end{smallmatrix}\right)$, respectively. We have a unique isomorphism of abelian groups 
\[
\iota\colon \ZZ/9\ZZ \times \ZZ/2\ZZ \times \ZZ/2\ZZ \xrightarrow{\sim} \calG/G
\] 
for which $(1,0,0)\mapsto g_1G$, $(0,1,0)\mapsto g_2 G$ and $(0,0,1)\mapsto g_3 G$.   So it suffices to describe the action of each $g_i G$ on $Y$.  For any point $y=(a_1,\ldots, a_9) \in Y(\Qbar)$, we have 
\begin{align*}
g_1 G \cdot y &:= (a_2,a_3,a_4,a_5,a_6, -a_1-a_4,a_7,a_8, a_9),\\
g_2 G \cdot y &:= (a_1,a_2,a_3,a_4,a_5,a_6,-a_7,a_8, a_9),\\
g_2 G \cdot y &:= (a_1,a_2,a_3,a_4,a_5,a_6,a_7,-a_8, a_9).
\end{align*}
The action of $\calG/G$ on $Y$ is faithful and does not affect the morphism $\phi$.  In fact, $\phi\colon Y\to U_\calG$ is an \'etale morphism of degree $36$ that is Galois with the action of $\calG/G$ giving its Galois group.  The curve $Y$ is defined over $\QQ$ and is smooth and geometrically irreducible. This completes our explicit description of $\phi\colon Y\to U_\calG$.\\        
        
As an example of how to use these equations, consider the elliptic curve $E/\QQ$ give by the Weierstrass equation 
\[
y^2+y = x^3+x^2+x.
\]   
The curve $E$ has $j$-invariant $32768/19$ and conductor $19$.   We have $j_E = \pi(-1)\in \pi_\calG(X_\calG(\QQ))$, so $G_E:=\rho^*_E(\Gal_\QQ)$ is conjugate in $\GL_2(\Zhat)$ to a subgroup of $\calG$.  So after replacing $\rho^*_E$ by an isomorphic representation, we may assume that $G_E \subseteq \calG$.  In particular, the agreeable closure $\calG_E$ of $G_E$ is a subgroup of $\calG$.   Using our groups and modular curves from Theorem~\ref{T:main agreeable}, we find that $\calG_E=\calG$.

Fix $u:=-1 \in U_\calG(\QQ)=\QQ-\{0\}$.   The fiber $\phi^{-1}(u)$ is the subscheme of $\AA^8_\QQ = \Spec \QQ[x_1,\ldots, x_8]$ defined by the equations $F_i(x_1,\ldots,x_8)=c_i(-1)$ with $1\leq i \leq 9$; it is reduced of dimension $0$, has degree $36$, and $\calG/G$ acts faithfully on the $\Qbar$-points.    Let $\alpha_u\colon \Gal_\QQ \to \calG/G$ be the homomorphism such that $\sigma(y)=\alpha_u(\sigma) \cdot y$ for any $\sigma\in \Gal_\QQ$ and any $y\in Y(\Qbar)$ with $\phi(y)=u$. 

Take any prime $p\nmid 6$ for which the $\ZZ_p$-subscheme $Z \subseteq \AA^8_{\ZZ_p} = \Spec \ZZ_p[x_1,\ldots, x_8]$ defined by the equations $F_i(x_1,\ldots,x_8)=c_i(-1)$ with $1\leq i \leq 9$, is smooth and $Z_{\FF_p}$ has degree $36$.  The action of $\calG/G$ on $Z(\FFbar_p)$ is simply transitive.   Using Hensel's lemma and the smoothness of $Z$, we find that $\Frob_p(y)= \alpha_u(\Frob_p)\cdot y$ for all $y\in Z(\FFbar_p)$.  This gives our computable description of $\alpha_u(\Frob_p)$; find a point $y\in Z(\FFbar_p)$, raise its coordinates to the $p$-th power and find the unique $\alpha_u(\Frob_p) \in \calG/G$ for which $\Frob_p(y)= \alpha_u(\Frob_p)\cdot y$.   In this way, one can show that:
\begin{align} \label{E:alphau intro hard}
\alpha_u(\Frob_5) =\iota((2,0,1)), \quad \alpha_u(\Frob_{11}) =\iota((6,0,1)),\quad \alpha_u(\Frob_{13}) =\iota((4,1,1)). 
\end{align}

 With notation as in \S\ref{SS:finding image setup}, we define $\alpha_E := \chi\cdot \alpha_u \colon \Gal_\QQ \to \calG/G$, where $\chi\colon \Gal_\QQ \to \{\pm 1\}$ is the homomorphism that factors through $\Gal(\QQ(\sqrt{-19})/\QQ) \hookrightarrow \{\pm 1\}$.     By Lemma~\ref{L:alphaE properties}, after replacing $\rho_E^*$ by an isomorphic representation, we may assume that $G_E \subseteq \calG$ and that the composition of $\rho_E^*$ with the quotient map $\calG\to \calG/G$ is $\alpha_E$.  
 
 Let 
\[
\gamma_E \colon \Zhat^\times \to \calG/G
\]
be the homomorphism such that $\gamma_E(\chi_\cyc(\sigma)^{-1})=\alpha_E(\sigma)$ for all $\sigma\in \Gal_\QQ$.    With $M=2\cdot 3 \cdot 19$ and $e=18$,   we can argue as in \S\ref{SS:finding Galois} that $\gamma_E$ factors through a homomorphism $\bbar \gamma_E\colon \ZZ_M^\times/(\ZZ_M^\times)^e \to \calG/G$ such that $p \cdot(\ZZ_M^\times)^e \mapsto\alpha_E(\Frob_p)^{-1}$ for all  primes $p\nmid M$.   In particular,
\begin{align} \label{E:alphau intro hard 2}
\bbar\gamma_E(5) =\iota((7,0,1)), \quad \bbar\gamma_E({11}) =\iota((3,0,1)),\quad \bbar\gamma_E({13}) =-\iota((5,1,1))=\iota((5,1,0)). 
\end{align}
Since the primes $5$, $11$ and $13$ generate $\ZZ_M^\times/(\ZZ_M^\times)^e$, we deduce that  $\gamma_E$ is determined by $M$ and the values (\ref{E:alphau intro hard 2}).    Using this, we can show that $\gamma_E$ is the composition of the reduction homomorphism $\Zhat^\times \to (\ZZ/57\ZZ)^\times$ with the unique homomorphism $(\ZZ/57\ZZ)^\times \to \calG/G$ for which $5 \mapsto \iota((7,0,1))$ and $13\mapsto \iota((5,1,0))$.   

Now that we know $\calG$ and $\gamma_E$, we can compute the group $\calH_E$ from (\ref{E:calHE intro def}).  The group $\calH_E$ is an open subgroup of $\GL_2(\Zhat)$ with level $2\cdot 27\cdot 19=1026$ and its image modulo $1026$ is generated by the matrices:
\[
\left(\begin{smallmatrix}31 & 198 \\10 & 97\end{smallmatrix}\right), \quad \left(\begin{smallmatrix}1 & 0 \\18 & 1\end{smallmatrix}\right), \quad \left(\begin{smallmatrix}28 & 729 \\27 & 703\end{smallmatrix}\right), \quad \left(\begin{smallmatrix}149 & 681 \\271 & 448\end{smallmatrix}\right), \quad \left(\begin{smallmatrix}994 & 9 \\689 & 790\end{smallmatrix}\right);
\]
the first three matrices also generate the image of $\calH_E\cap \SL_2(\Zhat)$ modulo $1026$.   By Lemma~\ref{L:intro HE}, this gives the image of $\rho_E^*$ up to conjugacy.

\subsection{Some related results}

There has been much research on modular curves and the image of Galois representations associated to non-CM elliptic curves over $\QQ$.   We now give a brief and incomplete description of some related recent progress.  \\

Given a fixed non-CM elliptic curve $E/\QQ$,  we can determine the (finite) set of primes $\ell$ for which $\rho_{E,\ell}(\Gal_\QQ)\neq \GL_2(\ZZ/\ell\ZZ)$ using the algorithm in \cite{surjectivityalgorithm} (it is based on Serre's original proof in \cite{Serre-Inv72}).  For each prime $\ell$ for which $\rho_{E,\ell}$ is not surjective, we can also compute the subgroup $\rho_{E,\ell}(\Gal_\QQ) \subseteq \GL_2(\ZZ/\ell\ZZ)$ up to conjugacy using \cite{possible images} or \cite{MR3482279} (the first reference uses explicit modular curves while the second reference uses Frobenius matrices to give a probabilistic algorithm).

Consider primes $\ell>13$.  There has been much progress towards Conjecture~\ref{C:Serre question}.
If $\rho_{E,\ell}$ is not surjective for a non-CM elliptic curve $E/\QQ$, then $E$ gives rise to a rational non-CM point on a modular curve $X_G$ with $G$ a maximal subgroup of $\GL_2(\ZZ/\ell\ZZ)$ satisfying $\det(G)=(\ZZ/\ell\ZZ)^\times$.    When $G$ is the subgroup of upper triangular matrices, Mazur \cite{MR482230} has found the rational points of $X_0(\ell):=X_G$.   The curve $X_0(\ell)$ has no non-CM rational points for $\ell>17$ and $\ell\neq 37$ (for $\ell\in \{17,37\}$, there are non-CM rational points which lead to the $j$-invariants in the statement of Conjecture~\ref{C:Serre question}).    When $G$ is the normalizer of a split Cartan subgroup, Bilu, Parent and Rebolledo \cite{MR3137477} have shown that $X_G$ has no non-CM points.  When the image of $G$ in $\GL_2(\ZZ/\ell\ZZ)/((\ZZ/\ell\ZZ)^\times \cdot I)$ is isomorphic to $\mathfrak{S}_4$, the modular curve $X_G(\QQ)$ has no non-CM points, see the remarks  on page 36 of \cite{MR488287}.   The remaining modular curves to consider are the curves $X_{\operatorname{ns}}^+(\ell):=X_G$, where $G$ is the normalizer of a non-split Cartan subgroup of $\GL_2(\ZZ/\ell\ZZ)$.   There has been recent progress on finding the rational points on $X_{\operatorname{ns}}^+(\ell)$ for small $\ell$ using generalized versions of Chabauty's method, cf.~\cite{MR3961086,chabautyalgo}, but the general case remains open.\\
   
Now consider the images of the $\ell$-adic representations $\rho_{E,\ell^\infty}$.    If $\rho_{E,\ell}(\Gal_\QQ)=\GL_2(\ZZ/\ell\ZZ)$ for a non-CM elliptic $E/\QQ$ and a prime $\ell \geq 5$, then we have $\rho_{E,\ell^\infty}(\Gal_\QQ)=\GL_2(\ZZ_\ell)$, cf.~\S\ref{L:modell to elladic}.  Taking Conjecture~\ref{C:Serre question} and our above discussion of modulo $\ell$ representations into account, it makes sense to focus on $\ell$-adic projections for the primes $\ell\in \{2,3,5,7,11,13,17,37\}$.  For the prime $\ell=2$, Rouse and Zureick-Brown \cite{MR3500996} gave a complete description of the images $\rho_{E,\ell^\infty}(\Gal_\QQ) \subseteq \GL_2(\ZZ_\ell)$, up to conjugacy, for all non-CM elliptic curves $E/\QQ$ (they found models of all relevant modular curves and computed their rational points).  For each prime $\ell$, Sutherland and Zywina \cite{MR3671434} 
described all open subgroups $G$ of $\GL_2(\ZZ_\ell)$ with $\det(G)=\ZZ_\ell^\times$ for which $X_G(\QQ)$ is infinite and then computed a model for $X_G$ along with the morphism to the $j$-line.
In \cite{RSZ}, Rouse, Sutherland and Zureick-Brown gave a complete description of $\ell$-adic images up to possible counterexamples to Conjecture~\ref{C:Serre question} and determining the rational points on a finite number of explicit modular curves $X_G$ of genus at least $2$.\\

Once one gets a handle on the $\ell$-adic Galois images, it is natural to consider the image modulo integers divisible by several distinct primes.   There has been much recent work on understanding and classifying ``entanglements''; for example, see \cite{arXiv:2105.02060,arxiv2008.09886,MR4374148,MR3957898,MR3447646}.   For relative prime positive integers $m$ and $n$, we say that $E$ has a $(m,n)$-\defi{entanglement} if $\QQ(E[m]) \cap \QQ(E[n]) \neq \QQ$; equivalently, $\rho_{E,mn}(\Gal_\QQ)$ can be viewed as a \emph{proper} subgroup of $\rho_{E,m}(\Gal_\QQ)\times \rho_{E,n}(\Gal_\QQ)$.  In particular, entanglements describe constraints on the image $\rho_E(\Gal_\QQ)$.

While the work in this paper does give some information, we have avoided a general study of possible entanglements.    What may seem surprising at first, is that to compute the group $\rho_E(\Gal_\QQ)$ we do not first compute the $\ell$-adic projections $\rho_{E,\ell^\infty}(\Gal_\QQ)$; even though they can be found using \cite{RSZ}.  The approach of computing the $\ell$-adic images and then describing all the possible entanglements seems to lead to an excessive number of cases.  Of course once we have found $\rho_E(\Gal_\QQ)$, up to conjugacy, we can then easily compute the $\ell$-adic projections $\rho_{E,\ell^\infty}(\Gal_\QQ)$.\\

There has also been some more general study on the image of $\rho_E$.  Jones has produced upper bounds for the level of $\rho_E(\Gal_\QQ)$ in $\GL_2(\Zhat)$ for non-CM elliptic curves $E/\QQ$, cf.~\cite{MR4190460,MR2439422}.   The paper \cite{MR3350106} of Jones contains a lot of group theoretic information about the image of $\rho_E$; in particular, he generalizes the notion of a Serre curve, cf.~Remark~\ref{R:generalized Serre curves}.   The doctoral thesis of Brau Avilo \cite{BrauThesis} appears to be the first place to explicitly point out that there is an algorithm to compute $\rho_E(\Gal_\QQ)$.  His algorithm first finds the level $m$ of $\rho_E(\Gal_\QQ)$ and then computes $\rho_{E,m}(\Gal_\QQ)$ by making use of division polynomials; it it not practical in general.


Define the set of integers
\[
\calI = \left\{\begin{array}{c}2, 4, 6, 8, 10, 12, 16, 20, 24, 30, 32, 36, 40, 48, 54, 60, 72,\\ 84, 96, 108, 112,120, 144, 192, {220}, {240},  288, 336, {360}, \\ 384, {504}, 576, 768, 864, 1152, 1200, 1296, 1536 \end{array}\right\}.
\]
In \cite{possibleindices}, it is shown that there is a finite set $J\subseteq \QQ$ such that for any elliptic curve $E/\QQ$ with $j_E \notin J$ and $\rho_{E,\ell}(\Gal_\QQ)=\GL_2(\ZZ/\ell\ZZ)$ for all primes $\ell>37$, we have $[\GL_2(\Zhat): \rho_E(\Gal_\QQ)] \in \calI$.  Moreover, $\calI$ is the smallest set with this property.   The results of this paper can be used to give an elaborate alternate proof of this result (in \cite{possibleindices}, modular curves are used but no models are computed).   Note that the only new integers that arise in Conjecture~\ref{C:brave} are: $80$, $128$, $160$, $182$, $200$, $216$, $224$, $300$, $480$, $2736$.

Rakvi \cite{Rakvi} has recently given a description of the pairs $(X_G,\pi_G)$, up to a suitable notion of isomorphism, as we vary over all open subgroups $G$ of $\GL_2(\Zhat)$ satisfying $\det(G)=\Zhat^\times$, $-I \in G$, and $X_G\cong \PP^1_\QQ$ (see the end of \S\ref{S:families} for details).\\

The above results, and similar ones, are often phrased in the context of progress towards the following overarching program:

\begin{programb} \cite{MR0450283}
Given a number field $K$ and a subgroup $H$ of $\GL_2(\Zhat) =\prod_{p} \GL_2(\ZZ_p)$ classify all elliptic curves $E/K$ whose associated Galois representation on torsion points map $\Gal(\Kbar/K)$ into $H\subseteq \GL_2( \Zhat)$.
\end{programb}

\subsection{Overview}
We briefly outline the structure of the paper.  In \S\ref{S:Galois first}, we recall the connection between $\rho_E^*$ with the cyclotomic character and explain how the image of $\rho_E^*$ changes when we replace $E$ by a quadratic twist.  

Consider a subgroup $G\subseteq \GL_2(\ZZ/N\ZZ)$ satisfying $\det(G)=(\ZZ/N\ZZ)^\times$ and $-I\in G$.  In \S\ref{S:first modular curve}, we give a quick definition of modular curves $X_G$ in terms of their function fields (which will be fields consisting of modular functions).   The curve comes with a non-constant morphism $\pi_G\colon X_G \to \PP^1_\QQ=\AA^1_\QQ\cup \{\infty\}$ from $X_G$ to the $j$-line.

In \S\ref{S:modular forms}, we will define a finite dimensional $\QQ$-vector space $M_{k,G}$ consisting of modular forms for each integer $k\geq 2$.  There will be a natural isomorphism $M_{k,G} \otimes_\QQ \CC \xrightarrow{\sim} M_k(\Gamma_G)$, where $\Gamma_G$ is the congruence subgroup of $\SL_2(\ZZ)$ consisting of matrices whose image modulo $N$ lies in $G$ and $M_k(\Gamma_G)$ is the usual space of weight $k$ modular forms on $\Gamma_G$.   Much of \S\ref{S:modular forms} is dedicated to describing how to compute an explicit basis of $M_{k,G}$; our approach makes use of Eisenstein series and a theorem of Khuri-Makdisi.   Our modular forms will be expressed in terms of their $q$-expansion at \emph{every} cusp (and for which we can compute arbitrarily many terms of each $q$-expansion).    

In \S\ref{S:computing modular forms}, we explain how to compute a model for the modular curve $X_G$ and in some cases compute the morphism $\pi_G$.  The key observation is that our space of modular forms $M_{k,G}$ is the global sections of a line bundle on the modular curve $X_G$ for even $k$.  For $k$ even and sufficiently large, this line bundle will be very ample and a basis for $M_{k,G}$ will allow us to compute an explicit model of $X_G$ in some projective space $\PP_{\QQ}^n$.

Let $U_G$ be the open subvariety $X_G-\pi_G^{-1}(\{0,1728,\infty\})$ of $X_G$.  In \S\ref{S:specializations}, we use modular functions to construct an explicit representation $\varrho\colon \pi_1(U_G)\to G$ of the \'etale fundamental group of $U_G$.  For each rational point $u\in U_G(\QQ)$, the specialization of $\varrho$ at $u$ will be a Galois representation $\Gal_\QQ \to G \subseteq \GL_2(\ZZ/N\ZZ)$ that is isomorphic to $\rho_{E,N}^*\colon \Gal_\QQ \to \GL_2(\ZZ/N\ZZ)$ for a certain elliptic curve $E/\QQ$ with $j$-invariant $\pi_G(u)$.

After recalling group theoretic results concerning subgroups and quotients of $\SL_2(\Zhat)$ and $\GL_2(\Zhat)$ in \S\ref{S:group theory facts}, we will study agreeable groups in \S\ref{S:agreeable}.   In particular, in \S\ref{S:agreeable} we will prove the existence of  agreeable closures and we will also explain how to find the maximal agreeable subgroups of an agreeable group.

In \S\ref{S:constructing agreeable}, we prove Theorem~\ref{T:main agreeable}.  In \S\ref{S:finding the agreeable closure}, we explain how to find the agreeable closure of $G_E:=\rho_E^*(\Gal_\QQ)$, up to conjugacy, for a non-CM elliptic curve $E/\QQ$.   In \S\ref{S:putting it together}, we finally explain how to compute $G_E$, up to conjugacy, after understanding how to construct a certain homomorphism $\gamma_E$ in \S\ref{S:finding image given agreeable closure}.

In \S\ref{S:universal elliptic curves}, we give some insight into computing universal elliptic curves; this will follow quickly from  earlier sections but we state it separately for easy reference.   Finally in \S\ref{S:families}, we make some remarks concerning  ``families'' of groups.

\subsection{Implementation}

As already noted, our algorithms have been implemented in \texttt{Magma} \cite{Magma} and code can be found in the repository \cite{github}:
\begin{center}
\url{https://github.com/davidzywina/OpenImage}
\end{center}
This also includes files containing all the relevant groups and modular curves (unfortunately, the number of cases makes it infeasible to express in a reasonable length table).

This motivation to have a practical algorithm underlies much of the exposition and structure of this paper.   At the onset of the project, it was unclear if the approach presented here was going to be computationally feasible; for example, some modular curves took hours to find models for, using known approaches, and we had thousands of curves to study.  The precomputation required for our algorithms, which are not especially optimized, took less than half a day.

\subsection{Notation} \label{SS:notation}

We now set some notation that will hold throughout.   All profinite groups will be viewed as topological groups with their profinite topology.  In particular, finite groups will have the discrete topology.   For a topological group $G$, we define its \defi{commutator subgroup} $[G,G]$ to be the smallest closed normal subgroup of $G$ for which $G/[G,G]$ is abelian.  Equivalently, $[G,G]$ is the topological subgroup of $G$ generated by the set of commutators $\{ghg^{-1}h^{-1}: g,h \in G\}$.

For each integer $N>1$, we let $\ZZ_N$ be the ring obtained by taking the inverse limit of the $\ZZ/N^e\ZZ$ with $e\geq 1$. Let $\Zhat$ be the ring obtained taking the inverse limit of $\ZZ/n\ZZ$ over all positive integers $n$.  With the profinite topology, $\ZZ_N$ and $\Zhat$ are compact topological rings.  We have natural isomorphisms 
\[
\ZZ_N = \prod_{\ell|N} \ZZ_\ell\quad  \text{and} \quad \Zhat= \ZZ_N \times \prod_{\ell \nmid N} \ZZ_\ell =\prod_{\ell}\ZZ_\ell,
\]   
where the product is over primes $\ell$.   The symbol $\ell$ will always denote a rational prime.

The \defi{level} of an open subgroup $G$ of $\GL_2(\Zhat)$ is the smallest positive integer $n$ for which $G$ contains the kernel of the reduction modulo $n$ homomorphism $\GL_2(\Zhat)\to \GL_2(\ZZ/n\ZZ)$.  The \defi{level} of an open subgroup $G$ of $\GL_2(\ZZ_N)$ is the smallest positive integer $n$ that divides some power of $N$ and for which $G$ contains the kernel of the reduction modulo $n$ homomorphism $\GL_2(\ZZ_N)\to \GL_2(\ZZ/n\ZZ)$. Similarly, we can define the \defi{level} of open subgroups of $\SL_2(\Zhat)$ and $\SL_2(\ZZ_N)$.

\subsection{Acknowledgements}
Thanks to Eray Karabiyik for his comments and corrections.

\section{Cyclotomic constraints on the image of Galois} \label{S:Galois first}

With a fixed non-CM elliptic curve $E$ defined over $\QQ$, we consider the group  $G_E:=\rho_E^*(\Gal_\QQ)$ which from Serre we know is an open  subgroup of $\GL_2(\Zhat)$.   

\subsection{Kronecker--Weber constraint}

Let $\chi_\cyc\colon \Gal_\QQ \to \Zhat^\times$ be the \defi{cyclotomic character}, i.e., the continuous homomorphism such that for every integer $n\geq 1$ and every $n$-th root of unity $\zeta\in \Qbar$ we have $\sigma(\zeta)=\zeta^{\chi_\cyc(\sigma)\bmod{n}}$ for all $\sigma\in \Gal_\QQ$.   By considering the Weil pairing on the groups $E[n]$, we know that $ \det\circ \rho_{E}=\chi_\cyc$ and hence $\det\circ \rho_E^* = \chi_\cyc^{-1}$. In particular, we have
\[
\det(G_E)=\chi_\cyc(\Gal_\QQ)=\Zhat^\times.
\]
By Lemma~\ref{L:KW}, which makes use of the Kronecker--Weber theorem, we have
\begin{align} \label{E:KW do over}
G_E \cap \SL_2(\Zhat) = [G_E,G_E].
\end{align}
and $[\GL_2(\Zhat): G_E]=[\SL_2(\Zhat): [G_E,G_E]]$. One consequence of (\ref{E:KW do over}) is that it is possible to compute the index of $G_E$ in $\GL_2(\Zhat)$ using a group that is possibly larger than $G_E$.

\begin{lemma} \label{L:abelian quotients and same commutators}
Suppose $\calG \subseteq \GL_2(\Zhat)$ is a group such that $G_E$ is a normal subgroup of $\calG$ and $\calG/G_E$ is abelian.   Then $G_E$ and $\calG$ have the same commutator subgroup.  In particular, we have $G_E\cap\SL_2(\Zhat)=[\calG,\calG]$ and
\[
[\GL_2(\Zhat): G_E]=[\SL_2(\Zhat): [\calG,\calG]].
\]
\end{lemma}
\begin{proof}
We have $[\calG,\calG]\subseteq G_E$ since $\calG/G_E$ is abelian.  Therefore, $[\calG,\calG] \subseteq G_E \cap \SL_2(\Zhat)=[G_E,G_E]$, where the last equality uses Lemma~\ref{L:KW}.  The opposite inclusion $[\calG,\calG]\supseteq [G_E,G_E]$ is clear since $\calG\supseteq G_E$.  Therefore, $[\calG,\calG]=[G_E,G_E]$.    The final statement about $G_E \cap \SL_2(\Zhat)$ and the index follows from Lemma~\ref{L:KW}.
\end{proof}

\subsection{Quadratic twists} \label{SS:quadratic twist}

Fix a squarefree integer $d$ and let $E'/\QQ$ be the quadratic twist of $E$ by $d$.   In this section, we describe how the images of $\rho_E^*$ and $\rho_{E'}^*$ are related.     

Let $\chi_d\colon \Gal_\QQ \to \{\pm 1\}$ be the homomorphism that factors through $\Gal(\QQ(\sqrt{d})/\QQ) \hookrightarrow \{\pm 1\}$.  There is a unique homomorphism $\psi\colon \Zhat^\times \to \{\pm 1\}$ such that $\chi_d = \psi \circ \chi_{\cyc}^{-1}$.   Define
\[
\calH:= \{  \psi(\det g) \cdot g : g \in G_E\};
\]
it is an open subgroup of $\GL_2(\Zhat)$.

\begin{lemma}  \label{L:quadratic twist group}
The groups $\rho_{E'}^*(\Gal_\QQ)$ and  $\calH$ are conjugate in $\GL_2(\Zhat)$.
\end{lemma}
\begin{proof}
After replacing $\rho_{E'}^*$ with an appropriate isomorphic representation, we may assume that $\rho_{E'}^* = \chi_d \cdot \rho_E^*$.  For any $\sigma\in \Gal_\QQ$, we have
\[
\rho_{E'}^*(\sigma) = \chi_d(\sigma) \cdot \rho_E^*(\sigma) = \psi(\chi_\cyc(\sigma)^{-1}) \cdot \rho_E^*(\sigma) = \psi(\det(\rho_{E}^*(\sigma)) \cdot \rho_E^*(\sigma),
\]
where we have used that $\det\circ \rho_{E}^*=\chi_\cyc^{-1}$.   Therefore, $\rho_{E'}^*(\Gal_\QQ)=\{  \psi(\det g) \cdot g : g \in \rho_E^*(\Gal_\QQ) \}.$
\end{proof}

Now suppose that we know the group $G_E$.  More specifically,, we have an integer $N\geq 1$ divisible by the level of $G_E$ and a set of generators of $G_E$ modulo $N$.   The homomorphism $\psi$ is easy to find; it factors through a Dirichlet character $ (\ZZ/D\ZZ)^\times \to \{\pm 1\}$, where $D$ is the discriminant of $\QQ(\sqrt{d})$.   Let $N'$ be the least common multiple of $N$ and $D$.  Then the level of $\calH$ divides $N'$ and we can find generators for the image of $\calH$ modulo $N'$.   By Lemma~\ref{L:quadratic twist group}, we have thus computed the image of $\rho_{E'}^*$ up to conjugacy in $\GL_2(\Zhat)$.

In particular, once we know the image of $\rho_E^*$, we can easily obtain the image for any quadratic twist of $E$ (equivalently, any elliptic curve over $\QQ$ with the same $j$-invariant).    In practice, when computing the image of $\rho_E^*$, we will first replace $E$ by a quadratic twist that has a minimal set of primes of bad reduction.

\section{The modular curve $X_G$} \label{S:first modular curve}

The goal of this section is to give a quick definition of the modular curve $X_G$, where $G$ is either an open subgroup of $\GL_2(\Zhat)$ with $\det(G)=\Zhat^\times$ or a subgroup of $\GL_2(\ZZ/N\ZZ)$ with $\det(G)=(\ZZ/N\ZZ)^\times$.  While we could define $X_G$ as a coarse moduli space, we will instead define it by explicitly giving its function field.  Let $\zeta_N$ be the primitive $N$-th root of unity $e^{2\pi i /N}$ in $\CC$.

\subsection{Modular functions} \label{SS:modular functions}

The group $\SL_2(\ZZ)$ acts by linear fractional transformations on the complex upper half-plane $\calH$ and the extended upper half-plane $\calH^*=\calH\cup \QQ \cup \{\infty\}$.  

Let $\Gamma$ be a congruence subgroup of $\SL_2(\ZZ)$.  The quotient $\calX_\Gamma:=\Gamma\backslash \calH^*$ is a smooth compact Riemann surface (away from the cusps and elliptic points, use the analytic structure coming from $\calH$ and extend to the full quotient).   Denote by $\CC(\calX_{\Gamma})$ the field of meromorphic functions on $\calX_{\Gamma}$.\\

Fix a positive integer $N$.  Every $f\in \CC(\calX_{\Gamma(N)})$ gives rise to a meromorphic function on $\calH$ that satisfies
\[
f(\tau) = \sum_{n\in \ZZ} c_n(f) q_N^{n}
\]
for $\tau\in \calH$, where $q_N:=e^{2\pi i \tau/N}$ and the $c_n(f)$ are unique complex numbers that are nonzero for only finitely many  $n <0$.   This Laurent series in $q_N$ is called the \defi{$q$-expansion} of $f$ (at the cusp $\infty$).

Let $\calF_N$ be the subfield of $\CC(\calX_{\Gamma(N)})$ consisting of all meromorphic functions $f$ such that $c_n(f)$ lies in $\QQ(\zeta_N)$ for all $n\in \ZZ$.  For example, $\calF_1=\QQ(j)$, where $j$ is the modular $j$-invariant.

\begin{lemma} \label{L:Shimura basic}
There is a unique right action $*$ of $\GL_2(\ZZ/N\ZZ)$ on the field $\calF_N$ such that the following hold for all $f\in \calF_N$:
\begin{itemize}
\item
For $A\in \SL_2(\ZZ/N\ZZ)$, we have $(f*A)(\tau) = f(\gamma\tau)$, where $\gamma\in \SL_2(\ZZ)$ is any matrix congruent to $A$ modulo $N$.
\item
For $A=\left(\begin{smallmatrix}1 & 0 \\0 & d\end{smallmatrix}\right) \in \GL_2(\ZZ/N\ZZ)$, the $q$-expansion of $f*A$ is $\sum_{n\in \ZZ} \sigma_d(c_n(f)) q_N^{n}$, where $\sigma_d$ is the automorphism of the field $\QQ(\zeta_N)$ that satisfies $\sigma_d(\zeta_N)=\zeta_N^d$.
\end{itemize}
\end{lemma}
\begin{proof}
This follows from Theorem~6.6 and Proposition~6.9 of \cite{MR1291394}.
\end{proof}

For a subgroup $H$ of $\GL_2(\ZZ/N\ZZ)$, let $\calF_N^H$ be the subfield of $\calF_N$ fixed by $H$ under the action of Lemma~\ref{L:Shimura basic}.

\begin{lemma} \label{L:FN basics}
\begin{romanenum}
\item \label{L:FN basics i}
The matrix $-I$ acts trivially on $\calF_N$ and the right action of $\GL_2(\ZZ/N\ZZ)/\{\pm I\}$ on $\calF_N$ is faithful.
\item \label{L:FN basics ii}
We have $\calF_N^{\GL_2(\ZZ/N\ZZ)}=\calF_1=\QQ(j)$ and $\calF_N^{\SL_2(\ZZ/N\ZZ)}=\QQ(\zeta_N)(j)$.
\item \label{L:FN basics iii}
The field $\QQ(\zeta_N)$ is algebraically closed in $\calF_N$.
\end{romanenum}
\end{lemma}
\begin{proof}
This also follows from Theorem~6.6 and Proposition~6.9 of \cite{MR1291394}.
\end{proof}

\subsection{Modular curves for finite groups}

Let $G$ be a subgroup of $\GL_2(\ZZ/N\ZZ)$ that satisfies $\det(G)=(\ZZ/N\ZZ)^\times$.  By Lemma~\ref{L:FN basics} and our assumption $\det(G)=(\ZZ/N\ZZ)^\times$, the field $\calF_N^G$ has transcendence degree $1$ and $\QQ$ is algebraically closed in $\calF_N^G$.   

\begin{definition}
The \defi{modular curve} $X_G$ is the smooth, projective and geometrically irreducible curve over $\QQ$ with function field $\calF_N^G$.    
\end{definition}

\subsection{Modular curves for open groups}

Consider an open subgroup $G$ of $\GL_2(\Zhat)$ that satisfies $\det(G)=\Zhat^\times$.   We define the \defi{modular curve} associated to $G$ to be the curve 
\[
X_G:=X_{\bbar{G}},
\] 
where $N$ is a positive integer that is divisible by the level of $G$ and $\bbar{G} \subseteq \GL_2(\ZZ/N\ZZ)$ is the reduction of $G$ modulo $N$.   Note that the function field $\QQ(X_G)=\calF_N^{\bbar{G}}$, and hence also $X_G$, does not depend on the choice of $N$.  

\begin{remark}
We will make use of both descriptions $X_G$ and $X_{\bbar{G}}$ of a modular curve interchangeably.  Working with open groups $G$ is more natural for our application and finite groups $\bbar{G}$ is better when dealing with computational issues.
\end{remark}

In the special case $G=\GL_2(\Zhat)$ and using $\QQ(X_G)=\QQ(j)$, we make an identification $X_G=\PP^1_\QQ$ and call it the \defi{$j$-line}.

Consider a larger group $G\subseteq G' \subseteq \GL_2(\Zhat)$.    The inclusion $\QQ(X_G) \supseteq \QQ(X_{G'})$ of fields induces a morphism $X_G\to X_{G'}$ of degree $[\pm G': \pm G]$.   In the special case $G'=\GL_2(\Zhat)$, we denote the morphism by $\pi_G\colon X_G \to \PP^1_\QQ$ (or $\pi_{\bbar{G}}\colon X_{\bbar{G}}\to \PP^1_\QQ$).\\

Let $\Gamma_G$ be the congruence subgroup $\SL_2(\ZZ) \cap G$  of $\SL_2(\ZZ)$; equivalently, the group of $A\in \SL_2(\ZZ)$ for which $A$ modulo $N$ lies in the group $\bbar{G}$ above.  We have an inclusion $\CC \cdot \QQ(X_G) \subseteq \CC(\calX_{\Gamma_G})$ of fields that both have degree $[\GL_2(\ZZ/N\ZZ):\pm G]= [\SL_2(\ZZ):\pm \Gamma_G]$ over $\CC(j)$.  Therefore, $\CC(\calX_{\Gamma_G})=\CC(X_G)$.   Using this equality of function fields, we shall identify $X_G(\CC)$ with the Riemann surface $\calX_{\Gamma_G}$.  Taking complex points, $\pi_G$ gives rise to the morphism $\calX_{\Gamma_G}\to \calX_{\SL_2(\ZZ)} \xrightarrow{\sim} \PP^1(\CC)$ of Riemann surfaces obtained by composing the natural quotient map with the isomorphism given by $j$.

The following property of $X_G$ is fundamental to our application to elliptic curves; it follows from Proposition~\ref{P:revised moduli property} which we will prove in \S\ref{SS:specializations}.

\begin{prop}  \label{P:intial moduli property}
Let $G$ be an open subgroup of $\GL_2(\Zhat)$ that satisfies $\det(G)=\Zhat^\times$ and $-I\in G$.
Let $E$ be any elliptic curve defined over $\QQ$ with $j_E\notin \{0,1728\}$.  Then $\rho_{E}^*(\Gal_\QQ)$ is conjugate in $\GL_2(\Zhat)$ to a subgroup of $G$ if and only if $j_E$ is an element of $\pi_G(X_G(\QQ)) \subseteq \QQ \cup \{\infty\}$.  
\end{prop}

\begin{remark}
As a warning we observe that in the literature, the notation $X_G$ sometimes denotes the modular curve that we call $X_{G^t}$, where $G^t$ is the group obtained by taking the transpose of the elements of $G$. The advantage of this alternate definition is that Proposition~\ref{P:intial moduli property} could be stated with $\rho_E$ instead of the dual representation $\rho_E^*$.   Our definition is more natural when working with the right actions of $G$ on spaces of modular forms.
\end{remark}

\section{Modular forms} \label{S:modular forms}

In this section, we recall what we need concerning modular forms.  For a modular form, we are particularly interested in computing arbitrarily many terms of the $q$-expansions at \emph{every} cusp.  For the basics on modular forms see \cite{MR1291394}.   For an overview on computing modular forms see \cite{2002.04717}; we will take our own approach using Eisenstein series that treats all the $q$-expansions at each cusp with equal importance.

For a subgroup $G$ of $\GL_2(\ZZ/N\ZZ)$ with $\det(G)=(\ZZ/N\ZZ)^\times$ and an even integer $k\geq 2$, we are especially interested in computing the space of modular forms $M_{k,G}$ from \S\ref{SS:MkG}.     We will see later that $M_{k,G}$ is the global sections of a line bundle on the modular curve $X_G$.  For $k$ large enough, the line bundle will be very ample and $M_{k,G}$ will allow us to compute an explicit model of $X_G$ in $\PP_{\QQ}^n$ for some $n$.

Fix a congruence subgroup $\Gamma$ of $\SL_2(\ZZ)$.  For a positive integer $N$, define the primitive $N$-th root of unity $\zeta_N:=e^{2\pi i/N}$ in $\CC$.  

\subsection{Setup and notation} \label{SS:modular form setup}

The group $\SL_2(\ZZ)$ acts by linear fractional transformations on the complex upper half-plane $\calH$ and the extended upper half-plane $\calH^*=\calH\cup \QQ \cup \{\infty\}$.   The quotient $\calX_\Gamma:=\Gamma\backslash \calH^*$ is a smooth compact Riemann surface (away from the cusps and elliptic points use the analytic structure coming from $\calH$ and extend to the full quotient).  

 Let $g$ be the genus of the Riemann surface $\calX_\Gamma$.    Let $P_1,\ldots, P_r$ be the cusps of $\calX_\Gamma$, i.e., the $\Gamma$-orbits of $\PP^1(\QQ) = \QQ \cup \{\infty\}$.    Let $Q_1,\ldots, Q_s$ be the elliptic points of $\calX_\Gamma$ and denote their orders by $e_1,\ldots, e_s$, respectively.   Each $e_i$ is either $2$ or $3$.  Let $v_2$ and $v_3$ be the number of elliptic points of $\calX_\Gamma$ of order $2$ and $3$, respectively.

Consider an integer $k\geq 0$.  For a meromorphic function $f$ on $\calH$ and a matrix $\gamma\in \GL_2(\RR)$ with positive determinant, define the meromorphic function $f|_k \gamma$ on $\calH$ by  $(f|_k \gamma)(\tau):= \det(\gamma)^{k/2}(c\tau+d)^{-k} f(\gamma \tau)$; we call this the \defi{slash operator} of weight $k$.

\subsection{Modular forms}

For an integer $k\geq 0$, we denote by $M_k(\Gamma)$ the set of \defi{modular forms} of weight $k$ on $\Gamma$; it is a finite dimensional complex vector space. Recall that each $f\in M_k(\Gamma)$ is a holomorphic function of the upper half-plane $\calH$ that satisfies $f|_k \gamma=f$ for all $\gamma\in \Gamma$ with the familiar growth condition at each cusp.  For each modular form $f\in M_k(\Gamma)$, we have 
\[
f(\tau) = \sum_{n=0}^\infty a_n(f)\, q_w^{n}
\]
for unique $a_n(f)\in \CC$, where $w$ is the width of the cusp $\infty$ of $\Gamma$ and $q_w:=e^{2\pi i \tau/w}$.  We call this power series in $q_w$, the \defi{$q$-expansion} of $f$ (at the cusp $\infty$).     For a subring $R$ of $\CC$, we denote by $M_k(\Gamma,R)$ the $R$-submodule of $M_k(\Gamma)$ consisting of modular forms whose $q$-expansion has coefficients in $R$.  

Define the graded $\CC$-algebra of modular forms on $\Gamma$ by
\[
R_\Gamma:= \bigoplus_{k \geq 0} M_k(\Gamma),
\]
where $k$ varies over all nonnegative integers.   The $\CC$-algebra $R_\Gamma$ is finitely generated.

\subsection{$q$-expansion at cusps} \label{SS:qexpansions}

We now consider $q$-expansions at all the cusps $P_1,\ldots, P_r$ of $\calX_\Gamma$.  For each $1\leq i \leq r$, choose a matrix $A_i \in \SL_2(\ZZ)$ so that $A_i\cdot \infty \in \QQ \cup \{\infty\}$ is a representative of the cusp $P_i$.    Let $w_i$ and $h_i$ be the minimal positive integers $m$ for which $\left(\begin{smallmatrix}1 & m \\0 & 1\end{smallmatrix}\right)$ lies in $A_i^{-1} \Gamma A_i$ and $A_i^{-1} (\pm\Gamma) A_i$, respectively.    We say that $P_i$ is a \defi{regular} cusp of $\Gamma$ if $w_i=h_i$; otherwise, it is an \defi{irregular} cusp and we have $w_i=2h_i$. 

 Consider a modular form $f\in M_k(\Gamma)$.  For $1\leq i \leq r$, we have 
\begin{align} \label{E:general qexpansions}
(f|_k A_i)(\tau) =\sum_{n=0}^\infty a_{n,i}(f) \,q_{w_i}^n
\end{align}
for unique $a_{n,i}(f) \in \CC$, where $q_{w_i}=e^{2\pi i \tau/{w_i}}$.  In particular, we can identify $f|_k A_i$ with a power series in $\CC\brak{q_{w_i}}$.
The ring $\CC\brak{q_{w_i}}$ is a discrete valuation ring and we denote the corresponding valuation by $\ord_{q_{w_i}}\colon \CC\brak{q_{w_i}} \twoheadrightarrow \ZZ \cup \{+\infty\}$.   Define the value
\[
\nu_{P_i}(f) := \frac{h_i}{w_i} \ord_{q_{w_i}}(f|_k A_i).
\]

\subsection{Modular forms as global sections}
\label{SS:global sections}

Fix an \emph{even} integer $k\geq 0$.   Take any modular form $f\in M_k(\Gamma)$.  Using that $f|_k \gamma=f$ for all $\gamma\in \Gamma$,  we find that the differential form 
\begin{align} \label{E:form on H}
(2\pi i)^{k/2} f(\tau) \, (d\tau)^{k/2} = w^{k/2} \bigg(\sum_{n=0}^\infty a_n(f)\, q_w^{n}\bigg) \bigg(\frac{dq_w}{q_w} \bigg)^{k/2}
\end{align}
on $\calH$ induces a meromorphic differential $k/2$-form $\omega_f$ on $\calX_\Gamma$.   
 
Let $\operatorname{div}(\omega_f)=\sum_{P\in \calX_\Gamma} n_P\cdot P$ be the divisor of $\omega_f$.   We now describe the integer $n_P$ in terms of $f$, cf.~equations (2.4.4) and (2.4.5) of \cite{MR1291394}.   If $P$ is a cusp, then $n_{P} = \nu_{P}(f)-k/2$ and hence $n_P +k/2 \geq 0$.  Now suppose $P\in \calX_\Gamma$ is not a cusp.  Choose a $z\in \calH$ that lies over $P$ and let $e$ be its order, i.e., the order of the cyclic group $\{\gamma\in \Gamma : \gamma\cdot z = z\}/(\Gamma \cap \{\pm I\})$.  We have $n_P = \nu_z(f)/e-k/2\cdot (1-1/e)$, where $\nu_z(f)$ is the order of vanishing of the meromorphic function $f$ at $z$.   Since $f$ is holomorphic at $z$, we have $n_P + \lfloor k/2\cdot (1-1/e) \rfloor \geq -k/2\cdot (1-1/e) + \lfloor k/2\cdot (1-1/e) > -1$.  So $n_P + \lfloor k/2\cdot (1-1/e) \rfloor \geq 0$ since $n_P$ is an integer.  Therefore, $\operatorname{div}(\omega_f) + D_k \geq 0$ where $D_k$ is the divisor
\begin{align} \label{E:divisor Dk}
\sum_{i=1}^r k/2\cdot P_i + \sum_{i=1}^s\lfloor k/2 \cdot (1-1/e_i) \rfloor \, \cdot  Q_i.
\end{align}
So we have an injective $\CC$-linear map 
\[
\psi_k\colon M_k(\Gamma) \to H^0(\calX_\Gamma,\calL_k), \quad f\mapsto \omega_f,
\]
where $\calL_k$ is the invertible sheaf $\Omega_{\calX_\Gamma}^{\otimes k/2}(D_k)$ on the Riemann surface $\calX_\Gamma$.  

Moreover, $\psi_k$ is an isomorphism.  Indeed, given a differential form $\omega\in H^0(\calX_\Gamma,\calL_k)$, it lifts to a differential form (\ref{E:form on H}) on $\calH$, where $f$ is a meromorphic function on $\calH$ that satisfies $f|_k \gamma=f$ for all $\gamma\in \Gamma$.  That $f$ is holomorphic on $\calH$ and has the desired conditions at the cusps follows from $\operatorname{div}(\omega) + D_k \geq 0$.\\

The invertible sheaf $\calL_k$ has degree  $B_{k,\Gamma}:= k/2\cdot (2g-2) + k/2 \cdot r + \lfloor k/4 \rfloor \cdot v_2 +  \lfloor k/3 \rfloor \cdot v_3.$  
We have 
\begin{align} \label{E:BkGamma bound}
B_{k,\Gamma} \leq k/12 \cdot [\SL_2(\ZZ):\pm \Gamma]\end{align} 
since $g-1+v_2/4+v_3/3+r/2 = [\SL_2(\ZZ):\pm \Gamma]/12$ by \cite[Proposition~1.40]{MR1291394}. 
 
When $k\geq 2$,  we have $\deg \calL_k > 2g-2$, cf.~\cite[\S2.6]{MR1291394} and use $r\geq 1$.  So if $k\geq 2$, the Riemann--Roch theorem implies that
\begin{align} \label{E:dimension formula}
\dim_\CC M_k(\Gamma)   
= \dim_\CC  \deg(\calL_k)-g+1 
= (k-1)(g-1) + k/2 \cdot r + v_2 \cdot \lfloor k/4 \rfloor + v_3 \cdot \lfloor k/3 \rfloor.
\end{align}
In the excluded case $k=0$, we have $M_0(\Gamma)=\CC$.  We now describe how many terms of the $q$-expansions of a modular form $f$ are required to determine it.  

\begin{lemma}[Sturm bound] \label{L:advanced Sturm}
For any $f,f'\in M_k(\Gamma)$, we have $f=f'$ if and only if $\sum_{j=1}^r \nu_{P_j}(f-f')  >  B_{k,\Gamma}$.
\end{lemma}
\begin{proof}
  If $f=f'$, then $\sum_{j=1}^r \nu_{P_j}(f-f')=+\infty$.   So take any distinct $f,f'\in M_k(\Gamma)$.  It remains to show that $\sum_{j=1}^r \nu_{P_j}(f-f')  \leq  B_{k,\Gamma}$.  Without loss of generality, we may assume that $f\neq0$ and $f'= 0$.

  The coefficient of the divisor $\operatorname{div}(\omega_f)+D_k$ at the cusp $P_i$  is $(\nu_{P_i}(f)-k/2)+k/2=\nu_{P_i}(f)$.  Since $\operatorname{div}(\omega_f)+D_k \geq 0$, we have $\sum_{i=1}^r \nu_{P_i}(f) \leq \deg(\operatorname{div}(\omega_f)+D_k) = k/2\cdot (2g-2)  + \deg D_k=B_{k,\Gamma}$.  Therefore, $\sum_{i=1}^r \nu_{P_i}(f) \leq B_{k,\Gamma}$ as required.
\end{proof}

Now assume further that $-I\in \Gamma$ and hence $M_k(\Gamma)=0$ for odd $k$.  Using that $\calL_{k} \otimes \calL_{k'} \subseteq \calL_{k+k'}$ for any even non-negative integers $k$ and $k'$, we find that the  isomorphisms $\psi_k$ combine to give an isomorphism of graded $\CC$-algebras:
\[
\psi\colon R_\Gamma \xrightarrow{\sim} \bigoplus_{k\geq 0 \text{ even}} H^0(\calX_\Gamma,\calL_k).
\]

\subsection{Actions} \label{SS:actions}

Fix positive integers $k$ and  $N$.  Since $\Gamma(N)$ is normal in $\SL_2(\ZZ)$, the slash operator of weight $k$ gives a right action of $\SL_2(\ZZ)$ on $M_k(\Gamma(N))$.    Take any modular form $f=\sum_{n=0}^\infty a_n(f) q_N^n$ in $M_k(\Gamma(N))$.  For every field automorphism $\sigma$ of $\CC$, there is a unique modular form $\sigma(f) \in  M_k(\Gamma(N))$ whose $q$-expansion is $\sum_{n=0}^\infty \sigma(a_n(f))\, q_N^{n}$.   This defines an action of $\Aut(\CC)$ on $M_k(\Gamma)$.   

The following describes how these actions of $\SL_2(\ZZ)$ and $\Aut(\CC)$ interact; it is \cite[Theorem~3.3]{BN2019} (they give two proofs, one using Katz modular forms and another making use of a result of Khuri-Makdisi on Eisenstein series, cf.~Theorem~\ref{T:Eisenstein span}).

\begin{lemma} \label{L:action compatibility}
Take any modular form $f\in M_k(\Gamma(N))$.  Take any $\sigma\in \Aut(\CC)$ and let $m$ be the unique element of $(\ZZ/N\ZZ)^\times$ for which $\sigma(\zeta_N)=\zeta_N^m$.   Take any $\gamma=\left(\begin{smallmatrix}a & b \\c & d\end{smallmatrix}\right) \in \SL_2(\ZZ)$ and let $\gamma'$ be any element of $\SL_2(\ZZ)$ congruent to $\left(\begin{smallmatrix}a & mb \\m^{-1}c & d\end{smallmatrix}\right)$ modulo $N$.  Then  $\sigma(f|_k \gamma) = \sigma(f)|_k \gamma'$.
\end{lemma}

Using Lemma~\ref{L:action compatibility} with $\sigma\in \Aut(\CC/\QQ(\zeta_N))$ and $\gamma\in \Gamma(N)$, we find that the action of $\SL_2(\ZZ)$ on $M_k(\Gamma(N))$ via the slash operator gives rise to a well-defined action on $M_k(\Gamma(N),\QQ(\zeta_N))$.  

We have an isomorphism $(\ZZ/N\ZZ)^\times \xrightarrow{\sim} \Gal(\QQ(\zeta_N)/\QQ)$, $d\mapsto \sigma_d$, where $\sigma_d(\zeta_N)=\zeta_N^d$.     We now recall an action of $\GL_2(\ZZ/N\ZZ)$ on $M_{k}(\Gamma(N),\QQ(\zeta_N))$ viewed as a $\QQ$-vector space.   

\begin{lemma} \label{L:star action}
There is a unique right action $*$ of $\GL_2(\ZZ/N\ZZ)$ on $M_k(\Gamma(N),\QQ(\zeta_N))$ such that the following hold:
\begin{itemize}
\item
if $A \in \SL_2(\ZZ/N\ZZ)$, then  $f*A = f|_k \gamma$, where $\gamma$ is any matrix in $\SL_2(\ZZ)$ that is congruent to $A$ modulo $N$,
\item
if $A=\left(\begin{smallmatrix}1 & 0 \\0 & d\end{smallmatrix}\right)$, then $f*A=\sigma_d(f)$.
\end{itemize}
\end{lemma}
\begin{proof}
See \cite[\S3]{BN2019}; it is  Lemma~\ref{L:action compatibility} that allows us to show that the actions in the two parts are compatible.
\end{proof}

\begin{remark}
We obtain a right action $*$ of $\GL_2(\ZZ/N\ZZ)$ on the graded ring $\bigoplus_{k\geq 0} M_k(\Gamma(N),\QQ(\zeta_N))$ that respects multiplication.  If $f$ and $f'\neq 0$ are modular forms in $M_k(\Gamma(N),\QQ(\zeta_N))$, then $f/f' \in \calF_N$ and for $A\in \GL_2(\ZZ/N\ZZ)$, we have $(f/f')*A = (f*A)/(f'*A)$ with the action from \S\ref{SS:modular functions}.
\end{remark}

Now suppose that $k\neq 1$.  The natural map
\[
M_k(\Gamma(N),\QQ(\zeta_N)) \otimes_{\QQ(\zeta_N)} \CC \to M_k(\Gamma(N))
\]
is an isomorphism of complex vector spaces, cf.~\cite[\S1.7]{MR0447119}.    For any congruence subgroup $\Gamma\subseteq \SL_2(\ZZ)$ whose level divides $N$, taking $\Gamma$-invariants shows that the natural map 
\begin{align} \label{E:tensor isom}
M_k(\Gamma,\QQ(\zeta_N)) \otimes_{\QQ(\zeta_N)} \CC \to M_k(\Gamma)
\end{align} 
is an isomorphism of complex vector spaces.

\subsection{The spaces $M_{k,G}$} \label{SS:MkG}

Fix a positive integer $N$ and let $G$ be a subgroup of $\GL_2(\ZZ/N\ZZ)$ that satisfies $\det(G)=(\ZZ/N\ZZ)^\times$.   For each integer $k\geq 0$, we define the $\QQ$-vector space
\[
M_{k,G}:=M_k(\Gamma(N),\QQ(\zeta_N))^G,
\]
where we are considering the subspace fixed by the $G$-action $*$ from Lemma~\ref{L:star action}.   Let $\Gamma_G$ be the congruence subgroup of $\SL_2(\ZZ)$ consisting of matrices that are congruent modulo $N$ to an element of $H:=G\cap\SL_2(\ZZ/N\ZZ)$.  Note that $M_{k,G} \subseteq M_k(\Gamma(N),\QQ(\zeta_N))^H = M_k(\Gamma_G,\QQ(\zeta_N))$. 

\begin{lemma} \label{L:tensor agreement}
The natural homomorphisms
\[
M_{k,G} \otimes_{\QQ} \QQ(\zeta_N) \to M_k(\Gamma_G,\QQ(\zeta_N)) \quad \text{ and } \quad 
M_{k,G} \otimes_{\QQ} \CC \to  M_k(\Gamma_G)
\]
are isomorphisms for $k\neq 1$.
\end{lemma}
\begin{proof}
Since $H$ is normal in $G$, we have a right action of $G/H$ on $M_k(\Gamma(N),\QQ(\zeta_N))^H = M_k(\Gamma_G,\QQ(\zeta_N))$.  Let $\varphi\colon G/H\to \Gal(\QQ(\zeta_N)/\QQ)$ be the homomorphism satisfying $\varphi(A)(\zeta_N)=\zeta_N^{\det A}$; it is an isomorphism since $\det(G)=(\ZZ/N\ZZ)^\times$.   Since $G/H$ is abelian, the isomorphism $\varphi$ induces a (left) action $\bullet$ of $\Gal(\QQ(\zeta_N)/\QQ)$ on $M_k(\Gamma_G,\QQ(\zeta_N))$.  We have $\sigma \bullet (cf)=\sigma(c) (\sigma\bullet f)$ for all $c\in \QQ(\zeta_N)$, $f\in M_k(\Gamma_G,\QQ(\zeta_N))$ and $\sigma\in \Gal(\QQ(\zeta_N)/\QQ)$.  By Galois descent for vector spaces (see the corollary to Proposition~6 in Chapter V \S10 of \cite{MR1994218}), the natural homomorphism 
\begin{align*}
M_{k,G}\otimes_\QQ \QQ(\zeta_N) = M_k(\Gamma_G,\QQ(\zeta_N))^{\Gal(\QQ(\zeta_N)/\QQ)} \otimes_\QQ \QQ(\zeta_N) \to M_k(\Gamma_G,\QQ(\zeta_N)) 
\end{align*}
is an isomorphism of $\QQ(\zeta_N)$-vector spaces.   The lemma follows by tensoring this isomorphism up to $\CC$ and using  (\ref{E:tensor isom}).
\end{proof}

Later we will need the following which guarantees the existence of nonzero weight $3$ modular forms whenever we have $-I \notin G$. 

\begin{lemma} \label{L:weight 3 existence}
If $-I \notin G$, then $M_{3,G}\neq 0$.
\end{lemma}
\begin{proof}
We need only verify that $M_3(\Gamma_G)\neq 0$ by Lemma~\ref{L:tensor agreement}.  There is an explicit formula for the dimension $d$ of $M_3(\Gamma_G)$ over $\CC$, cf.~\cite[Theorem~2.25]{MR1291394}.  From this formula, we will clearly have $d\geq 1$ when $\Gamma_G$ has genus at least $1$.   In the genus $0$ case, we verified that $d\geq 1$ by using the classification of genus $0$ congruence subgroups from \cite{MR2016709}. 
\end{proof}
 
In \S\ref{SS:basis of M_{k,G}}, we will describe how to compute a basis of $M_{k,G}$ using Eisenstein series when $k\geq 2$.

\subsection{Eisenstein series}

We now describe some explicit modular forms.  See \cite[\S3]{MR0447119} for the basics on Eisenstein series.  For further information, we refer to \S\S2--3 of \cite{BN2019} where all the basic results below are summarized and referenced (except for the explicit constant $c_0$ in Lemma~\ref{L:Eisenstein expansions}, see \cite[Lemma~3.1]{MR3705252} instead).  

Fix positive integers $k$ and $N$.  Take any pair $\alpha\in (\ZZ/N\ZZ)^2$ and choose $a,b\in \ZZ$ so that $\alpha\equiv (a,b) \pmod{N}$.  With $\tau\in \calH$, consider the series
\begin{align} \label{E:Ealpha def}
E_\alpha^{(k)}(\tau,s) = \frac{(k-1)!}{(-2\pi i)^k} 
\sum_{\substack{\omega\in \ZZ + \ZZ\tau\\\omega \neq - (a\tau+b)/N}}  \Big(\frac{a\tau+b}{N} + \omega\Big)^{-k}
\cdot \Big| \frac{a\tau+b}{N} + \omega\Big|^{-2s}.
\end{align}
The series (\ref{E:Ealpha def}) converges when the real part of $s\in \CC$ is large enough.   Hecke proved that $E_\alpha^{(k)}(\tau,s)$ can be analytically continued to all $s\in \CC$.    Using this analytic continuation, we define the \defi{Eisenstein series} 
\[
E_{\alpha}^{(k)}(\tau) :=E_\alpha^{(k)}(\tau,0).
\]
When $k\geq 3$, we can also obtain $E_\alpha^{(k)}(\tau)$ by simply setting $s=0$ in the series (\ref{E:Ealpha def}).    

For $\gamma\in \SL_2(\ZZ)$, we have
\begin{align*}
E_{\alpha}^{(k)} |_k \gamma = E_{\alpha \gamma}^{(k)},
\end{align*}
where $\alpha \gamma \in (\ZZ/N\ZZ)^2$ denotes matrix multiplication.  In particular, $E_\alpha^{(k)}$ is fixed by $\Gamma(N)$.

\begin{lemma} \label{L:Eisenstein expansions}
Suppose that $k\geq 1$ and $k\neq 2$.  Then $E^{(k)}_\alpha$ is a modular form of weight $k$ on $\Gamma(N)$ with $q$-expansion
\[
c_0 + \sum_{\substack{m,n \geq 1\\ m\equiv a \bmod{N}}} n^{k-1} \zeta_N^{bn} q_N^{mn} +(-1)^k  \sum_{\substack{m,n \geq 1\\ m\equiv -a \bmod{N}}} n^{k-1} \zeta_N^{-bn} q_N^{mn},
\]
where $c_0$ is an element of $\QQ(\zeta_N)$.  When $k=1$, we have 
\[
c_0= \begin{cases}
     0 & \text{if $a\equiv b \equiv 0\pmod{N}$}, \\
     \frac{1}{2} \, \frac{1+\zeta_N^b}{1-\zeta_N^b} & \text{if $a\equiv 0 \pmod{N}$ and $b\not\equiv 0 \pmod{N}$},\\
     \frac{1}{2} -  \frac{a_0}{N} & \text{if $a\not\equiv 0 \pmod{N}$,}     
\end{cases}
\]
where $0\leq a_0<N$ is the integer congruent to $a$ modulo $N$.
\end{lemma}

\begin{remark}
For the excluded case $k=2$, one should instead consider $E_{\alpha}^{(2)}- E_{(0,0)}^{(2)}$ which belongs to $M_2(\Gamma(N))$ and has a computable $q$-expansion.
\end{remark}

Remarkably, we can recover all higher weight modular forms from the Eisenstein series of weight $1$.  

\begin{thm}[Khuri-Makdisi] \label{T:Eisenstein span}
Suppose $N\geq 3$.   Let $\calR_N$ be the $\CC$-subalgebra of $R_{\Gamma(N)}=\bigoplus_{k\geq 0} M_k(\Gamma(N))$ generated by the Eisenstein series $E_\alpha^{(1)}$ with $\alpha\in (\ZZ/N\ZZ)^2$.    Then $\calR_N$ contains all modular forms of weight $k$ on $\Gamma(N)$ for all $k\geq 2$. 
\end{thm}
\begin{proof}
 This particular formulation of results of Khuri-Makdisi \cite{MR2904927} is Theorem~3.1 of \cite{BN2019}.
\end{proof}

For our applications, the important part of this theorem is we have an explicit set of modular forms that span $M_k(\Gamma(N),\QQ(\zeta_N))$ and that we understand the action $*$ of $\GL_2(\ZZ/N\ZZ)$ on these modular forms. 

\begin{cor}  \label{C:Eisenstein span}
Fix integers $k\geq 2$ and $N\geq 3$.
\begin{romanenum}
\item \label{C:Eisenstein span i}
The $\QQ(\zeta_N)$-vector space $M_k(\Gamma(N),\QQ(\zeta_N))$ is spanned by the set
\begin{align*} \label{E:gen set KN}
\Big\{ E_{\alpha_1}^{(1)} \cdots E_{\alpha_k}^{(1)} : \alpha_1,\ldots, \alpha_k \in (\ZZ/N\ZZ)^2-\{0\} \Big\}.
\end{align*}
\item \label{C:Eisenstein span ii}
For $\alpha_1,\ldots, \alpha_k \in (\ZZ/N\ZZ)^2$ and $A\in \GL_2(\ZZ/N\ZZ)$, we have 
\[
(E_{\alpha_1}^{(1)} \cdots E_{\alpha_k}^{(1)})*A  = E_{\alpha_1A}^{(1)} \cdots E_{\alpha_kA}^{(1)}.
\]
\end{romanenum}
\end{cor}
\begin{proof}
Let $S$ be the set of modular forms $E_{\alpha_1}^{(1)} \cdots E_{\alpha_k}^{(1)}$ with $\alpha_i \in (\ZZ/N\ZZ)^2$.  We have $S\subseteq M_k(\Gamma(N),\QQ(\zeta_N))$ by Lemma~\ref{L:Eisenstein expansions}.  As noted in \S\ref{SS:actions}, the natural map $M_k(\Gamma(N),\QQ(\zeta_N)) \otimes_{\QQ(\zeta_N)} \CC \to M_k(\Gamma(N))$ is an isomorphism.    Since $S$ spans $M_k(\Gamma(N))$ by Theorem~\ref{T:Eisenstein span}, we deduce that $S$ spans the $\QQ(\zeta_N)$-vector space $M_k(\Gamma(N),\QQ(\zeta_N))$. This proves part (\ref{C:Eisenstein span i}) after noting that $E_{(0,0)}^{(1)}=0$.

We now prove (\ref{C:Eisenstein span ii}) for a fixed matrix $A\in \SL_2(\ZZ/N\ZZ)$.  Choose any $\gamma\in \SL_2(\ZZ)$ for which $\gamma\equiv A\pmod{N}$.     We have $E_{\alpha_i}^{(1)}|_1 \gamma = E^{(1)}_{\alpha_i \gamma}=E^{(1)}_{\alpha_i A}$ for $1\leq i \leq k$ and hence 
\[
(E^{(1)}_{\alpha_1}\cdots E^{(1)}_{\alpha_k})*A = (E^{(1)}_{\alpha_1}\cdots E^{(1)}_{\alpha_k})|_k \gamma = (E_{\alpha_1}^{(1)}|_1 \gamma)\cdots (E_{\alpha_k}^{(1)}|_1 \gamma) = E^{(1)}_{\alpha_1 A} \cdots E^{(1)}_{\alpha_k A}.
\]   
It thus suffices to prove (\ref{C:Eisenstein span ii}) for any matrix $A=\left(\begin{smallmatrix}1 & 0 \\0 & d\end{smallmatrix}\right) \in \GL_2(\ZZ/N\ZZ)$.   If $\alpha_i=(a_i,b_i)\in (\ZZ/N\ZZ)^2$, the explicit $q$-expansion of $E_{\alpha_i}^{(1)}$ in Lemma~\ref{L:Eisenstein expansions} gives us that $\sigma_d(E_{\alpha_i}^{(1)}) = E^{(1)}_{(a_i,b_id)}  = E^{(1)}_{\alpha_i A}$.  Therefore, $(E_{\alpha_1}^{(1)} \cdots E_{\alpha_k}^{(1)})*A = \sigma_d(E_{\alpha_1}^{(1)} \cdots E_{\alpha_k}^{(1)})=E_{\alpha_1A}^{(1)} \cdots E_{\alpha_kA}^{(1)}$.
\end{proof}

\begin{cor}  \label{C:Eisenstein span 2}
Fix integers $k\geq 2$ and $N\geq 3$.  Let $G$ be a subgroup of $\GL_2(\ZZ/N\ZZ)$ that satisfies $\det(G)=(\ZZ/N\ZZ)^\times$.  Then the $\QQ$-vector space $M_{k,G}$ is spanned by the set of modular forms of the form
\begin{align} \label{E:is a trace}
\sum_{g\in G} \zeta_N^{j \det g}\, E_{\alpha_1 g}^{(1)} \cdots E_{\alpha_k g}^{(1)}
\end{align}
with $\alpha_i \in (\ZZ/N\ZZ)^2-\{0\}$ and $0\leq j < \phi(N):=|(\ZZ/N\ZZ)^\times|$.
\end{cor}
\begin{proof}
Define the $\QQ$-linear map $T\colon M_k(\Gamma(N),\QQ(\zeta_N))\to M_{k,G}$ by $f\mapsto\sum_{g\in G} f* g$.   The map $T$ is surjective since it is multiplication by $|G|$ when restricted to $M_{k,G}$. 

Let $S$ be the set of modular forms $\zeta_N^j E_{\alpha_1}^{(1)} \cdots E_{\alpha_k}^{(1)}$ with $\alpha_i \in (\ZZ/N\ZZ)^2-\{0\}$ and $0\leq j < \phi(N)$.  By Corollary~\ref{C:Eisenstein span}(\ref{C:Eisenstein span i}), we find that $S$ spans $M_k(\Gamma(N),\QQ(\zeta_N))$ as a $\QQ$-vector space.   Therefore, the $\QQ$-vector space $M_{k,G}$ is spanned by the set $T(S)$.     

The corollary follows by noting that $T(\zeta_N^j E_{\alpha_1}^{(1)} \cdots E_{\alpha_k}^{(1)})$ agrees with (\ref{E:is a trace}) by Corollary~\ref{C:Eisenstein span}(\ref{C:Eisenstein span ii}).  
\end{proof}

\subsection{Finding a basis for $M_{k,G}$} \label{SS:basis of M_{k,G}}

 Fix a positive integer $N$ and let $G$ be a subgroup of $\GL_2(\ZZ/N\ZZ)$ that satisfies $\det(G)=(\ZZ/N\ZZ)^\times$ and $-I \in G$.     Let $\Gamma_G$ be the congruence subgroup of $\SL_2(\ZZ)$ consisting of matrices whose image modulo $N$ lies in $G$.   Fix notation as in \S\ref{SS:modular form setup} with $\Gamma:=\Gamma_G$.  In particular, let $P_1,\ldots, P_r \in \QQ\cup\{\infty\}$ be representatives of the cusps of $\calX_\Gamma$.   The cusps are all regular since $ -I \in \Gamma_G$.
 
Fix an integer $k\geq 0$.   In this section, we describe how to compute a basis of $M_{k,G}$.   A modular form in our basis will be explicitly given by its $q$-expansions at each cusp of $\calX_{\Gamma_G}$ with enough terms computed to uniquely determine it.  Moreover, we will be able to compute arbitrarily many terms of these $q$-expansions.    

We may assume that $k$ is even since $-I \in G$ implies that $M_{k,G}=0$ when $k$ is odd.   We may assume that $k\geq 2$ since $M_{0,G}=\QQ$.  We can further assume that $N\geq 3$ (when $N\leq 2$, we can replace $G$ by its inverse image under the reduction map $\GL_2(\ZZ/4\ZZ)\to \GL_2(\ZZ/N\ZZ)$; this does not change $M_{k,G}$).\\

Let $d$ be the dimension of $M_{k,G}$ over $\QQ$; it agrees with the dimension of the complex vector space $M_k(\Gamma_G)$ and hence is computable by (\ref{E:dimension formula}).  We may assume that $d\geq 1$ since otherwise $M_{k,G}=0$.  Set $B_{k,G}:=B_{k,\Gamma_G}$ and let $b_{k,G}$ be the smallest integer satisfying 
\begin{align} \label{E:b inequality}
b_{k,G}  >  B_{k,G}\cdot N /[\SL_2(\ZZ):\Gamma_G].
\end{align}
Define $m:=\sum_{i=1}^r m_i$, where  $m_i := \lceil  w_i b_{k,G}/N\rceil$.   

For each $1\leq i \leq r$, we have chosen a matrix $A_i \in \SL_2(\ZZ)$ satisfying $A_i\cdot \infty=P_i$ and this gives rise to $q$-expansions (\ref{E:general qexpansions}).      Define the $\QQ$-linear map
\begin{align*}
\varphi_k\colon M_{k,G} &\to \QQ(\zeta_N)^{m},\\
\quad f &\mapsto  \big(a_{0,1}(f),\ldots, a_{m_1-1,1}(f), \,\,a_{0,2}(f),\ldots, a_{m_2-1,2}(f),\,\,\ldots,\,\, a_{0,r}(f),\ldots, a_{m_r-1,r}(f)\big).
\end{align*}

\begin{lemma}
The map $\varphi_k$ is injective.
\end{lemma}
\begin{proof}
Take any $f\in \ker\varphi_k$.  For each $1\leq i \leq r$, we have
\[
\nu_{P_i}(f) = \ord_{q_{w_i}}(f) \geq  m_i \geq w_i b_{k,G}/N
\]
and hence $\sum_{i=1}^r \nu_{P_i}(f)  \geq \sum_{i=1}^r w_i\cdot b_{k,G}/N$.  We have $\sum_{i=1}^r w_i = [\SL_2(\ZZ): \Gamma_G]$; one way to see this is to add the ramification indices of the cusps with respect to the natural morphism $\calX_{\Gamma_G} \to \calX_{\SL_2(\ZZ)}$ of degree $[\SL_2(\ZZ):\pm \Gamma_G]=[\SL_2(\ZZ):\Gamma_G]$.  Therefore,  $\sum_{i=1}^r \nu_{P_i}(f)  \geq [\SL_2(\ZZ):\Gamma_G] \cdot b_{k,G}/N > B_{k,G}$.  We have $f=0$ by Lemma~\ref{L:advanced Sturm}.
\end{proof}

\begin{remark}
Of course the map $\varphi$ remains injective if we replace $b_{k,G}$ by any larger integer $b$.  In particular,  any integer $b> kN/12$ will work by (\ref{E:BkGamma bound}).
\end{remark}

\begin{algorithm}
This algorithm computes a basis $\beta$ of the $\QQ$-vector space $\varphi_k(M_{k,G})$.

\begin{enumerate}
\item \label{step:compute E}
Compute the $q$-expansion $E_{\alpha}^{(1)} + O(q_N^{b})$ for all $\alpha\in (\ZZ/N\ZZ)^2-\{0\}$ using the explicit expression from Lemma~\ref{L:Eisenstein expansions}.  

\item 
Let $S$ be the set of all $k$-tuples $(\alpha_1,\ldots, \alpha_k)$ with $\alpha_i \in (\ZZ/N\ZZ)^2-\{0\}$.  Set $\beta:=\emptyset$.

\item  \label{step:choose alpha}
Choose a $k$-tuple $(\alpha_1,\ldots, \alpha_k) \in S$.   For each integer $0\leq j < \phi(N)$, define the modular form 
\begin{align} \label{E:algorithm expansion 0}
f_j:=\sum_{g\in G} \zeta_N^{j \det g}\, E_{\alpha_1 g}^{(1)} \cdots E_{\alpha_k g}^{(1)};
\end{align}
it lies in $M_{k,G}$ by Corollary~\ref{C:Eisenstein span 2}.
Using our approximations of the $E_{\alpha}^{(1)}$ from Step \ref{step:compute E},  compute 
\begin{align} \label{E:algorithm expansion}
f_j|_k A_i + O(q_N^{b}) = \sum_{g\in G} \zeta_N^{j \det g}\, E_{\alpha_1 g A_i}^{(1)} \cdots E_{\alpha_k g A_i}^{(1)} + O(q_N^{b}) 
\end{align}
for all $1\leq i \leq r$.  Since $f_j|_k A_i \in \QQ(\zeta_N)\brak{q_{w_i}}$, this gives us $f_j|_k A_i + O(q_{w_i}^{m_i})$ where $m_i=\lceil w_ib/N \rceil$.   In particular, we can compute the vector $\varphi_k(f_j) \in \QQ(\zeta_N)^m$.

Running over the integers $0\leq j < \phi(N)$, if $\varphi_k(f_j)$ is not in the span of $\beta$ in $\QQ(\zeta_N)^m$ as a $\QQ$-vector space, then adjoin $\varphi_k(f_j)$ to the set $\beta$.

\item  \label{step:end}
Remove from the set $S$ all $k$-tuples of the form $(\alpha_{\sigma(1)} g, \ldots,\alpha_{\sigma(k)}g)$ for some $\sigma\in \mathfrak{S}_k$ and some $g\in G$.  If $|\beta| < d$, then return to Step~\ref{step:choose alpha}.

\end{enumerate}

\end{algorithm}

Since $\varphi_k$ is injective and $\dim_{\QQ} M_{k,G}=d$, if the algorithm terminates, then it will produce a basis $\beta$ of the $\QQ$-vector space $\varphi_k(M_{k,G})$.  The modular form (\ref{E:algorithm expansion 0}) does not change if we replace $(\alpha_1,\ldots, \alpha_k)$ with a $k$-tuple $(\alpha_{\sigma(1)} g, \ldots,\alpha_{\sigma(k)}g)$ with $\sigma\in \mathfrak{S}_k$ and $g\in G$; this justifies the elements removed from the set $S$ in Step~\ref{step:end} (they would not produce new modular forms).   Corollary~\ref{C:Eisenstein span 2} ensures that we will eventually find enough modular forms so that the algorithm halts with a set $\beta$ of cardinality $d$.\\

This algorithm computes the basis of $M_{k,G}$ in the sense that there is a unique basis $F_1,\ldots, F_d$ of $M_{k,G}$ as a $\QQ$-vector space so that $\varphi_k(F_1),\ldots, \varphi_k(F_d)$ is the ordered basis $\beta$. For each $1\leq e \leq d$,  the modular form $F_e$ will be of the form (\ref{E:algorithm expansion 0}) for explicit $j$ and $(\alpha_1,\ldots, \alpha_k)$.  By computing the $q$-expansions of the relevant Eisenstein series to higher precision, one can compute an arbitrary number of terms in the $q$-expansion of $F_e$ at each cusp. By Lemma~\ref{L:tensor agreement}, $F_1,\ldots, F_d$ is also a basis of the complex vector space $M_k(\Gamma_G)$.

\begin{remark} \label{R:main remarks}
\begin{romanenum}
\item
Since our basis of $M_{k,G}$ is given by a $q$-expansion at each cusp, we can also compute subspaces obtained by forcing vanishing conditions at the cusps.  For example, let $S_{k,G}$ be the $\QQ$-subspace of $M_{k,G}$ consisting of modular forms $f\in M_{k,G}$ that satisfy $a_{0,i}(f)=0$ for all $1\leq i \leq r$.   Alternatively, the group $\GL_2(\ZZ/N\ZZ)$ acts on the cusps forms $S_k(\Gamma(N),\QQ(\zeta_N))$ via $*$ and we have $S_{k,G}=S_k(\Gamma(N),\QQ(\zeta_N))^G$.    Thus $\varphi(S_{k,G})$ is an explicit subspace of $\varphi(M_{k,G})$ and a basis can be computed.

\item  \label{R:main remarks i}
Once you have a basis of $M_{k,G}$, you can construct a ``nicer'' one.   We have $(2N)^k\cdot \varphi_k(\beta) \subseteq \ZZ[\zeta_N]^m$ by considering the $q$-expansion in Lemma~\ref{L:Eisenstein expansions}.  Let 
\[
\iota\colon \ZZ[\zeta_N]^m \xrightarrow{\sim} (\ZZ^{\phi(N)})^m = \ZZ^{\phi(N)m}
\] 
be the isomorphism of $\ZZ$-modules obtained by applying to the coordinates the map $\ZZ[\zeta_N]\to \ZZ^{\phi(N)}$, $\sum_{i=1}^{\phi(N)} a_{i} \zeta_N^{i-1} \mapsto (a_1,\ldots, a_{\phi(N)})$.   Define $C:= \iota((2N)^k\cdot \varphi_k(\beta)) \subseteq  \ZZ^{\phi(N)m}$.   

Let $L$ be the subgroup of $\ZZ^{\phi(N)m}$ generated by $C$ and let $L'$ be its \emph{saturation}, i.e., the group of $a\in \ZZ^{\phi(N)m}$ such that $da \in L$ for some positive integer $d$.     Let $C'$ be a basis of the free abelian group $L'$; applying the LLL algorithm \cite[\S2.6]{MR1228206} will produce basis elements with small entries.    Define $\beta':=\iota^{-1}(C')$; it is also a basis of the $\QQ$-vector space $\varphi_k(M_{k,G})$ but with integral (and in practice simpler) entries that the original basis $\beta$.

\item
When some modular forms $f\in M_{k,G}$ are already known, we can adjoin the vectors $\varphi_k(f)$ to the set $\beta$ before beginning the algorithm (making sure the set $\beta$ is linearly independent over $\QQ$).  For example, some modular forms in $M_{k,G}$ can be obtained from modular forms of lower weight or from a larger group.

\item
Set $H:=G\cap \SL_2(\ZZ/N\ZZ)$ and let $R$ be a set of representatives of the cosets $G/H$.   With notation as above, we have 
\[
f_j|_k A_i = \sum_{g\in R}  \zeta_N^{j \det g}\,\Big(\sum_{h\in H} E_{\alpha_1 g h A_i}^{(1)} \cdots E_{\alpha_k g h A_i}^{(1)}\Big).
\]
The inner sums of this expression can be computed first and used for all $j$.

Another observation that makes computing $f_j|_k A_i$ quicker is that its $q$-expansion is a power series in $q_{w_i}=q_N^{N/w_i}$.  Let $U_i$ be the subgroup of $A_i^{-1} H A_i$ generated by $\left(\begin{smallmatrix}1 & w_i \\0 & 1\end{smallmatrix}\right)$ and let $R_i$ be a set of representatives of the cosets $(A_i^{-1} H A_i) /U_i$.  We then have
\[
\sum_{h\in H} E_{\alpha_1 g h A_i}^{(1)} \cdots E_{\alpha_k g h A_i}^{(1)} = \sum_{h\in A_i^{-1} H A_i} E_{\alpha_1 g A_i  h }^{(1)} \cdots E_{\alpha_k  g A_i  h}^{(1)} = \sum_{u\in U_i}\Big( \sum_{h\in R_i } E_{\alpha_1 g A_i  h }^{(1)} \cdots E_{\alpha_k  g A_i  h}^{(1)}\Big)*u.
\]
An easy computation shows that for a modular form $f=\sum_{n=0}^\infty c_n q_N^n$ of weight $k$, we have 
\[
\sum_{u\in U_i} f*u = \frac{N}{w_i} \sum_{\substack{n=0\\n \equiv 0 \pmod{N/w_i}}}^\infty   c_{n} q_N^n.
\]
\item 
Consider a subgroup $G\subseteq \GL_2(\ZZ/N\ZZ)$ that satisfies $-I\in G$.  Let $L_G$ be the subfield of $\QQ(\zeta_N)$ fixed by the group $\{\sigma_d : d\in \det(G)\}$.   We have $L_G=\QQ$ if and only if $\det(G)=(\ZZ/N\ZZ)^\times$.   We can define $M_{k,G}$ as before.    Now $M_{k,G}$ is an $L_G$-vector space of dimension $d$ and hence a $\QQ$-vector space of dimension $d\cdot [L_G:\QQ] = d [(\ZZ/N\ZZ)^\times : \det(G)]$.  Some easy changes in the above algorithm can be made to compute a basis of $M_{k,G}$ over $\QQ$ or $L_G$.
\item
Much of this section can be easily generalized to deal with odd weights $k\geq 3$.
\item
There are other methods for constructing a basis of $M_{k,G}$ whose $q$-expansions at each cusp we can compute arbitrarily many terms of.

In \cite{actionsoncuspsforms}, we explained how to compute an explicit basis of $M_k(\Gamma(N),\QQ(\zeta_N))$ and express the right action of $\SL_2(\ZZ)$ with respect to this basis (we did numerical computations and then identified the algebraic numbers that arose).  From this, we can then compute $M_{k,G}$.  We did not take this approach since $M_k(\Gamma(N),\QQ(\zeta_N))$ is often significantly large than $M_{k,G}$.   It should be possible to adapt the methods of \cite{actionsoncuspsforms} to be more appropriate, but we instead have used Eisenstein series since they were more algebraic in flavour.
\end{romanenum}
\end{remark}

\subsection{Explicit slash action}  \label{S:explicit slash}

Fix a modular form $f\in M_{k,G}$ with assumptions as in \S\ref{SS:basis of M_{k,G}}.    Now consider any matrix $B\in \GL_2(\QQ)$ with positive determinant.  We shall  explain how the $q$-expansion of $f$ at all the cusps allows us to find the Fourier expansion for $f|_k B$.   

It is clear how scalar matrices act, so we may assume that $B$ is in $M_2(\ZZ)$ and that the greatest common divisors of its entries is $1$.   There is a unique $1\leq j \leq r$ and a matrix $\gamma \in \Gamma$ such that $B\cdot  \infty = (\gamma A_j)\cdot \infty$.   Therefore, 
\[
B = \varepsilon \gamma A_j \left(\begin{smallmatrix}a & b \\0 & d\end{smallmatrix}\right)
\] 
for some $\varepsilon \in \{\pm 1\}$ and integers $a,b,d \in \ZZ$ with $a$ and $d$ positive and relatively prime.  Since $f|_k(-I)=(-1)^k f$ and $f|_k\gamma=f$, this implies that $f|_k B = \varepsilon^k (f|_k A_j)|_k \left(\begin{smallmatrix}a & b \\0 & d\end{smallmatrix}\right)$.  Using (\ref{E:general qexpansions}), we deduce that
\begin{align*}
(f|_k B)(\tau)&= \varepsilon^k \Big(\sum_{n=0}^\infty a_{n,j}(f) \cdot q_{w_j}^n\Big)\Big|_k \left(\begin{smallmatrix}a & b \\0 & d\end{smallmatrix}\right)
= \varepsilon^k (a/d)^{k/2} \sum_{n=0}^\infty a_{n,j}(f) \zeta_{dw_j}^b \cdot q_{dw_j}^{an}.
\end{align*}

\section{Computing modular curves} \label{S:computing modular forms}

Fix a positive integer $N$ and let $G$ be a subgroup of $\GL_2(\ZZ/N\ZZ)$ that satisfies $\det(G)=(\ZZ/N\ZZ)^\times$ and $-I\in G$.   In this section, using modular forms, we describe how to compute an explicit model of the curve $X_G$.

\subsection{Modular forms revisited} \label{SS:mf revisited}

The cusps of $\calX_{\Gamma_G}=X_G(\CC)$ are precisely the points lying over $\infty$ via $\pi_G$.   The elliptic points of $\calX_{\Gamma_G}$ of order $2$ and $3$ are the points $P\in \calX_{\Gamma}$ for which $\pi_G$ is unramified and $\pi_G(P)$ is $1728$ and $0$, respectively.  In particular, the cusps, elliptic points of order $2$, and elliptic points of order $3$ define subschemes of $X_G$.

Fix an even integer $k\geq 0$.    Let $D_k$ be the divisor of $ \calX_{\Gamma_G}=X_G(\CC)$ given by (\ref{E:divisor Dk}) with $\Gamma=\Gamma_G$.   Note that $D_k$ is also a divisor of $X_G$ defined over $\QQ$.    Define the invertible sheaf $\scrL_k := \Omega_{X_G}^{\otimes k/2}(D_k)$ on $X_G$.    Note that $\scrL_k$ gives rise to the invertible sheaf $\calL_k$ on $X_G(\CC)=\calX_{\Gamma_G}$ with notation as in \S\ref{SS:global sections} with $\Gamma=\Gamma_G$.   In particular, we have an inclusion $H^0(X_G, \scrL_k) \subseteq H^0(\calX_{\Gamma_G},\calL_k)$ which induces an isomorphism $H^0(X_G, \scrL_k) \otimes_\QQ \CC \to H^0(\calX_{\Gamma_G},\calL_k)$.  

Recall from \S\ref{SS:global sections}, we have an explicit isomorphism 
\[
\psi_k \colon  M_k(\Gamma_G) \xrightarrow{\sim} H^0(\calX_{\Gamma_G},\calL_k).
\]
Under the isomorphism $\psi_k$, we now show that $H^0(X_G, \scrL_k)$ corresponds to the $\QQ$-subspace $M_{k,G}$ of $M_k(\Gamma_G)$ from \S\ref{SS:MkG}.
 
\begin{lemma} \label{L:compatible Q structures} 
The map $\psi_k$ restricts to an isomorphism $M_{k,G} \xrightarrow{\sim} H^0(X_G,\scrL_k)$ of vector spaces over $\QQ$.
\end{lemma}
\begin{proof}
The isomorphism $\psi_k$ induces an isomorphism $M_{k,G}\otimes_\QQ \CC \xrightarrow{\sim} H^0(X_G,\scrL_k)\otimes_\QQ \CC$ of complex vector spaces.  Therefore, the $\QQ$-vector spaces $M_{k,G}$ and $H^0(X_G,\scrL_k)$ have the same dimension.  Since $\psi_k$ is an isomorphism, it thus suffices to prove that $\psi_k^{-1}(H^0(X_G,\scrL_k))\subseteq M_{k,G}$.

Take any differential form $\omega \in H^0(X_G, \scrL_k) \subseteq H^0(\calX_{\Gamma_G},\calL_k)$.   Choose an element $u\in \calF_N^G-\QQ$.    Since $\QQ(X_G)=\calF_N^G$, there is a unique $v\in \calF_N^G$ such that $\omega= v\, (du)^{k/2}$.   Via the quotient map $\calH \to \calX_{\Gamma_G}$, $\omega=v (du)^{k/2}$ pulls back to the differential form $v(\tau) u'(\tau)^{k/2} (d\tau)^{k/2}$ on $\calH$.  So $f:=\psi_k^{-1}(\omega) \in M_k(\Gamma_G)$ is given by $f(\tau)= (2\pi i)^{-k/2} v(\tau) u'(\tau)^{k/2}$.

Taking the derivative of the $q$-expansion $u= \sum_{n \in \ZZ} c_n(u) q_N^n$ gives $u'(\tau) = \sum_{n\in \ZZ} 2\pi i n/N\, c_n(u) q_N^n$.   Therefore, $f(\tau)= \big(\sum_{n\in \ZZ} c_n(v) q_N^n \big)\cdot  \big(\sum_{n\in \ZZ} n/N\, c_n(u) q_N^n \big)^{k/2}$.  In particular, we find that $f$ is an element of $M_k(\Gamma_G,\QQ(\zeta_N))$ since the $q$-expansions of both $u$ and $v$ have coefficients in $\QQ(\zeta_N)$.

Now take any $A\in G$.   Set $m:=\det(A) \in (\ZZ/N\ZZ)^\times$ and take any matrix $\gamma=\left(\begin{smallmatrix}a & b \\c & d\end{smallmatrix}\right) \in \SL_2(\ZZ)$ for which $A\equiv \gamma \left(\begin{smallmatrix}1 & 0 \\0 & m\end{smallmatrix}\right) \pmod{N}$.  We have $ u|_0 \gamma = \sigma_m^{-1}(u)$ and $v|_0 \gamma = \sigma_m^{-1}(v)$ since $u*A=u$ and $v*A=v$.    The equality $ u|_0 \gamma = \sigma_m^{-1}(u)$ is the same as
$u(\gamma \tau) = \sum_{n\in \ZZ} \sigma_m^{-1}(c_n(u)) q_N^n$ and taking derivatives of both sides gives $u'(\gamma \tau) (c\tau+d)^{-2} = 2\pi i \sum_{n\in \ZZ} n/N\, \sigma_m^{-1}(c_n(u)) q_N^n$.  Therefore,
\begin{align*}
f|_k \gamma & = (2\pi i)^{-k/2} \, v|_0 \gamma \, (u' |_2 \gamma)^{k/2} \\
&= (2\pi i)^{-k/2} \, \sigma_m^{-1}(v) \, \Big(2\pi i \sum_{n\in \ZZ} \tfrac{n}{N}\, \sigma_m^{-1}(c_n(u)) q_N^n\Big)^{k/2}\\
&= \sigma_m^{-1}\Big( \Big(\sum_{n\in \ZZ} c_n(v) q_N^n  \Big)   \Big(\sum_{n\in \ZZ} \tfrac{n}{N}\, c_n(u) q_N^n\Big)^{k/2}\Big) = \sigma_m^{-1}(f).\\
\end{align*}
Since $f |_k \gamma = \sigma_m^{-1}(f)$, we have $f*A=f$.   Since $A$ was an arbitrary element of $G$, this implies that $f=\psi_k^{-1}(\omega)$ lies in $M_k(\Gamma_G,\QQ(\zeta_N))^G=M_{k,G}$.  We have $\psi_k^{-1}(H^0(X_G,\scrL_k)) \subseteq M_{k,G}$ since $\omega$ was an arbitrary element of $H^0(X_G,\scrL_k)$.
\end{proof}

Combining the $\psi_k$, we obtain an isomorphism 
\[
\bigoplus_k M_{k,G} \xrightarrow{\sim} \bigoplus_k H^0(X_G,\scrL_k)
\]
of graded $\QQ$-algebras, where the sums are over even integers $k\geq0$.

\subsection{Galois action on the cusps}

Let $U$ be the group of upper triangular matrices in $\SL_2(\ZZ)$; it is generated by $-I$ and $\left(\begin{smallmatrix}1 & 1 \\0 & 1\end{smallmatrix}\right) $. Let $U_N \subseteq \GL_2(\ZZ/N\ZZ)$ be the image of $U$ modulo $N$.  Define the set of double cosets $\scrC_G:=G\backslash \GL_2(\ZZ/N\ZZ)/U_N$.   In this section, we explain how to identify $\scrC_G$ with the set of cusps of $\calX_{\Gamma_G}=X_G(\CC)$ and describe the Galois action on it.    

The map $\SL_2(\ZZ)\to \QQ\cup\{\infty\}= \PP^1(\QQ)$, $A\mapsto A\cdot \infty$ is surjective and induces a bijection $\iota\colon\SL_2(\ZZ)/U \to \PP^1(\QQ)$.   The map $\iota$ respects the natural $\SL_2(\ZZ)$-actions and hence gives a bijection between the set of double cosets $\Gamma_G\backslash \SL_2(\ZZ)/U$ and the set of $\Gamma_G$-orbits of $\PP^1(\QQ)$, i.e., the set of cusps of $\calX_{\Gamma_G}$.  With notation as in \S\ref{SS:qexpansions}, the matrices $A_1,\ldots, A_r$ are representatives of the double cosets $\Gamma_G\backslash \SL_2(\ZZ)/U$.  Using that the level of $\Gamma_G$ divides $N$ and $\det(G)=(\ZZ/N\ZZ)^\times$, we find that the map $\Gamma_G\backslash \SL_2(\ZZ)/U \to \scrC_G$ obtained from reduction modulo $N$ is also a bijection.   This allows us to identify $\scrC_G$ with the cusps of $\calX_{\Gamma_G}=X_G(\CC)$.

For an $m\in (\ZZ/N\ZZ)^\times$ and a cusp $P:=G \cdot A \cdot U_N \in \scrC_G \subseteq X_G(\CC)$, define $m\cdot P :=G \cdot A  \left(\begin{smallmatrix}1 & 0 \\0 & m\end{smallmatrix}\right) \cdot U_N$.  This is well-defined and gives an action of $(\ZZ/N\ZZ)^\times$ on $U_N$.

\begin{lemma} \label{L:cusp Galois}
The cusps of $X_G$ are all defined over $\QQ(\zeta_N)$.  For a cusp $P\in X_G(\QQ(\zeta_N))$ and an $m\in (\ZZ/N\ZZ)^\times$, we have $\sigma_m(P)=m\cdot P$.
\end{lemma}
\begin{proof}
Take any $m\in (\ZZ/N\ZZ)^\times$ and choose any automorphism $\sigma$ of $\CC$ for which $\sigma(\zeta_N)=\zeta_N^m$.   Take any cusp $P\in \scrC_G\subseteq X_G(\CC)$.  Fix matrices $A,A' \in \SL_2(\ZZ)$ for which $A\cdot \infty$ and $A'\cdot \infty$ are representatives of the cusps $P$ and $m\cdot P$, respectively.    There are $g\in G$ and $u\in U_N$ such that $A\left(\begin{smallmatrix}1 & 0 \\0 & m\end{smallmatrix}\right) \equiv g A' u \pmod{N}$.

Consider a rational function $f\in \QQ(X_G)$ that does not have a pole at any of the cusps.  Consider the $q$-expansion $f|_0 A =\sum_{n\in \ZZ} c_n q_N^n$.   We have $c_n=0$ for all $n<0$ since $f$ does not have a pole at $P$ and we have $c_0 = f(P)$.  We have
\[
(f|_0A')*u=f *(gA'u) =
f *\big(A\left(\begin{smallmatrix}1 & 0 \\0 & m\end{smallmatrix}\right)\big) 
= (f|_0 A)*\left(\begin{smallmatrix}1 & 0 \\0 & m\end{smallmatrix}\right) = \sum_{n\in \ZZ} \sigma_m(c_n) q_N^n
\]
and hence $f(m\cdot P)=\sigma_m(c_0) = \sigma(f(P))$.  So $f(m\cdot P)=f(\sigma(P))$ for all $f\in \QQ(X_G)$ that do not have poles at any cusps.   Therefore, $m\cdot P = \sigma(P)$.  Since $\sigma|_{\QQ(\zeta_N)}=\sigma_m$, it remains to show that $P$ lies in $X_G(\QQ(\zeta_N))$.

  Since $m\in (\ZZ/N\ZZ)^\times$ is arbitrary, we have shown that for an automorphism $\sigma$ of $\CC$, $\sigma(P)$ depends only on the restriction of $\sigma$ to $\QQ(\zeta_N)$.  Therefore, $P\in X_G(\QQ(\zeta_N))$.
\end{proof}

\subsection{Constructing a model of $X_G$} \label{L:morphisms of X_G}

Fix notation as in \S\ref{SS:mf revisited} with an even integer $k\geq 2$.   Let $P_1,\ldots, P_r$ be the cusps of $X_G(\CC)=\calX_{\Gamma_G}$.    

Take any divisor $E:=\sum_{i=1}^r e_i P_i$ of $X_G$ with integers $e_i \geq 0$ that is defined over $\QQ$.  Note that the divisor $E$ is defined over $\QQ$ if and only if the integer $e_i$ depends only on the Galois orbit of $P_i$ (such $E$ can be constructed using the explicit Galois action on cusps from Lemma~\ref{L:cusp Galois}).   Define the $\QQ$-vector space
\[
V:= \{ f\in M_{k,G} :  \nu_{P_i}(f)\geq e_i \text{ for all $1\leq i \leq r$}\}
\]
From \S\ref{SS:basis of M_{k,G}}, there is an algorithm to compute a basis of $M_{k,G}$ with enough terms of the $q$-expansions known in order to find an explicit basis of $V$.   Each modular form in the basis is given by its $q$-expansion at all the cusps of $X_G$ and we can compute arbitrarily many terms of each expansion.  

We now assume that $\dim_\QQ V \geq 2$.   Set $d:=\dim_\QQ V-1$ and denote our basis of $V$ by $f_0,\ldots, f_d$.  For each $0\leq i,j \leq d$, we have $f_j/f_i \in \calF_N^G=\QQ(X_G)$.  Let 
\[
\varphi\colon X_G \to \PP_\QQ^d
\]
be the morphism defined by $\varphi(P)=[f_0(P),\ldots,f_d(P)]$ for all but finitely many $P$.  For a fixed $0\leq i \leq d$ and all but finitely points $P$ of $X_G$, we have $\varphi(P)=[(f_0/f_i)(P),\ldots,(f_d/f_i)(P)]$.   Up to composition with an automorphism of $\PP^d_\QQ$, $\varphi$ does not depend on the choice of basis of $V$.

Set $C:=\varphi(X_G)\subseteq \PP^d_\QQ$; it is a curve since $d\geq 1$ and hence $f_0/f_1$ is non-constant.  Let $I(C)$ be the homogeneous ideal of $\QQ[x_0,\ldots,x_d]$ of $C\subseteq \PP^d_\QQ$.   We have the usual grading $I(C)=\bigoplus_{n\geq 0} I(C)_n$, where $I(C)_n$ consist of homogeneous polynomials of degree $n$.   Note that for a homogeneous polynomial $F \in \QQ[x_0,\ldots,x_d]$, the polynomial $F$ lies in $I(C)$ if and only if $F(f_0,\ldots, f_d)=0$.   \\

For a fixed integer $n\geq 0$, let us now explain how to compute a basis of the $\QQ$-vector space $I(C)_n$.   Let $\calM_n$ be the set of monic polynomials in $\QQ[x_0,\ldots, x_d]$ of degree $n$.  With notation as in \S\ref{SS:basis of M_{k,G}}, we have an injective $\QQ$-linear map $\varphi_{nk}\colon M_{nk,G}\hookrightarrow \QQ(\zeta_N)^{h}$ for an explicit integer $h\geq 1$.   For each $m\in \calM_n$, we can compute $v_m:=\varphi_{nk}(m(f_0,\ldots, f_d)) \in \QQ(\zeta_N)^{h}$ assuming we have computed enough terms of the $q$-expansions of the $f_i$.   We can then compute the $\QQ$-vector space $W:=\{ (c_m) \in \QQ^{\calM_n} : \sum_{m \in \calM_n} c_m v_m = 0\}$.   The injectivity of $\varphi_{nk}$ implies that the map $W\to I(C)_n$, $(c_m) \mapsto \sum_{m\in \calM_n} c_m m$ is an isomorphism of vector spaces over $\QQ$ and we have found $I(C)_n$.  In practice, to produce a ``nice'' basis of $W$, we apply the LLL algorithm to $W\cap \ZZ^{\calM_n}$.\\

 Define the invertible sheaf $\calF:= \scrL_k(-E)$ on $X_G$.   The isomorphism $\psi_k$ restricts to an isomorphism $V\xrightarrow{\sim} H^0(X_G,\calF)$ of $\QQ$-vector spaces; use Lemma~\ref{L:compatible Q structures} and the description of $\psi_k$ from \S\ref{SS:global sections}.  We have 
\begin{align} \label{E:deg F}
\deg \calF =  \deg \mathscr{L}_k - {\sum}_{i=1}^r e_i = k/2\cdot (2g-2) + k/2 \cdot r + \lfloor k/4 \rfloor \cdot v_2 +  \lfloor k/3 \rfloor \cdot v_3 - {\sum}_{i=1}^r e_i.
\end{align}
Using the Riemann--Roch theorem, we have $d=\deg \calF - g$ when $\deg \calF > 2g-2$.

\subsubsection{Large degree case} \label{SSS:large degree}
Assume that $\deg \calF \geq 2g+1$.  By the Riemann--Roch theorem, $\calF$ is very ample and hence $\varphi$ is an embedding giving an isomorphism between $X_G$ and $C$.     The homomorphism
\[
\QQ[x_0,\ldots, x_d]/I(C) \to \bigoplus_{n\geq 0} H^0(X_G,\calF^{\otimes n})
\]
of graded $\QQ$-algebras given by $x_i\mapsto \psi_k(f_i)$ is an isomorphism (the surjectivity follows from \cite[Theorem~6]{MR0282975} and our assumption $\deg \calF\geq 2g+1$).  The ideal $I(C) \subseteq \QQ[x_0,\ldots, x_d] $ is generated by $I(C)_2$ and $I(C)_3$, cf.~\cite{MR289516}.   If $\deg \calF \geq 2g+2$, then the homogenous ideal $I(C)$ is generated by just $I(C)_2$, cf.~\cite{MR289516}.    

Consider the special case where $\deg \calF=2g+1$.  When $g=0$, we have $C=\PP^1_\QQ$.   When $g=1$, the curve $C\subseteq \PP^2_\QQ$ is a plane cubic.

\begin{remark}
Suppose $\deg \calF \geq 2g+2$.    Consider a fixed positive integer $b$.  Suppose we have computed the $q$-expansions $f_j|_k A_i + O(q_N^b)$ for all $0\leq j \leq d$ and $1\leq i \leq r$.   From these expansions, we can find a basis of the $\QQ$-vector space $W:=\{F \in \QQ[x_0,\ldots, x_d]_2: F(f_0,\ldots, f_d) + O(q_N^b) = 0 + O(q_N^b)\}$.   Let $C'$ be the subvariety of $\PP^d_\QQ$ defined by a basis of $W$.   We have $C'\subseteq C$ since $I(C)_2 \subseteq W$ and $I(C)$ is generated by $I(C)_2$.   So if $C'$ is not zero dimensional, we have $C'=C$ and $I(C)_2=W$.   Of course, we will have $I(C)_2=W$ by taking $b$ sufficiently large.   The benefit of this approach is that we can often use a $b$ that is significantly smaller that the one arising from applying the Sturm bound.
\end{remark}

\subsubsection{Canonical map} \label{SSS:canonical map}
Consider the special case where $g\geq 3$, $k=2$, and $E=\sum_{i=1}^r P_i$.  We have 
\[
\calF:=\scrL_2(-E)=\Omega_{X_G}(D_2-E) =\Omega_{X_G}.
\]  
So $d=g-1$ and $\varphi\colon X_G\to \PP_\QQ^{g-1}$ is the \defi{canonical map}.   Define $C:=\varphi(X_G) \subseteq \PP^{g-1}_\QQ$.  We now recall some basic details, see \cite[\S7]{actionsoncuspsforms} for further details and some computation details.

First suppose that $X_G$ is (geometrically) hyperelliptic. 
Then the curve $C$ has genus $0$ and the morphism $\varphi\colon X_G\to C$ has degree $2$.   The ideal $I(C)$ is generated by $I_2(C)$ and $\dim_\QQ I_2(C)=(g-1)(g-2)/2$.

Suppose that $X_G$ is not hyperelliptic.  Then $\varphi$ is an embedding and in particular $C$ is isomorphic to $X_G$.  The dimension of $I_2(C)$ and $I_3(C)$ over $\QQ$ are $(g-2)(g-3)/2$ and $(g-3)(g^2+6g-10)/6$, respectively. If $g\geq 4$, the ideal $I(C)$ is generated by $I_2(C)$ and $I_3(C)$.   If $g=3$, then $I(C)$ is generated by $I_4(C)$ and $\dim_\QQ I_4(C)=1$.

We can compute $I_2(C)$ and its dimension over $\QQ$ determines whether or not $X_G$ is hyperelliptic.   Assume $X_G$ is not hyperelliptic.  Then by computing $I_3(C)$, and $I_4(C)$ when $g=3$, we can find equations for the curve $C\subseteq \PP_\QQ^1$.

\subsubsection{Computing a model for $X_G$} \label{L:computing a model}

If $g\geq 3$, we can first compute the image of the canonical map $\varphi\colon X_G \to \PP^{g-1}_\QQ$ and find equations defining the image $C:=\varphi(X_G)$, cf.~\S\ref{SSS:canonical map}.    If $C$ does not have genus $0$, i.e., $C$ is not hyperelliptic, then $C$ is a model of $X_G$.    If $C$ has genus $0$ and $C(\QQ)=\emptyset$, then $X_G(\QQ)=\emptyset$ (which in our application means we do not need to compute a model of $X_G$).

Now consider the general case.   Choose the smallest even integer $k\geq 2$ such that 
\[
k/2\cdot (2g-2) + k/2 \cdot r + \lfloor k/4 \rfloor \cdot v_2 +  \lfloor k/3 \rfloor \cdot v_3 \geq 2g+1.
\]
Such an integer $k$ exists since $g-1+v_2/4+v_3/3+r/2 = [\SL_2(\ZZ):\pm \Gamma_G]/12$ by \cite[Proposition~1.40]{MR1291394}.  Choose an effective divisor  $E:=\sum_{i=1}^r e_i P_i$ of $X_G$ defined over $\QQ$ so that the right hand side of (\ref{E:deg F}) is at least $2g+1$ and is as small as possible.   By our choices, we have $\deg \calF \geq 2g+1$.  One can then compute a model as in \S\ref{SSS:large degree}.

\subsubsection{Cusps of our model}
With $f_0,\ldots, f_d \in M_{k,G}$ defining the morphism
$\varphi\colon X_G \to C\subseteq \PP_\QQ^d$, we now describe the image of the cusps of $X_G$.   

Take any $1\leq j \leq r$.  From Lemma~\ref{L:cusp Galois}, we know that $\varphi(P_j)$ will lie in $C(\QQ(\zeta_N))$.  There is an $0\leq m \leq d$ such that for all $0\leq i \leq d$, the $q$-expansion of $f_i/f_m$ at the cusp $P_j$ is a power series for all $0\leq i \leq m$; denote its constant term by $c_i \in \QQ(\zeta_N)$.  Note that $f_i/f_m \in \QQ(X_G)$ is regular at $P_j$ and $(f_i/f_m)(P_j)=c_i$.   So we have  $\varphi(P_i) = [c_0,\ldots, c_d] \in \PP^d(\QQ(\zeta_N))$.

\subsection{Curves of genus $0$ and $1$} \label{SS:explicit low genus model}

Assume that $X_G$ has genus at most $1$.  Such curves are important in our application since they potentially could have infinitely many rational points.

Assume that we have found an explicit smooth projective model $C\subseteq \PP^n_\QQ$ for the modular curve $X_G$ as in \S\ref{L:morphisms of X_G}.    In particular, we have modular forms $f_0,\ldots, f_d$ in $M_{k,G}$ for some even $k\geq 2$ such that $C$ is defined by the homogeneous polynomials $F\in \QQ[x_0,\ldots, x_d]$ for which $F(f_0,\ldots, f_d)=0$.

\subsubsection{Finding a simple model} \label{SSS:Finding a simple model}
Suppose we have found a rational point $P\in C(\QQ)$.  See \S\ref{SSS:is there a rational point} for details on how we check if there is a rational point. 

If $X_G$ has genus $0$, then using the point $P$, we can compute an isomorphism $\psi\colon C \xrightarrow{\sim} \PP^1_\QQ$. Using the modular forms $f_i$ and $\psi$, we can then compute a modular function $f$ for which $\QQ(X_G)=\QQ(f)$.  Note that $f$ is given by its $q$-expansions at the cusps of $X_G$ and we can compute arbitrarily many terms of each expansion.

If $X_G$ has genus $1$, then using the point $P$, we can compute an isomorphism $\psi\colon C \xrightarrow{\sim} E$, where $E$ is an elliptic curve over $\QQ$ and $\psi(P)=0$. Using the modular forms $f_i$ and $\psi$, we can compute modular functions $x$ and $y$ for which $\QQ(X_G)=\QQ(x,y)$ and for which $x$ and $y$ satisfy a Weierstrass equation with rational coefficients defining $E$.  Note that $x$ and $y$ are given by their $q$-expansions at the cusps of $X_G$ and we can compute arbitrarily many terms of each expansion.

\subsubsection{Recognizing elements of the function field} \label{SSS:recognizing elements in function field}

Assume we have found a model for $X_G$ as in \S\ref{SSS:Finding a simple model}.  In particular, $\QQ(X_G)$ is $\QQ(f)$ or $\QQ(x,y)$ if $X_G$ has genus $0$ or $1$, respectively.

Now suppose we have a function $h\in \QQ(X_G)$ given by a $q$-expansion at each cusp of $X_G$ for which we can compute arbitrarily many terms.    Also suppose we also have an upper bound on the number of poles, with multiplicity, of $h$.   When $h$ is a quotient of two elements of $M_{k',G}$ for some even $k'$, then the number of poles of $h$ is bounded above by $k'/12\cdot [\SL_2(\ZZ): \pm \Gamma_G] = k'/12 \cdot [\GL_2(\ZZ/N\ZZ): G]$ (the divisor of a nonzero modular form in $M_{k',G}$, cf.~\cite[\S2.4]{MR1291394}, is effective and its degree can be computed using Propositions~2.16 and 1.40 of \cite{MR1291394}).

We want to express $h$ in terms of the generators of the field $\QQ(f)$ or $\QQ(x,y)$.

One approach is just to brute force search for an expression.  Consider the genus $0$ case.  For a fixed integer $d\geq 0$, one can look for a relation $h=F_1(f)/F_2(f)$ where $F_1$ and $F_2$ are polynomials in $\QQ[t]$ of degree at most $d$.   For each cusp of $X_G$, substituting our $q$-expansions in the expression $F_2(f) h -  F_1(f) =0$, we obtain a homogeneous system of linear equations over $\QQ(\zeta_N)$ whose unknowns are the coefficients of the polynomials $F_1$ and $F_2$.  For the minimal $d$ for which such a relation exists, these linear equations will have a $1$-dimension space of solutions over $\QQ(\zeta_N)$ and scaling will produce a solution in rationals unique up to scaling by a nonzero rational.    To rigorously verify that $F_2(f) h =   F_1(f)$, with specific coefficients, is $0$ it suffices to compute enough terms of the $q$-expansions at the cusps so that we have more zeros at the cusps, with multiplicity, than poles (we assumed we had a bound on the poles of $h$ and $f$ has a single simple pole).  We can increase $d\geq 0$ until we find a relation.

In the special case where there is a degree $1$ rational function $b\in \QQ(\zeta_N)(t)$ such that $b(h) \in \QQ(\zeta_N)(X_G)$ has all its poles at the cusps, we can take a more direct approach.    Since we know the $q$-expansions of $f$ at the cusps, we can find a partial fractions decomposition of $b(h)$ in $\QQ(\zeta_N)(f)$ and then find the desired expression for $h$ since $b$ has degree $1$.

Similar remarks hold when $X_G$ has genus $1$ except now we are looking for a relation $h=(F_1(x) + F_2(x)y)/F_3(x)$, where $F_1,F_3 \in \QQ[t]$ have degree at most $d$ and $F_2\in \QQ[t]$ has degree at most $d-1$.

\subsubsection{Constructing the map to the $j$-line} \label{SSS:map to jline}

We want to describe the natural morphism $\pi_G$ from $X_G$ to the $j$-line.   Since we are interested in rational points, we shall assume that $X_G$ has a rational point and that we have found a model for $X_G$ as in \S\ref{SSS:Finding a simple model}.   In particular, $\QQ(X_G)$ is $\QQ(f)$ or $\QQ(x,y)$ when $X_G$ has genus $0$ or $1$, respectively.   

For simplicity, suppose that $X_G$ has genus $0$ and hence $\QQ(X_G)=\QQ(f)$.   Since the modular $j$-invariant $j$ lies in $\QQ(X_G)$, the morphism $\pi_G$ is given by the unique $\pi(t) \in \QQ(t)$ for which $\pi(f)=j$.   We can find $\pi(t)$ using the methods from \S\ref{SSS:recognizing elements in function field} and use that $\pi(t)$ has degree $[\GL_2(\ZZ/N\ZZ):G]$.

In practice, it is more efficient to make use of intermediate fields lying between $\QQ(j)$ and $\QQ(f)$.  We can assume that $G\neq \GL_2(\ZZ/N\ZZ)$ since otherwise $X_G$ is the $j$-line.   Choose a group $G\subsetneq G_0 \subseteq \GL_2(\ZZ/N\ZZ)$ for which $[G_0:G]$ is minimal.   Then $\QQ(X_{G_0})=\QQ(f')$ and hence $f'=\varphi(f)$ for a unique rational function $\varphi(t)\in \QQ(t)$ of degree $[G_0:G]$.  We can then find $\varphi$ using the methods from \S\ref{SSS:recognizing elements in function field}.   This reduces the computation to the curve $X_{G_0}$ and we can continue in this manner until we get to the $j$-line.   The advantage of working with intermediate subfields is that in the method of \S\ref{SSS:recognizing elements in function field} we require less terms of the $q$-expansions.

Similar remarks hold when $X_G$ has genus $1$ except now $X_{G_0}$ will have genus $0$ or $1$.

\subsubsection{Determining whether there is a rational point} \label{SSS:is there a rational point}

We are interested in determining whether $X_G$, equivalently $C$, has a rational point.  We first check whether $X_G$ has a local obstruction to rational points. 

For real points this can be done without the model since we know that $X_G(\RR)\neq \emptyset$ if and only if $G$ contains an element that is conjugate in $\GL_2(\ZZ/N\ZZ)$ to $\left(\begin{smallmatrix}1 & 0 \\0 & -1\end{smallmatrix}\right)$ or $\left(\begin{smallmatrix}1 & 1 \\0 & -1\end{smallmatrix}\right)$, cf.~\cite[Proposition~3.5]{possibleindices}.

Using our explicit model, we can verify whether $C(\QQ_p)$ is empty for any given prime $p$.   If $X_G(\RR)=\emptyset$ or $C(\QQ_p)= \emptyset$ for some prime $p$, then we of course have $X_G(\QQ)=\emptyset$.

When $X_G$ has genus $0$ and $X_G(\RR)$ is nonempty, we know that we will either be able to find a rational point on $C$ or we will be able to find a prime $p$ such that $C(\QQ_p)=\emptyset$. 

\begin{remark}
There is a more general definition of $X_G$ as a coarse moduli space that would realize it as a smooth scheme over $\ZZ[1/N]$.  For a prime $p\nmid N$, the reduction $(X_G)_{\FF_p}$ would then be a smooth projective and geometrically irreducible curve of genus at most $1$ over $\FF_p$ and hence have an $\FF_p$-point.  Hensel's lemma would then implies that $X_G(\QQ_p)\neq \emptyset$.   So when looking for local obstructions, we limit ourselves to primes $p$ dividing $N$.
\end{remark}

Now suppose that $X_G$ is genus $1$ for which we have not found a rational point and we have found no local obstructions.   The Jacobian of $C$ is an elliptic curve $E$ over $\QQ$.   Using our embedding $C\subseteq \PP^d_\QQ$, we find that $C$ is a principal homogeneous space of $E$ that has order $n:=d+1$ in the Weil--Ch\^atelet group of $E$.

When $n\leq 5$, there are explicit equations for $E$ and for a covering map $\varphi\colon C \to E$ of degree $n^2$ (when base extended to a field where $C$ and $E$ are isomorphic, it corresponds to multiplication by $n$ on $E$).   When $n\leq 4$ or $n=5$, see \cite{MR1858080} and \cite{MR2448246}, respectively.   Note that these constructions have been implemented in \texttt{Magma}.  If $C$ has a rational point, then we find that $\varphi(C(\QQ))$ is a coset of $nE(\QQ)$ in $E(\QQ)$.  We can compute the weak Mordell--Weil group $E(\QQ)/nE(\QQ)$.  For points $P\in E(\QQ)$ that represent the elements of the finite group $E(\QQ)/nE(\QQ)$, we can check whether the fiber $\varphi^{-1}(P) \subseteq C$ has any rational points.  If not, then we have verified that $C(\QQ)$ is empty (otherwise, we have found a point).\\

In our application, the above determines whether $X_G$ has a rational point for all but a few cases we consider.   In the cases where this does not apply (i.e., $n>5$), there was always a group $G\subseteq G_0 \subseteq \GL_2(\ZZ/N\ZZ)$ for which $[G_0:G]=2$ and $X_{G_0} \cong \PP^1_\QQ$.  In these cases, we could find an alternate model $C'$ of $X_G$ given as a hyperelliptic curve over $\QQ$.  The above reasoning then applies to check whether there are rational points; the relevant morphism $\varphi\colon C'\to E$ now has degree $2^2$.

\subsection{Some alternate models for curves} \label{SS:alternate models}

Assume that $X_G$ has genus at least $2$ and choose a group $G\subsetneq G_0 \subseteq \GL_2(\ZZ/N\ZZ)$ for which $X_{G_0}$ has genus at most $1$.   In practice, one chooses $G_0$ with $[G_0:G]$ minimal.   We shall assume that $X_{G_0}$ has a rational point; we are interested in knowing the rational points of $X_{G}$ and there are none if $X_{G_0}(\QQ)=\emptyset$. Assume we have found a model as in \S\ref{SSS:Finding a simple model} with $\QQ(X_{G_0})$ of the form $\QQ(f)$ or $\QQ(x,y)$.   

\subsubsection{Minimal polynomial} \label{SSS:Minimal polynomial}
 
Let $S$ be a set of matrices in $\SL_2(\ZZ)$ whose reductions modulo $N$ represent the cosets $G\backslash G_0$.  

Consider a modular function $h\in \QQ(X_{G})$ for which we have computed arbitrarily many terms of the $q$-expansion at each of the cusps of $X_{G}$.  Using these expansions and \S\ref{S:explicit slash}, we can compute the $q$-expansions of $h|_k A$ for all $A\in S$.    Assume that $h|_k A \neq h$ for all $A\in S$ and hence $\QQ(X_{G_0})(h)=\QQ(X_{G})$.   To construct a suitable element $h \in \QQ(X_{G})$, we can look at the quotient of nonzero elements in $M_{k,G}$ for  even $k\geq 2$; there will be such modular forms for $k$ large enough by \S\ref{L:computing a model}. 

Define the polynomial
\[
P(t):=\prod_{A\in S}(t - h|_k A).
\]
By our choice of $S$, $P(t)$ is a polynomial of degree $[G_0:G]$ with coefficients in $\QQ(X_{G_0})$ that satisfies $P(h)=0$.    The coefficients of $P(t)$ line in $\QQ(X_{G_0})$ and can be given explicitly in $\QQ(f)$ or $\QQ(x,y)$ by the method described in \S\ref{SSS:recognizing elements in function field} (the methods of \S\ref{S:explicit slash} can be used to take the $q$-expansions, which are in terms of the cusps of $X_{G}$, to those at cusps of $X_{G_0}$).

Since $\QQ(X_{G_0})(h)=\QQ(X_{G})$, the polynomial $P(t)$ gives rise to a possibly singular model of the curve $X_{G}$ in $\PP^2_\QQ$ or $\PP^3_\QQ$ with a rational map to $X_{G_0}$ corresponding to the natural morphism $X_{G}\to X_{G_0}$.  In many situations, a singular model of $X_{G}$ will be preferably to a model that lies in some large dimensional ambient space.   An important special case is when $[G_0:G]=2$ and $X_{G_0}\cong \PP^1_\QQ$ since then we can use the singular model to find a smooth model of $X_{G}$ as a hyperelliptic curve.

\subsubsection{Serre type} \label{SSS:Serre type}

Now suppose that $N=N_1 N_2$ with $N_1>1$ a power of $2$ and $N_2$ odd.   Set $G_1=\GL_2(\ZZ/N_1\ZZ)$.   Suppose further that there is a subgroup $G_2$ of $\GL_2(\ZZ/N_2\ZZ)$ such that $G$ is an index $2$ subgroup of $G_0:=G_1\times G_2 \subseteq \GL_2(\ZZ/N\ZZ)$ so that the projective map $G\to G_i$ is surjective for $i\in \{1,2\}$.   

The kernel of the projection maps $G\to G_1$ and $G\to G_2$ are of the form $\{I\}\times H_2$ and $H_1\times \{I\}$, where $H_i$ is an index $2$ subgroup of $G_i$.  Note that $H_1\times H_2 \subseteq G \subseteq G_1\times G_2$.

Take any $i \in \{1,2\}$ and suppose that $\det(H_i)=(\ZZ/N_i\ZZ)^\times$.  As in \S\ref{SSS:Minimal polynomial}, we can find a $h\in \QQ(X_{H_i})$ such that $\QQ(X_{H_i})=\QQ(X_{G_0})(h)$ and compute an irreducible polynomial of degree $2$ in $\QQ(X_{G_0})[x]$ with root $h$.   By taking the discriminant of the polynomial, we obtain an element $c_i \in \QQ(X_{G_0})$ that is a square in $\QQ(X_{H_i})$ but not in $\QQ(X_{G_0})$.   Note that $c_i$ is unique up to multiplication by a nonzero square in $\QQ(X_{G_0})$.  

Take any $i \in \{1,2\}$ and suppose that $\det(H_i)\neq (\ZZ/N_i\ZZ)^\times$.  The group $\det(H_i)$ is an index $2$ subgroup of $(\ZZ/N_i\ZZ)^\times$.  Let $K_i \subseteq \Qbar$ be the corresponding quadratic extension of $\QQ$, i.e., the quadratic extension for which the image of  $\chi_\cyc(\Gal_{K_i})\subseteq \Zhat^\times$ modulo $N_i$ is the group $\det(H_i)$.   Let $c_i$ be the unique squarefree integer for which $K_i=\QQ(\sqrt{c_i})$.

When $i=1$, a computation shows that $c_i$ can always be chosen to lie in the set $\{-1,\pm 2, \pm (j-1728), \pm 2 (j-1728)\}$.  

Putting everything together, we find that there is a $y \in \QQ(X_G)$ such that $y^2=c_1 c_2$ and $\QQ(X_G)=\QQ(X_{G_0})(y)$.   Therefore, $y^2=c_1 c_2$ defines a singular model of $X_G$.   In the special case where $X_{G_0} \cong \PP^1_\QQ$, this gives rise to a hyperelliptic model.

\begin{remark}
Though this might seem like a very niche case, groups with the above conditions arise frequently in our application to Serre's open image theorem and are often the slowest to deal with using the strategy from \S\ref{SSS:Minimal polynomial}.
\end{remark}

\section{A generic family and modular functions} \label{S:specializations}

Let $\calE$ be the elliptic curve over $\QQ(j)$ defined by the Weierstrass equation 
\begin{align} 
\label{E:generic Weierstrass}
y^2 =  x^3 -  27 \cdot j (j-1728)  \cdot  x +54 \cdot j (j-1728)^2,
\end{align}
where $j$ is the modular $j$-invariant.   The elliptic curve $\calE$ has $j$-invariant $j$.

Fix an integer $N\geq 3$ and a nonzero modular form $f_0\in M_3(\Gamma(N),\QQ(\zeta_N))$.   Let $\calF_N$ be the field of modular functions of level $N$ whose $q$-expansions have coefficients in $\QQ(\zeta_N)$, cf.~\S\ref{SS:modular functions}.   We have $f_0^2/E_6 \in \calF_N$, where $E_6$ is the usual Eisenstein series of weight $6$ as given by (\ref{E:usual E4 and E6}).  In a field extension of $\calF_N$, choose a $\beta$ satisfying \[
\beta^2 = j\cdot {f_0^2}/{E_6}.
\] 
Note that $\beta$ need not be a modular function.   

In \S\ref{SS:N torsion of calE}, we will show that $\calF_N(\beta)$ is a minimal field extension of $\QQ(j)$ for which all of the $N$-torsion points of $\calE$ are defined.    Let $\Gal_{\QQ(j)}$ be the absolute Galois group of $\QQ(j)$ for which the implicit algebraic closure contains $\calF_N(\beta)$.   With respect to a suitable basis of $\calE[N]$, we will show that there is a surjective Galois representation
\[
\rho_{\calE,N}^* \colon \Gal_{\QQ(j)} \to \GL_2(\ZZ/N\ZZ)
\]
that satisfies
\[
\sigma(f)= f * \rho_{\calE,N}^*(\sigma)^{-1}
\]
for all $\sigma\in \Gal_{\QQ(j)}$ and $f\in \calF_N$, where the $*$ action is described in \S\ref{SS:modular functions}.  For details, see \S\ref{SS:representations modular explicit}.   In particular, $\rho_{\calE,N}^*$ induces isomorphisms $\Gal(\calF_N(\beta)/\QQ(j)) \xrightarrow{\sim} \GL_2(\ZZ/N\ZZ)$ and $\Gal(\calF_N/\QQ(j)) \xrightarrow{\sim} \GL_2(\ZZ/N\ZZ)/\{\pm I\}$.

\subsection{$N$-torsion points of $\calE$} \label{SS:N torsion of calE}

Define the usual Eisenstein series 
\begin{align} \label{E:usual E4 and E6}
E_4(\tau)= 1+240 \sum_{n=1}^\infty n^3 q^n/(1-q^n)\quad \text{ and }\quad 
E_6(\tau)= 1-504 \sum_{n=1}^\infty n^5 q^n/(1-q^n);
\end{align}
they are modular forms on $\SL_2(\ZZ)$ of weight $4$ and $6$, respectively.    Define the weight $12$ modular form $\Delta(\tau)=q\prod_{n=1}^\infty (1-q^n)^{24}$ on $\SL_2(\ZZ)$.   We have relations $E_4^3-E_6^2=1728\Delta$, $j=E_4^3/\Delta$ and $j-1728=E_6^2/\Delta$.  For $\tau\in \calH$, let $\wp(z;\tau)$ be the Weierstrass elliptic function for the lattice $\ZZ\tau+\ZZ \subseteq \CC$.   Let $\wp'(z; \tau)$ be the derivative of $\wp(z;\tau)$ with respect to $z$.   
 
Take any nonzero $\alpha\in (\ZZ/N\ZZ)^2$ and choose integers $r$ and $s$ such that $\alpha$ is congruent to $(r,s)$ modulo $N$.   Let $x_{\alpha}$ and $u_{\alpha}$ be the function of the upper half-plane defined by 
\begin{align*}
x_{\alpha}(\tau)&:= 36\, \frac{E_4(\tau) E_6(\tau)}{\Delta(\tau)} \cdot (2\pi i)^{-2}\wp(\tfrac{r}{N}  \tau + \tfrac{s}{N}; \tau)\\
u_{\alpha}(\tau)&:= 108 \,\frac{ E_6(\tau)^2}{f_0(\tau) \Delta(\tau)} \cdot (2\pi i)^{-3} \wp'(\tfrac{r}{N} \cdot \tau + \tfrac{s}{N}; \tau).
\end{align*}
These definitions do not depend on the choice of $r$ and $s$ since $\wp(z;\tau)$ and $\wp'(z;\tau)$ are unchanged if we replace $z$ by $z+\omega$ with $\omega \in \ZZ\tau + \ZZ$.

\begin{lemma} \label{L:action on x and u}
\begin{romanenum}
\item \label{L:action on x and u i}
For any nonzero $\alpha\in (\ZZ/N\ZZ)^2$, $x_\alpha$ and $u_\alpha$ are elements of $\calF_N$.
\item \label{L:action on x and u ii}
The field $\calF_N$ is the extension of $\QQ(j)$ generated by $x_\alpha$ with nonzero $\alpha\in (\ZZ/N\ZZ)^2$.
\item \label{L:action on x and u iii}
Take any nonzero $\alpha\in (\ZZ/N\ZZ)^2$ and $A\in \GL_2(\ZZ/N\ZZ)$.  We have 
\[
x_\alpha *A = x_{\alpha A} \quad \text{ and }\quad u_\alpha *A = \tfrac{f_0}{f_0*A} \,u_{\alpha A}.
\]
In particular, $u_\alpha * A =  u_{\alpha A}$ if $f_0 * A = f_0$.
\end{romanenum}
\end{lemma}
\begin{proof}
Take any nonzero $\alpha\in (\ZZ/N\ZZ)^2$.  Fix integers $0\leq r < N$ and $s$ so that $(r,s)$ is congruent to $\alpha$ modulo $N$.    Our function $x_\alpha$ agrees with the function $f_{(r/N,s/N)}$ from \S6.1 of \cite{MR1291394} up to multiplication by a nonzero rational number.   Part (\ref{L:action on x and u ii}) follows from Proposition~6.9(1) of \cite{MR1291394}; this implies (\ref{L:action on x and u i}) for the function $x_\alpha$.\\

Let $h_\alpha$ be the function of the upper half-plane defined by $(2\pi i)^{-3} \wp'(\tfrac{r}{N} \cdot \tau + \tfrac{s}{N}; \tau)$.    We now explain why $h_\alpha$ lies in $M_3(\Gamma(N),\QQ(\zeta_N))$.   Recall that for $\tau \in \calH$, we have the notation $q=e^{2\pi i \tau}$ and $q_N=e^{2\pi i \tau/N}$.   We have
\[
(2\pi i)^{-2} \wp(u;\tau)=\frac{1}{12} -2 \sum_{m=1}^\infty \sum_{n=1}^\infty n q^{mn} + \sum_{n=1}^\infty n e^{2\pi i n u} 
+\sum_{m=1}^\infty\sum_{n=1}^\infty n \Big( e^{2\pi i n u} q^{mn} + e^{-2\pi i n u} q^{mn}\Big);
\]
see the proof of Proposition~6.9 on p.~140 in \cite{MR1291394} (with $\omega_1=\tau$, $\omega_2=1$, and $v$ in loc.~cit.~should be defined as $u/\omega_2$).  Therefore, 
\[
(2\pi i)^{-3} \wp'(u;\tau)= \sum_{n=1}^\infty n^2 e^{2\pi i n u} 
+\sum_{m=1}^\infty\sum_{n=1}^\infty n^2 \Big( e^{2\pi i n u} q^{mn} - e^{-2\pi i n u} q^{mn}\Big).
\]
 With $u=r/N\cdot \tau + s/N$, we have $e^{2\pi i u} = \zeta_N^s q_N^r$ and hence
\begin{align} \label{E:expansion h alpha}
h_\alpha(\tau)= \sum_{n=1}^\infty n^2 \zeta_N^{sn} q_N^{rn} 
+\sum_{m=1}^\infty\sum_{n=1}^\infty n^2 \Big( \zeta_N^{sn} q_N^{rn} q^{mn} - \zeta_N^{-sn} q_N^{-rn} q^{mn}\Big).
\end{align}
Since $0\leq r < N$, this shows that $h_\alpha$ has a $q$-expansion that is a power series in $q_N$ with coefficients in $\QQ(\zeta_N)$.

Take any matrix $\gamma =\left(\begin{smallmatrix}a & b \\c & d\end{smallmatrix}\right) \in \SL_2(\ZZ)$. Since $\wp'(u;\tau)=  \sum_{\lambda\in \ZZ\tau+\ZZ} -2/(u+\lambda)^3$, we have
\begin{align*}
(h_\alpha |_3 \gamma)(\tau)&= (c\tau+d)^{-3} \cdot (2\pi i)^{-3} \sum_{\lambda \in \ZZ\gamma\tau + \ZZ} \frac{-2}{(\tfrac{r}{N}\cdot \gamma\tau + \tfrac{s}{N} + \lambda)^3}\\
&=(2\pi i)^{-3} \sum_{\lambda \in \ZZ(a\tau+b) + \ZZ(c\tau+d)} \frac{-2}{(\tfrac{r}{N}\cdot (a\tau+b) + \tfrac{s}{N} (c\tau+d) + \lambda)^3}\\
&=(2\pi i)^{-3} \sum_{\lambda \in \ZZ\tau+\ZZ} \frac{-2}{(\tfrac{ra+sc}{N} + \tfrac{rb+sd}{N} + \lambda)^3}.
\end{align*}
So $h_\alpha |_3 \gamma = h_{\alpha \gamma}$ for all $\gamma \in \SL_2(\ZZ)$.   In particular, $h_\alpha|_3 \gamma = h_\alpha$ for all $\gamma\in \Gamma(N)$.   We conclude that $h_\alpha$ is an element of $M_3(\Gamma(N),\QQ(\zeta_N))$ since the $q$-expansions of $h_\alpha |_3 \gamma = h_{\alpha \gamma}$ is a power series in $q_N$ with coefficients in $\QQ(\zeta_N)$ for all $\gamma \in \SL_2(\ZZ)$.  

For any $A\in \GL_2(\ZZ/N\ZZ)$, we claim that $h_\alpha * A = h_{\alpha A}$.  From our above proof, this holds when $A\in \SL_2(\ZZ/N\ZZ)$, so we need only prove the claim in the case where $A=\left(\begin{smallmatrix}1 & 0 \\0 & d\end{smallmatrix}\right)$.    Applying $\sigma_d$ to the coefficients of the expansion (\ref{E:expansion h alpha}) gives the expansion where we replace $s$ by $ds$.   Therefore, $h_\alpha*A = h_{\alpha A}$ as expected.

That $u_\alpha$ lies in $\calF_N$ is clear since it is the quotient of modular forms on $\Gamma(N)$ with coefficients in $\QQ(\zeta_N)$ and it has weight $2\cdot 6-3-12 +3 =0$.  Take any $A\in \GL_2(\ZZ/N\ZZ)$.   Since $E_4$ and $E_6$ are modular forms on $\SL_2(\ZZ)$ with coefficients in $\QQ$, we find that 
\[
u_\alpha * A = 108\, \frac{E_6^2}{f_0*A \cdot \, \Delta} \cdot h_\alpha *A = \frac{f_0}{f_0*A} \cdot 108\, \frac{E_6^2}{f_0 \Delta} \cdot h_{\alpha A} =  \frac{f_0}{f_0*A} \cdot u_{\alpha A}.
\]
We have proved the parts of the lemma concerning the functions $u_\alpha$. 

Now take any nonzero $\alpha\in (\ZZ/N\ZZ)^2$.  It remains to prove that $x_\alpha * A=x_{\alpha A}$ for all $A\in \GL_2(\ZZ/N\ZZ)$.   For $A\in \SL_2(\ZZ/N\ZZ)$, this follows from equation (6.1.3) of \cite{MR1291394}.   It remains to prove it for $A=\left(\begin{smallmatrix}1 & 0 \\0 & d\end{smallmatrix}\right)$; this follows from the $q$-expansion of $x_\alpha$, see equation (6.2.1) of \cite{MR1291394}.
\end{proof}

Recall that we fixed a nonzero modular form $f_0$ in $M_3(\Gamma(N),\QQ(\zeta_N))$ and chose a $\beta$ satisfying $\beta^2 = j\cdot {f_0^2}/{E_6} \in \calF_N$.  For each nonzero $\alpha\in (\ZZ/N\ZZ)^2$, define $y_\alpha:= \beta \cdot u_\alpha \in \calF_N(\beta)$ and the pair $P_\alpha:=(x_\alpha,y_\alpha)$.

\begin{lemma} \label{L:torsion of calE}
\begin{romanenum}
\item \label{L:torsion of calE i}
For every nonzero $\alpha\in (\ZZ/N\ZZ)^2$, $P_\alpha$ is a $N$-torsion point in $\calE(\calF_N(\beta))$.
\item \label{L:torsion of calE ii}
We have a group isomorphism 
\[
\iota_N\colon (\ZZ/N\ZZ)^2 \xrightarrow{\sim} \calE[N], \quad \alpha\mapsto P_\alpha,
\] 
where $P_{(0,0)}$ is defined as the identity of $\calE$.
\end{romanenum}
\end{lemma}
\begin{proof}
Fix any $\tau\in \calH$ satisfying $j(\tau) \notin \{0,1728\}$ and $f_0(\tau)\neq 0$.  Let $\scrE_\tau$ be the elliptic curve over $\CC$ defined by the Weierstrass equation $y^2=4x^3-g_2(\tau)x-g_3(\tau)$, where $g_2$ and $g_3$ are modular invariants.  Note that  $g_2/(2\pi i)^4 = E_4/12$ and $g_3/(2\pi i)^6=-E_6/216$.    Let $\scrE_\tau'$ be the elliptic curve over $\CC$ defined by the equation
\begin{align} \label{E:twist fixed tau}
j(\tau)f_0(\tau)^2/E_6(\tau)\cdot y^2 =  x^3 -  27 \cdot j(\tau) (j(\tau)-1728)  \cdot  x +54 \cdot j(\tau) (j(\tau)-1728)^2.
\end{align}

Define $c:=36 (2\pi i)^{-2} E_4(\tau) E_6(\tau)/\Delta(\tau)$ and $d:= 108 (2\pi i)^{-3} E_6(\tau)/(\Delta(\tau) f_0(\tau))$.  By using $j=E_4^3/\Delta$ and $j-1728=E_6^2/\Delta$, and dividing (\ref{E:twist fixed tau}) by $c^3/4$, we find that
\begin{align*}
\frac{4}{c^3} \cdot \frac{j(\tau)f_0(\tau)^2}{E_6(\tau)} \cdot y^2 &= 4 (x/c)^3 - 27 \frac{4 E_4(\tau)^3E_6(\tau)^2}{c^2\Delta(\tau)^2} (x/c) + 54 \frac{4E_4(\tau)^3 E_6(\tau)^4}{c^3\Delta(\tau)^3}\\
&= 4 (x/c)^3 - (2\pi i)^4 E_4(\tau)/12\cdot (x/c) + (2\pi i)^6 E_6(\tau)/216\\
&= 4 (x/c)^3 - g_2(\tau) \cdot(x/c) -g_3(\tau).
\end{align*}
Observing that
\[
\frac{4}{c^3} \cdot \frac{j(\tau)f_0(\tau)^2}{E_6(\tau)} = \frac{4}{c^3} \cdot \frac{E_4(\tau)^3f_0(\tau)^2}{\Delta(\tau) E_6(\tau)} = \Big(\frac{(2\pi i)^3 \Delta(\tau) f_0(\tau)}{108 E_6(\tau)}\Big)^2 = \frac{1}{d}^2,
\]
we deduce that the map $(x,y)\mapsto (x/c, y/d)$ defines an isomorphism $\scrE'_\tau\to \scrE_\tau$ of elliptic curves over $\CC$.
Using this isomorphism and properties of the Weierstrass function, we obtain an isomorphism $\CC/(\ZZ\tau +\ZZ) \xrightarrow{\sim} \scrE'_\tau(\CC)$ of complex Lie groups which maps a nonidentity element $z+(\ZZ\tau+\ZZ)$ to $(c\wp(z; \tau),d\wp'(z;\tau))$.  In particular, we have a group isomorphism $\ZZ^2/N\ZZ^2  \xrightarrow{\sim} \scrE_\tau'[N]$ which away from the identity is defined by 
\begin{align} \label{E:proof isomorphism N torsion}
(r,s) +N^2\ZZ \mapsto \big(c\wp(\tfrac{r}{N}\tau + \tfrac{s}{N}; \tau),d\wp'(\tfrac{r}{N}\tau + \tfrac{s}{N};\tau)\big) = (x_\alpha(\tau),u_\alpha(t)),
\end{align}
where $\alpha$ is the image of $(r,s)$ modulo $N$.

We now view $\tau$ as a variable again.  Letting $\scrE'$ be the elliptic curve defined by (\ref{E:twist fixed tau}), say over the field of meromorphic functions of $\calH$, we find that (\ref{E:proof isomorphism N torsion}) also defines an isomorphism $\ZZ^2/N\ZZ^2 \to \scrE'[N]$ (the points are torsion and satisfy the expected property since they hold for all but finitely many specializations of $\tau$).

Now by our choice of $\beta$, we deduce that the points $P_\alpha=(x_\alpha,y_\alpha)$ lie in $\calE(\calF_N(\beta))$ and  that $(\ZZ/N\ZZ)^2 \to \calE[N]$, $\alpha\mapsto P_\alpha$ is an isomorphism.  Note that $\beta$ is not a function of the upper half-plane, so we avoided using it when we fixed specific values of $\tau$.
\end{proof}

\subsection{Galois representations of $\calE$} \label{SS:representations modular explicit}

Fix an integer $N \geq 3$.  Fix notation as in \S\ref{SS:N torsion of calE}; in particular, we have a basis $P_{(1,0)}$ and  $P_{(0,1)}$ of the $\ZZ/N\ZZ$-module $\calE[N] \subseteq \calE(\calF_N(\beta))$.   With respect to the basis $\{P_{(1,0)},P_{(0,1)}\}$, let
\[
\rho_{\calE,N}\colon \Gal_{\QQ(j)} \to \GL_2(\ZZ/N\ZZ)
\]
be the associated Galois representation, where the implicit algebraic closure of $\QQ(j)$ used for the absolute Galois group contains $\calF_N(\beta)$.    So for $\alpha\in (\ZZ/N\ZZ)^2$ and $\sigma\in \Gal_{\QQ(j)}$, we have $\sigma(P_\alpha) = P_{\alpha A}$ where $A:= \rho_{\scrE,N}(\sigma)^t = \rho_{\scrE,N}^*(\sigma)^{-1}$.

\begin{lemma} \label{L:FN revisited}
\begin{romanenum}
\item \label{L:FN revisited i}
We have $\QQ(j)(\calE[N])=\calF_N(\beta)$.

\item \label{L:FN revisited ii}
We have $\rho^*_{\calE,N}(\Gal_{\QQ(j)})=\GL_2(\ZZ/N\ZZ)$ and $\rho^*_{\calE,N}(\Gal_{\Qbar(j)})=\SL_2(\ZZ/N\ZZ)$.

\item \label{L:FN revisited iii}
For $\sigma\in \Gal_{\QQ(j)}$ and $f\in \calF_N$, we have 
\[
\sigma(f) = f * \rho_{\calE,N}^*(\sigma)^{-1}.
\]
\item \label{L:FN revisited iv}
Composing $\rho_{\calE,N}^*$ with the quotient map to $\GL_2(\ZZ/N\ZZ)/\{\pm I\}$ induces an isomorphism 
\[
\Gal(\calF_N/\QQ(j))\xrightarrow{\sim} \GL_2(\ZZ/N\ZZ)/\{\pm I\}.
\]
\item  \label{L:FN revisited v}
Let $G$ be a subgroup of $\GL_2(\ZZ/N\ZZ)$ that satisfies $\det(G)=(\ZZ/N\ZZ)^\times$ and $-I \in G$.  Then $\rho^*_{\calE,N}(\Gal_{\QQ(X_G)}) = G$.

\item  \label{L:FN revisited vi}
For $\sigma\in \Gal_{\QQ(j)}$, we have $\sigma(\beta) = \tfrac{f_0*A}{f_0} \cdot \beta$, where $A:=\rho_{\calE,N}^*(\sigma)^{-1}$.  
\end{romanenum}
\end{lemma}
\begin{proof}
Part (\ref{L:FN revisited i}) follows from Lemma~\ref{L:action on x and u} and the definition of our points $P_\alpha$.  Take any nonzero $\alpha\in (\ZZ/N\ZZ)^2$ and $\sigma \in \Gal_{\QQ(j)}$.   Define $A:=\rho_{\scrE,N}^*(\sigma)^{-1}=\rho_{\scrE,N}(\sigma)^t$.   By our choice of basis in defining $\rho_{\scrE,N}$, we have $\sigma(P_\alpha)=P_{\alpha A}$.  We then have $\sigma(x_\alpha)=x_{\alpha A}$ by considering $x$-coordinates.  By Lemma~\ref{L:action on x and u}(\ref{L:action on x and u iii}), we have $\sigma(x_\alpha)= x_\alpha * A$.  Since $\alpha$ was an arbitrary nonzero element of $(\ZZ/N\ZZ)^2$, Lemma~\ref{L:action on x and u}(\ref{L:action on x and u ii}) implies that $\sigma(f) = f * A$ for all $f\in \calF_N$.  We have thus proved (\ref{L:FN revisited iii}).   

Using $\sigma(P_\alpha)=P_{\alpha A}$ and taking $y$-coordinates, gives $\sigma(\beta) \sigma(u_\alpha) = \beta u_{\alpha A}$.  Therefore, we have
\[
\sigma(\beta)\cdot  u_\alpha*A =\sigma(\beta) \sigma(u_\alpha) = \beta u_{\alpha A}=\beta \cdot \tfrac{f_0 *A}{f} \cdot u_\alpha *A,
\]
where we have used part (\ref{L:FN revisited iii}) and Lemma~\ref{L:action on x and u}(\ref{L:action on x and u iii}).  Part (\ref{L:FN revisited vi}) follows by cancelling $u_\alpha *A$ from both sides.

By (\ref{L:FN revisited iii}),  composing $\rho_{\calE,N}^*$ with the obvious quotient map gives an injective homomorphism $\Gal(\calF_N/\QQ(j)) \hookrightarrow \GL_2(\ZZ/N\ZZ)/\{\pm I\}$; it is an isomorphism since we know the degree of the extension $\calF_N/\QQ(j)$, cf.~Lemma~\ref{L:FN basics}(\ref{L:FN basics i}).    Therefore, $\pm\rho^*_{\calE,N}(\Gal_{\QQ(j)})=\GL_2(\ZZ/N\ZZ)$ and $\pm\rho^*_{\calE,N}(\Gal_{\Qbar(j)})=\SL_2(\ZZ/N\ZZ)$ by Lemma~\ref{L:FN basics}.   We have $\rho^*_{\calE,N}(\Gal_{\Qbar(j)})=\SL_2(\ZZ/N\ZZ)$ since there is no proper subgroup $H$ of $\SL_2(\ZZ/N\ZZ)$ for which $\pm H = \SL_2(\ZZ/N\ZZ)$, cf.~Lemma~\ref{L:pm H eq SL2}.  Part (\ref{L:FN revisited ii}) now follows.

Finally take $G$ as in (\ref{L:FN revisited v}).  From parts (\ref{L:FN revisited i}) and (\ref{L:FN revisited iii}) and $-I \in G$, the subfield of $\calF_N(\beta)$ fixed by $(\rho_{\calE,N}^*)^{-1}(G)$ is $\calF_N^G$.  Part (\ref{L:FN revisited v}) is now immediate since $\calF_N^G=\QQ(X_G)$.
\end{proof}

\subsection{Specializations} \label{SS:specializations}

Define $U:=\AA^1_\QQ - \{0,1728\} = \Spec \QQ[j, j^{-1},(j-1728)^{-1}]$ and view it as an open subvariety of the $j$-line.    Let $\pi_1(U,\bbar{\eta})$ be the \'etale fundamental group of $U$, where $\bbar{\eta}$ is the geometric generic point of $U$ corresponding to our choice of algebraic closure of $\QQ(j)$.    The Weierstrass equation (\ref{E:generic Weierstrass}) has discriminant $2^{12} 3^{12} j^2 (j - 1728)^3$ and hence defines an elliptic scheme $\scrE \to U$ whose generic fiber is the elliptic curve $\calE$ over $\QQ(j)$.  

Let $\scrE[N]$ be the $N$-torsion subscheme of $\scrE$.  We can identify the fiber of $\scrE[N]\to U$ at $\bbar{\eta}$ with $\calE[N]$.   We can view $\scrE[N]$ as a rank $2$ lisse sheaf of $\ZZ/N\ZZ$-modules over $U$ and it thus corresponds to a representation
\[
\varrho_{\scrE,N}\colon \pi_1(U,\bbar{\eta})\to \Aut(\calE[N]) \cong \GL_2(\ZZ/N\ZZ).
\]
By making an appropriate choice of basis, we may assume that the specialization of 
\[
\varrho^*_{\scrE,N}\colon \pi_1(U,\bbar{\eta})\to \GL_2(\ZZ/N\ZZ) 
\]
at the generic fiber of $U$ gives our representation $\rho^*_{\calE,N}\colon \Gal_{\QQ(j)} \to \GL_2(\ZZ/N\ZZ)$.  The representation $\varrho^*_{\scrE,N}$ is surjective by Lemma~\ref{L:FN revisited}(\ref{L:FN revisited ii}).

Take any subgroup $G$ of $\GL_2(\ZZ/N\ZZ)$ satisfying $\det(G)=(\ZZ/N\ZZ)^\times$ and $-I\in G$.   Let $\pi_G\colon U_G\to U$ be the \'etale cover corresponding to the subgroup $(\varrho^*_{\scrE,N})^{-1}(G)$ of $\pi_1(U,\bbar\eta)$.    The function field of $U_G$ is $\calF_N^G=\QQ(X_G)$ by Lemma~\ref{L:FN revisited}(\ref{L:FN revisited v}).  So we can identify $U_G$ with an open subvariety of the modular curve $X_G$ and the morphism $\pi_G$ extends to the morphism from $X_G$ to the $j$-line that we had also denoted by $\pi_G$.    In particular, $U_G$ is the open subvariety of $X_G$ that is the complement of $\pi_G^{-1}(\{0,1728,\infty\})$.  (When $-I \notin G$, we will simply define $U_G$ to be $X_G -\pi_G^{-1}(\{0,1728,\infty\})$.)

\begin{prop} \label{P:revised moduli property}
Let $G$ be a subgroup of $\GL_2(\ZZ/N\ZZ)$ satisfying $\det(G)=(\ZZ/N\ZZ)^\times$ and $-I\in G$.   Let $E$ be an elliptic curve defined over a number field $K$ with $j_E\notin\{0,1728\}$.  Then $\rho_{E,N}^*(\Gal_K)$ is conjugate in $\GL_2(\ZZ/N\ZZ)$ to a subgroup of $G$ if and only if $j_E =\pi_G(u)$ for some $u\in U_G(K)$.
\end{prop}
\begin{proof}
We may assume $K\subseteq \Qbar$. Define the surjective homomorphism $\varrho:=\varrho^*_{\scrE,N}\colon \pi_1(U,\bbar\eta)\to \GL_2(\ZZ/N\ZZ)$.   For each $u\in U(K)=K-\{0,1728\}$, let $\varrho_u\colon \Gal_K \to \GL_2(\ZZ/N\ZZ)$ be the specialization of $\varrho$ at $u$; it is uniquely determined up to conjugation by an element of $\GL_2(\ZZ/N\ZZ)$.

We claim that $\varrho_u(\Gal_K)$ is conjugate in $\GL_2(\ZZ/N\ZZ)$ to a subgroup of $G$ if and only if $u=\pi_G(P)$ for some $P\in U_G(K)$.    Let $\varphi\colon Y\to U$ be the \'etale cover corresponding to $\varrho$; the curve $Y$ is defined over $\QQ$ but need not be geometrically irreducible.  The group $\GL_2(\ZZ/N\ZZ)$ acts on $Y$ and acts simply faithful on the fiber $\varphi^{-1}(u):=\{P \in Y(\Qbar) : \varphi(P)=u\}$.  The group $\Gal_K$ acts on $\varphi^{-1}(u)$ since $\varphi$ and $u$ are defined over $K$. There is a point $P_0 \in \varphi^{-1}(u)$ such that $\sigma(P_0)=\varrho_u(\sigma)\cdot P_0$ for all $\sigma \in \Gal_K$ (a different choice of $P_0$ results in a conjugate of $\varrho_u$).  The $G$-orbits of $\varphi^{-1}(u)$  correspond with the points $P\in U_G(\Qbar)$ that satisfy $\pi_G(P)=u$ via the natural morphism $Y\to U_G$.  Since $\varphi$ and $u$ are defined over $K$, those $G$-orbits of $\varphi^{-1}(u)$ that are stable under the $\Gal_K$-action correspond with the points $P\in U_G(K)$ that satisfy $\pi_G(P)=u$.  Therefore, there is a point $P\in U_G(K)$ with $\pi_G(P)=u$ if and only if there is a point $P_0 \in \varphi^{-1}(u)$ such that for each $\sigma\in \Gal_K$, we have $\sigma(P_0)= g\cdot P_0$ for some $g\in G$; equivalently, $\varrho_u(\Gal_K)$ is conjugate to a subgroup of $G$.  This proves the claim.

The specialization $\varrho_u\colon \Gal_K\to \GL_2(\ZZ/N\ZZ)$ of $\varrho$ at $u$ is isomorphic to the representation $\rho_{\scrE_u,N}^*\colon \Gal_K \to \GL_2(\ZZ/N\ZZ)$, where $\scrE_u$ is the elliptic curve over $K$ defined by the Weierstrass equation (\ref{E:generic Weierstrass}) with $j$ replaced by $u$.  Since $\scrE_u$ has $j$-invariant $u$, the above claim proves the lemma in the case where $E=\scrE_u$.

Since $-I\in G$, it now suffices to show that $\pm \rho_E^*(\Gal_K)$ and $\pm \rho_{\scrE_u}^*(\Gal_K)$ are conjugate in $\GL_2(\ZZ/N\ZZ)$ where $u:=j_E \in K-\{0,1728\}=U(K)$.   The curve $E$ and $\scrE_u$ are quadratic twists of each other since they have the same $j$-invariant which is neither $0$ or $1728$.  So after choosing compatible bases, there is a quadratic character $\chi\colon \Gal_K \to \{\pm 1\}$ such that $\rho_{\scrE_u,N}^*=\chi\cdot \rho_{E,N}^*$.   In particular, $\pm \rho_{\scrE_u,N}^*(\Gal_K)=\pm \rho_{E,N}^*(\Gal_K)$.
\end{proof}

\subsubsection{Elliptic scheme for the group $G$} \label{SSS:elliptic scheme for G}

Let $G$ be a subgroup of $\GL_2(\ZZ/N\ZZ)$ satisfying $\det(G)=(\ZZ/N\ZZ)^\times$ and $-I\in G$.   Let $\scrE_G\to U_G$ be the elliptic scheme obtained by base extending $\scrE\to U$ by the \'etale morphism $\pi_G\colon U_G\to U$.  Then we can construct a representation $\varrho_{\scrE_G,N}^*\colon \pi_1(U_G,\bbar{\eta}) \to \GL_2(\ZZ/N\ZZ)$ as before, where now $\bbar\eta$ denotes the geometric generic point of $U_G$ corresponding to our choice of algebraic closure of $\QQ(j)$ (which contains $\QQ(X_G)$).   The morphism $\pi_G$ allows us to view $\pi_1(U_G,\bbar\eta)$ as a subgroup of $\pi_1(U,\bbar\eta)$; we may assume that our representation was chosen so that the restriction of $\varrho_{\scrE,N}^*$ to $\pi_1(U_G,\bbar\eta)$ agrees with $\varrho_{\scrE_G,N}^*$.    In particular, we have a surjective homomorphism
\[
\varrho_{\scrE_G,N}^*\colon \pi_1(U_G,\bbar\eta) \to G.
\]

Take any number field $K\subseteq \Qbar$ and point $u\in U_G(K)$.   Let $(\scrE_G)_u$ be the elliptic curve over $K$ that is the fiber of $\scrE_G\to U_G$ over $u$; it is isomorphic to the elliptic curve over $K$ defined by (\ref{E:generic Weierstrass}) with $j$ replaced by $\pi_G(u)$.  The specialization of $\varrho_{\scrE_G,N}^*$ at $u$ is a representation $\Gal_\QQ \to G \subseteq \GL_2(\ZZ/N\ZZ)$ that is isomorphic to $\rho_{(\scrE_G)_u,N}^*$.

\section{Some basic group theory} \label{S:group theory facts}

We now collect some basic group theoretic results that we will use.  Most of it concerns the groups $\SL_2(\ZZ_\ell)$ and $\GL_2(\ZZ_\ell)$, with $\ell$ prime, and the groups $\SL_2(\Zhat)$ and $\GL_2(\Zhat)$.  

\subsection{Goursat's lemma}
We will make frequent use of the following in \S\ref{S:agreeable}.

\begin{lemma}[Goursat's lemma, \cite{Ribet-76}*{Lemma 5.2.1}]  
\label{L:Goursat}
Let $G_1$ and $G_2$ be two groups and let $H$ be a subgroup of $G_1\times G_2$ so that the projection maps $p_1\colon H\to G_1$ and $p_2\colon H \to G_2$ are surjective.   Let $N_1$ and $N_2$ be the normal subgroups of $G_1$ and $G_2$, respectively, for which $\ker(p_2)=N_1\times \{1\}$ and $\ker(p_1)=\{1\}\times N_2$.    Then the image of $H$ in $(G_1\times G_2)/(N_1\times N_2)= G_1/N_1\times G_2/N_2$ is the graph of an isomorphism $G_1/N_1\xrightarrow{\sim} G_2/N_2$.
\end{lemma}

\subsection{Determining closed subgroups by their reductions}

\begin{lemma} \label{L:openness criterion GL}
Fix an integer $N=\prod_p p^{e_p}>1$ with $e_2 \neq 1$.  For each prime $\ell$ dividing  $N$, define the integer
\[
N_\ell := \ell^{e_\ell} \prod_{{p|N,\, p^2\equiv 1 \bmod{\ell}}} p;
\]
it is a divisor of $N$.  Let $G$ be a closed subgroup of $\GL_2(\ZZ_N)$.   For each prime $\ell$ dividing $N$, assume that the image of $G$ in $\GL_2(\ZZ/N_\ell\ell\ZZ)$ contains all matrices that are congruent to $I$ modulo $N_\ell$.   Then $G$ is an open subgroup of $\GL_2(\ZZ_N)$ with level dividing $N$. 
\end{lemma}
\begin{proof}
We first consider the special case where $N=\ell^{e_\ell}>1$ is a prime power (with $e_\ell\geq 2$ if $\ell=2$).  For $i\geq 1$, define the group $H_i:=G \cap (I+\ell^i M_2(\ZZ_\ell))$.   We have an injective homomorphism
\[
\phi_i\colon H_i/H_{i+1} \hookrightarrow \mathfrak{g},\quad 1+ \ell^i A \mapsto A \bmod{\ell},
\]
where $\mathfrak{g}:=M_2(\FF_\ell)$.

We claim that $\phi_i$ is surjective for all $i \geq e_\ell$.  We shall prove this by induction on $i$.   We have $N_\ell=\ell^{e_\ell}$ and the homomorphism $\phi_{e_\ell}$ is surjective by our assumption that the image of $G$ in $\GL_2(\ZZ/N_\ell \ell \ZZ)=\GL_2(\ZZ/\ell^{e_\ell+1}\ZZ)$ contains all matrices that are congruent to $I$ modulo $\ell^{e_\ell}$.      Now consider any $i\geq e_\ell$ for which $\phi_i$ is surjective.     Take any $B\in \mathfrak{g}$.   Since $\phi_i$ is surjective, there is a matrix $A\in M_2(\ZZ_\ell)$ such that $I + \ell^i A \in G$ and $A\equiv B \pmod{\ell}$.    Raising to the $\ell$-th power, we find that $g:=(I + \ell^i A)^\ell$ is an element of $G$ with $g\equiv I + \ell^{i+1} A \pmod{\ell^{i+2}}$ (this uses that $ {\ell \choose j}\equiv 0 \pmod{\ell}$ when $0<j<\ell$ and that $i \geq 2$ when $\ell=2$).   So $g \in H_{i+1}$ and $\phi_{i+1}(g\cdot H_{i+2}) \equiv A \equiv B \pmod{\ell}$.   Since $B\in \mathfrak{g}$ was arbitrary, this proves that $\phi_{i+1}$ is surjective.  The claim follows by induction.

For any $i\geq e_\ell$, the above claim implies that the image of $G$ in $\GL_2(\ZZ/\ell^i \ZZ)$ contains all matrices $A\in\GL_2(\ZZ/\ell^i \ZZ)$ with $A\equiv I \pmod{\ell^{e_\ell}}$.  Since $G$ is a closed subgroup of $\GL_2(\ZZ_\ell)$, we deduce that $G$ is an open subgroup of $\GL_2(\ZZ_\ell)$ whose level divides $\ell^{e_\ell}$.   This completes the proof in the special case where $N$ is a prime power.\\

We now consider the general case.  Take any prime $\ell|N$ and let 
\[
\varphi\colon G \to \prod_{p|N,\, p\neq \ell} \GL_2(\ZZ_p)
\]
be the projection homomorphism.   We have $\ker(\varphi)=H_\ell \times \{I\}$, where $H_\ell$ is a closed subgroup of $\GL_2(\ZZ_\ell)$.

Take any  $B\in \mathfrak{g}$.    By assumption, the image of $G$ in $\GL_2(\ZZ/
N_\ell \ell \ZZ)$ contains all matrices that are congruent to $I$ modulo $N_\ell$.   So there is a $g_0\in G$ such that $g_0\equiv I +\ell^{e_\ell} B \pmod{\ell^{e_\ell+1}}$ and $g_0\equiv I \pmod{p}$ for all $p|N$ with $p^2\equiv 1 \pmod{\ell}$.  Observe that if $\ell$ divides $|\GL_2(\ZZ/p\ZZ)|=(p-1)^2(p+1) p$ for a prime $p\neq \ell$, then $p^2\equiv 1\pmod{\ell}$.  So for any $m\geq 1$, there is a positive integer $f\equiv 1 \pmod{\ell}$ for which $g_0^f\equiv \varphi(g_0)^f \equiv I \pmod{\prod_{p|N, p\neq \ell} p^{m}}$.  We have $g_0^f \equiv  I +\ell^{e_\ell} B \pmod{\ell^{e_\ell+1}}$.  By taking $m$ larger and larger and using that $G$ is a closed, and hence compact, subgroup of $\GL_2(\ZZ_N)$, we deduce that there is a $g\in \ker(\varphi)$ for which $g\equiv I +\ell^{e_\ell} B \pmod{\ell^{e_\ell+1}}$.    Since $B \in \mathfrak{g}$ was arbitrary, we deduce that the image of $H_\ell$ in $\GL_2(\ZZ/\ell^{e_\ell+1}\ZZ)$ contains all matrices that are congruent to $I$ modulo $\ell^{e_\ell}$.   

By the prime power case of the lemma, which we have already proved, $H_\ell$ is an open subgroup of $\GL_2(\ZZ_\ell)$ whose level divides $\ell^{e_\ell}$.  Since $\ell$ was an arbitrary prime divisor of $N$, we deduce that $\prod_{\ell|N} H_\ell$ is a subgroup of $G$ and hence $G$ is an open subgroup of $\GL_2(\ZZ_N)$ whose level divides $\prod_{\ell|N}\ell^{e_\ell} = N$.
\end{proof}

For an open subgroup $G$ of $\GL_2(\ZZ_N)$, the following lemma is useful for finding its maximal subgroups; the issue being that these maximal subgroups may have strictly larger level.

\begin{lemma}
Fix an integer $N=\prod_p p^{e_p}>1$ with $e_2\neq 1$.   Let $G$ be an open subgroup of $\GL_2(\ZZ_N)$ whose level divides $N$.  If $M$ is a maximal open subgroup of $G$, then the level of $M$ divides $N \ell$ for some $\ell |N$.
\end{lemma}
\begin{proof}
Let $M$ be a maximal open subgroup of $G$.   Take any prime $\ell |N$.  If the image of $M$ and $G$ modulo $N\ell$ give different subgroups of $\GL_2(\ZZ/N\ell\ZZ)$, then the level of $M$ divides $N\ell$ (using that $M$ is maximal, it agrees with the group of $g\in G$ whose image modulo $N\ell$ lies in the image of $M$ modulo $N\ell$).   

So we can assume that $M$ and $G$ have the same image modulo $N\ell$ for all primes $\ell |N$.   Since $G$ has level $N$, we deduce that for each $\ell|N$, the image of $M$ modulo $N\ell$ contains all matrices $A\in \GL_2(\ZZ/N\ell\ZZ)$ satisfying $A\equiv I \pmod{N}$.   Applying Lemma~\ref{L:openness criterion GL} to the group $M$, we deduce that the level of $M$ divides $N$.
\end{proof}

The following is an analogous version of Lemma~\ref{L:openness criterion GL}.

\begin{lemma} \label{L:openness criterion SL}
Fix an integer $N=\prod_p p^{e_p}>1$ with $e_2 \neq 1$.  For each prime $\ell$ dividing  $N$, define the integer
\[
N_\ell := \ell^{e_\ell} \prod_{{p|N,\, p^2\equiv 1 \bmod{\ell}}} p;
\]
it is a divisor of $N$. 
  Let $G$ be a closed subgroup of $\SL_2(\ZZ_N)$.   For each prime $\ell$ dividing $N$, assume that the image of $G$ in $\SL_2(\ZZ/N_\ell\ell\ZZ)$ has level which divides $N_\ell$.    Then $G$ is an open subgroup of $\SL_2(\ZZ_N)$ with level dividing $N$. 
\end{lemma}
\begin{proof}
This can be proved in the exact same way as Lemma~\ref{L:openness criterion GL} with $\GL_2$ replaced by $\SL_2$ and with $\mathfrak{g}=\{A\in M_2(\FF_\ell) :  \tr(A)=0\}$.
\end{proof}

\begin{lemma} \label{L:reduction original Serre alt}
Fix a prime $\ell$.  Let $G$ be a closed subgroup of $\SL_2(\ZZ_\ell)$ whose image modulo $\ell^e$ is $\SL_2(\ZZ/\ell^e\ZZ)$, where $e$ is $3$, $2$ or $1$ when $\ell$ is $2$, $3$ or at least $5$, respectively.  Then $G=\SL_2(\ZZ_\ell)$.
\end{lemma}
\begin{proof}
For $\ell \geq 5$, the lemma follows from \cite[IV, Lemma 3]{Serre-abelian}. For $\ell \in \{2,3\}$, this follows from Lemma~\ref{L:openness criterion SL}.  
\end{proof}

Many of the open subgroups $G$ of $\GL_2(\ZZ_N)$ we consider will contain all the scalar matrices; the following will be useful for computing levels.

\begin{lemma} \label{L:level by adjoining scalars}
Fix an integer $N>1$.  Let $H$ be a subgroup of $\GL_2(\ZZ_N)$ for which $H \cap \SL_2(\ZZ_N)$ is an open subgroup of $\SL_2(\ZZ_N)$ whose level divides $N_0$.  Assume that $N_0\equiv 0 \pmod{4}$ if $N_0$ is even.  Define $N_1:=N_0$ if $N_0$ is odd and $N_1:=2N_0$ if $N_0$ is even.   Then $\ZZ_N^\times \cdot H$ is an open subgroups of $\GL_2(\ZZ_N)$ whose level divides $N_1$.
\end{lemma}
\begin{proof}
By considering the $\ell$-adic projections, we reduce to the case where $N$ is a power of a prime $\ell$ and $N_0=\ell^e$ (with $e\geq 2$ if $\ell=2$).  Let $f=e$ if $\ell$ is odd and $f=e+1$ if $\ell=2$.  One can check that $(1+\ell^e\ZZ_\ell)^2=1+\ell^f \ZZ_\ell$; this uses that $e\geq 2$ when $\ell=2$.

  Take any matrix $A \in I+\ell^f M_2(\ZZ_\ell)$.   We have $\det(A) \in 1+\ell^f \ZZ_\ell = (1+\ell^e \ZZ_\ell)^2$, so $\det(A)=d^2$ for some $d\in 1+\ell^e \ZZ_\ell$.  The matrix $d^{-1}A \in \GL_2(\ZZ_\ell)$ has determinant $1$ and is congruent to $I$ modulo $\ell^e$.   Since the level of $H\cap \SL_2(\ZZ_N)$ in $\SL_2(\ZZ_N)$ divides $N_0=\ell^e$, we must have $d^{-1} A \in H$.   Therefore, $A = d\cdot d^{-1} A$ is an element of $\Zhat^\times H$.  The level of $\Zhat^\times H$ divides $\ell^f=N_1$ since $A$ was an arbitrary element of $I+\ell^f M_2(\ZZ_\ell)$.
\end{proof}

\subsection{Commutator subgroups}

We first consider some basic results about commutator subgroups.

\begin{lemma} \label{L:basics commutators}
\begin{romanenum}
\item  \label{L:basics commutators i}
If $\ell\geq 5$, the group $\SL_2(\ZZ_\ell)$ is perfect, i.e., it is equal to its commutator subgroup.
\item \label{L:basics commutators ii}
The commutator subgroup of $\SL_2(\ZZ_3)$ has level $3$ and index $3$.
\item \label{L:basics commutators iii}
The commutator subgroup of $\SL_2(\ZZ_2)$ has level $4$ and index $4$, and the quotient is a cyclic group of order $4$.
\item  \label{L:basics commutators iv}
 If $\ell \geq 3$, the commutator subgroup of $\GL_2(\ZZ_\ell)$ is $\SL_2(\ZZ_\ell)$.
 \item  \label{L:basics commutators v}
 The commutator subgroup of $\GL_2(\ZZ_2)$ is a subgroup of $\SL_2(\ZZ_2)$ of level $2$ and index $2$.
\item 
\label{L:basics commutators vi}
The commutator subgroup of $\GL_2(\Zhat)$ is an open subgroup of $\SL_2(\Zhat)$ of level $2$ and index $2$.
\end{romanenum}
\end{lemma}
\begin{proof}
Parts (\ref{L:basics commutators i})--(\ref{L:basics commutators iii}) follow immediately from \cite[Lemma A.1]{MR2721742}.  Now take any prime $\ell$.   We have 
\[
[\SL_2(\ZZ_\ell), \SL_2(\ZZ_\ell)] \subseteq [\GL_2(\ZZ_\ell),\GL_2(\ZZ_\ell)] \subseteq \SL_2(\ZZ_\ell).
\]   
Part (\ref{L:basics commutators iv}) with $\ell \geq 5$ is now a consequence of (\ref{L:basics commutators i}).

When $\ell = 3$, we deduce from (\ref{L:basics commutators ii}) that $[\GL_2(\ZZ_3),\GL_2(\ZZ_3)]$ has level $1$ or $3$ in $\SL_2(\ZZ_3)$.   Part (\ref{L:basics commutators iv}) with $\ell=3$ follows by verifying that $[\GL_2(\ZZ/3\ZZ),\GL_2(\ZZ/3\ZZ)]=\SL_2(\ZZ/3\ZZ)$.

When $\ell = 2$, we deduce from (\ref{L:basics commutators iii}) that the level of $[\GL_2(\ZZ_2),\GL_2(\ZZ_2)]$ in $\SL_2(\ZZ_2)$ divides $4$.    Part (\ref{L:basics commutators v}) with $\ell=2$ follows by verifying that $[\GL_2(\ZZ/4\ZZ),\GL_2(\ZZ/4\ZZ)]$ is a subgroup of $\SL_2(\ZZ/4\ZZ)$ of index $2$ that contains all the matrices $A\in \SL_2(\ZZ/4\ZZ)$ with $A\equiv I \pmod{2}$.

For the group $\GL_2(\Zhat)=\prod_{\ell} \GL_2(\ZZ_\ell)$, part (\ref{L:basics commutators vi}) follows from (\ref{L:basics commutators iv}) and (\ref{L:basics commutators v}).
\end{proof}

\begin{lemma} \label{L:pm H eq SL2}
Take any integer $N>1$.  There are no proper subgroups $H$ of $\SL_2(\ZZ/N\ZZ)$ that satisfy $\pm H=\SL_2(\ZZ/N\ZZ)$.
\end{lemma}
\begin{proof}
Suppose that $H$ is a proper subgroup of $\SL_2(\ZZ/N\ZZ)$ that satisfies $\pm H=\SL_2(\ZZ/N\ZZ)$.  The group $H$ has index $2$ in $\SL_2(\ZZ/N\ZZ)$ and contains the commutator subgroup of $\SL_2(\ZZ/N\ZZ)$.   Using the description of commutator subgroups of $\SL_2(\ZZ_\ell)$ with $\ell|N$ from Lemma~\ref{L:basics commutators}, we find that $N$ is even and $H$ is the unique index $2$ subgroup of $\SL_2(\ZZ/N\ZZ)$ containing all matrices $A\in \SL_2(\ZZ/N\ZZ)$ for which $A \equiv I \pmod{2}$.   In particular, we have $-I \in H$ since $-I\equiv I \pmod{2}$.   This contradicts that $H$ is a proper subgroup of $\pm H$.
\end{proof}

\begin{lemma}  \label{L:modell to elladic}
For a prime $\ell \geq 5$, let $G$ be a closed subgroup of $\GL_2(\ZZ_\ell)$ for which $\det(G)=\ZZ_\ell^\times$ and $G$ modulo $\ell$ is equal to $\GL_2(\ZZ/\ell\ZZ)$.  Then $G=\GL_2(\ZZ_\ell)$.
\end{lemma}
\begin{proof}
The commutator subgroup $[G,G]$ is a closed subgroup of $\SL_2(\ZZ_\ell)$ whose image modulo $\ell$ is $[\GL_2(\ZZ/\ell\ZZ),\GL_2(\ZZ/\ell\ZZ)]=\SL_2(\ZZ/\ell\ZZ)$, where we have used our assumption on the image of $G$ modulo $\ell$ and Lemma~\ref{L:basics commutators}(\ref{L:basics commutators i}).   By Lemma~\ref{L:reduction original Serre alt}, we have $[G,G]=\SL_2(\ZZ_\ell)$.   We conclude that $G=\GL_2(\ZZ_\ell)$ since $G\supseteq [G,G]=\SL_2(\ZZ_\ell)$ and $\det(G)=\ZZ_\ell^\times$.
\end{proof}

\begin{lemma} \label{L:openness of commutator}
Let $G$ be an open subgroup of $\GL_2(\Zhat)$ or $\SL_2(\Zhat)$.  Then the commutator subgroup $[G,G]$ is an open subgroup of $\SL_2(\Zhat)$.
\end{lemma}
\begin{proof}
We obviously have $[G,G]\subseteq \SL_2(\Zhat)$.  We can always replace $G$ by a smaller group since this will make $[G,G]$ smaller as well.   So without loss of generality, we may assume that $G$ is an open subgroup of $\SL_2(\Zhat)$ and that 	$G=\prod_{\ell} G_\ell$ with $G_\ell$ an open subgroup of $\SL_2(\ZZ_\ell)$.     For $\ell\geq 5$ large enough, we have $G_\ell=\SL_2(\ZZ_\ell)$ and $[G_\ell,G_\ell]=\SL_2(\ZZ_\ell)$ by Lemma~\ref{L:basics commutators}(\ref{L:basics commutators i}).    So for any fixed prime $\ell$, we need only show that $[G_\ell,G_\ell]$ is an open subgroup of $\SL_2(\ZZ_\ell)$.

Take any prime $\ell$.   By Lemma~\ref{L:level by adjoining scalars}, there is an integer $e\geq 1$ such that $W:=\ZZ_\ell^\times \cdot G_\ell \supseteq 1+\ell^e M_2(\ZZ_\ell)$.   By Lemma~1 of \cite{MR0568299}*{p.163}, we have $[G_\ell,G_\ell]=[W,W]\supseteq (1+\ell^{2e} M_2(\ZZ_\ell)) \cap \SL_2(\ZZ_\ell)$ and hence $[G_\ell,G_\ell]$ is an open subgroup of $\SL_2(\ZZ_\ell)$.
\end{proof}

\subsubsection{Computing commutator subgroups} \label{SSS:computing commutator subgroups}

Now let $G$ be an open subgroup of $\SL_2(\Zhat)$ or $\GL_2(\Zhat)$.     Assume $G$ is given explicitly, i.e., we have generators for the image of $G$ modulo $m$ for some positive integer $m$ divisible by the level of $G$.    

We will need to be able to compute the commutator subgroup $[G,G]$; note the level of $[G,G]$ in $\SL_2(\Zhat)$ may be strictly larger than $m$.

By Lemma~\ref{L:openness of commutator}, $[G,G]$ is an open subgroup of $\SL_2(\Zhat)$.  By  Lemma~\ref{L:basics commutators}(\ref{L:basics commutators i}), we find that the level of $[G,G]$ is not divisible by any prime $\ell\geq 5$ for which $\ell\nmid m$.  For any positive integer $n$, we can compute the image of $[G,G]$ modulo $n$; it equals $[\bbar{G},\bbar{G}]\subseteq \SL_2(\ZZ/n\ZZ)$, where $\bbar{G}$ is the image of $G$ modulo $n$.   By computing $[G,G]$ modulo $n$ for suitable $n$ that are not divisible by any prime $\ell\geq 5$ satisfying $\ell\nmid m$, we can use the criterion of Lemma~\ref{L:openness criterion SL} to find a positive integer $N$ that is divisible by the level of $[G,G]$ in $\SL_2(\Zhat)$.   Once we have such a $N$, we have explicitly found $[G,G]$ since it is determined by $N$ and its image modulo $N$.

\subsection{Full $\ell$-adic images}

The next lemma shows that if $G$ is a proper closed subgroup of $\GL_2(\ZZ_\ell)$ with full determinant, then it is prosolvable.
 
\begin{lemma} \label{L:JH factors} 
Take any prime $\ell$ and let $G$ be a closed subgroup of $\GL_2(\ZZ_\ell)$ with $\det(G)=\ZZ_\ell^\times$.  If $S$ is a nonabelian simple group that occurs as the quotient of some closed normal subgroup of $G$, then $\ell\geq 5$,  $G=\GL_2(\ZZ_\ell)$ and $S\cong \SL_2(\FF_\ell)/\{\pm I\}$.
\end{lemma}
\begin{proof}
The kernel of the reduction modulo $\ell$ homomorphism $G\to \GL_2(\FF_\ell)$ is prosolvable.   So the image $\bbar{G}$ of $G$ in $\GL_2(\FF_\ell)$ satisfies $\det(\bbar{G})=\FF_\ell^\times$ and contains the nonabelian simple group $S$ as a factor in its composition series.  
Since $\GL_2(\FF_2)$ and $\GL_2(\FF_3)$ are solvable, we have $\ell\geq 5$.   

First suppose that the cardinality of $\bbar{G}$ is not divisible by $\ell$.  Since $\bbar{G}$ is nonsolvable, we find from \S2.6 of \cite{Serre-Inv72} that the image of $\bbar{G}$ in $\operatorname{PGL}_2(\FF_\ell):=\GL_2(\FF_\ell)/(\FF_\ell^\times\cdot I)$ is isomorphic to the alternating group $\mathfrak{A}_5$.  Since $\mathfrak{A}_5$ is nonabelian and simple, we must have $\det(\bbar{G})\subseteq (\FF_\ell^\times)^2$ which is a contradiction.

So the prime $\ell$ must divide the cardinality of $\bbar{G}$.   Since $\bbar{G}$ is nonsolvable,  \cite[Proposition~15]{Serre-Inv72} shows that $\bbar{G}\supseteq \SL_2(\FF_\ell)$.  The group $\SL_2(\FF_\ell)$ is perfect since $\ell\geq 5$ and hence $[G,G]$ is a closed subgroup of $\SL_2(\ZZ_\ell)$ whose image modulo $\ell$ is $\SL_2(\FF_\ell)$. We have $[G,G]=\SL_2(\ZZ_\ell)$ by Lemma~\ref{L:reduction original Serre alt}.   We thus have $G=\GL_2(\ZZ_\ell)$ since $\det(G)=\ZZ_\ell^\times$.   Finally $S$ must be isomorphic to $\SL_2(\FF_\ell)/\{\pm I\}$ since it is the only nonabelian simple group that occurs as a factor in a composition series of $\bbar{G}=\GL_2(\FF_\ell)$.
\end{proof}

\begin{lemma} \label{L:prime divisors of level}
Let $G$ be an open subgroup of $\GL_2(\Zhat)$ satisfying $\det(G)=\Zhat^\times$.  Take any prime $\ell \geq 5$ and let $G_\ell$ be the image of $G$ under the projection map $\GL_2(\Zhat)\to \GL_2(\ZZ_\ell)$.   Then the following are equivalent:
\begin{alphenum}
\item \label{I:a}
$G_\ell = \GL_2(\ZZ_\ell)$,
\item \label{I:b}
$\ell$ does not divide the level of $[G,G]$ in $\SL_2(\Zhat)$,
\item \label{I:c}
$\ell$ does not divide the level of $G\cap \SL_2(\Zhat)$ in $\SL_2(\Zhat)$.
\end{alphenum}
\end{lemma}
\begin{proof}
First suppose that $G_\ell = \GL_2(\ZZ_\ell)$.    Let $H$ be one of the groups $[G,G]$ or $G\cap \SL_2(\Zhat)$; it is an open subgroup of $\SL_2(\Zhat)$.  Since $G\cap \SL_2(\Zhat) \supseteq [G,G]$, we have $\SL_2(\ZZ_\ell)\supseteq H_\ell \supseteq [G,G]_\ell = [G_\ell,G_\ell]$, where $H_\ell$ and $[G,G]_\ell$ are the images of $H$ and $[G,G]$, respectively, under the projection map $\SL_2(\Zhat)\to \SL_2(\ZZ_\ell)$.    By Lemma~\ref{L:basics commutators}(\ref{L:basics commutators iv}), we deduce that $H_\ell=\SL_2(\ZZ_\ell)$.      Let $H'$ be the image of $H$ under the projection $\SL_2(\Zhat)\to \prod_{p\neq \ell} \SL_2(\ZZ_p)$.   So we may identify $H$ with a closed subgroup of $H_\ell \times H'$.   Let $B$ and $B'$ be the normal subgroups of $H_\ell$ and $H'$, respectively, such that $B\times\{I\}$ is the kernel of the projection $H\to H'$ and $\{I\} \times B'$ is the kernel of the projection $H\to H_\ell$.   Since $H$ is open in $\SL_2(\Zhat)$, the groups $B$ and $B'$ are open in $H_\ell$ and $H'$, respectively.  By Goursat's lemma, we have an isomorphism of (finite) groups $H_\ell/B \cong H'/B'$.   

Suppose that $H_\ell/B\neq 1$.  There is a simple group $S$ that is the quotient of both $H_\ell$ and $H_p$ for some prime $p\neq \ell$.   The groups $H_\ell$ and $H_p$ are normal subgroups of $G_\ell$ and $G_p$, respectively, since $H$ is a normal subgroup of $G$.  Since the groups $\SL_2(\FF_\ell)/\{\pm I\}$ and $\SL_2(\FF_p)/\{\pm I\}$ have different cardinalities and $G_\ell$ and $G_p$ have full determinants, Lemma~\ref{L:JH factors} implies that $S$ is abelian.  However, $H_\ell=\SL_2(\ZZ_\ell)$ has no nontrivial abelian quotients by Lemma~\ref{L:basics commutators}(\ref{L:basics commutators i}).   Therefore, $H_\ell/B=1$.  

The groups $H_\ell/B$ and $H'/B'$ are trivial.  Therefore, $H \supseteq B\times B' = H_\ell \times H'$ and hence $H=H_\ell\times H' = \SL_2(\ZZ_\ell) \times H'$.  This description of $H$ shows that $\ell$ does not divide its level.  This shows that (\ref{I:a}) implies (\ref{I:b}) and (\ref{I:c}).

Now suppose that (\ref{I:b}) or (\ref{I:c}) holds.  In either case, we have $G_\ell \supseteq \SL_2(\ZZ_\ell)$.  Since $G$ has full determinant, this implies (\ref{I:a}).
\end{proof}

\subsection{Some $3$-adic computations} \label{SS:3-adic computations}

To study the prime $\ell=3$, we define a surjective homomorphism
\[
\varphi_3\colon \GL_2(\ZZ_3)\to \GL_2(\FF_3)\to \GL_2(\FF_3)/[\SL_2(\FF_3),\SL_2(\FF_3)] \xrightarrow{\sim} \mathfrak{S}_3,
\]
where we compose reduction modulo $3$, the natural quotient map, and a choice of isomorphism with the symmetric group $\mathfrak{S}_3$.  Note that $\varphi_3(\SL_2(\ZZ_3))$ is the alternating group $\mathfrak{A}_3$.

\begin{lemma} \label{L:varphi3 condition}
If $H$ is a closed normal subgroup of $\GL_2(\ZZ_3)$, then either $\varphi_3(H)=1$ or $H\supseteq \SL_2(\ZZ_3)$.
\end{lemma}
\begin{proof}
If $H\supseteq \SL_2(\ZZ_3)$, then $\varphi_3(H) \supseteq \varphi_3(\SL_2(\ZZ_3))=\mathfrak{A}_3\neq 1$.  So we may now assume that $\varphi_3(H)\neq 1$.  We need to show that $H\supseteq \SL_2(\ZZ_3)$.  

We have $\varphi_3(H)\supseteq \mathfrak{A}_3$ since $\varphi_3(H)$ is a nontrivial normal subgroup of $\mathfrak{S}_3$.  So there is an $a\in H$ such that $\varphi_3(a) \in \mathfrak{A}_3-\{1\}$.  The two elements of $\mathfrak{A}_3-\{1\}$  are conjugate in $\mathfrak{S}_3$, so there is a $g\in \GL_2(\ZZ_3)$ such that  $\varphi_3(g a^{-1} g^{-1}) = \varphi_3(g) \varphi_3(a)^{-1} \varphi_3(g)^{-1}$ equals $\varphi_3(a)$.  Therefore, $\varphi_3(g a^{-1} g^{-1} a)= \varphi_3(a)^2$ has order $3$.    Since $H$ is normal in $\GL_2(\ZZ_3)$, we deduce that $\varphi_3(g a^{-1} g^{-1} a)$ is an element of $\varphi_3(H \cap \SL_2(\ZZ_3))$ of order $3$.   So after replacing $H$ by $H\cap \SL_2(\ZZ_3)$, we may assume that $H\subseteq \SL_2(\ZZ_3)$.   


The image $\bbar{H}$ of $H$ modulo $3$ is a normal subgroup of $\SL_2(\ZZ/3\ZZ)$ that contains an element of order $3$.   We have $\bbar{H}=\SL_2(\ZZ/3\ZZ)$ since the only maximal normal subgroup of $\SL_2(\ZZ/3\ZZ)$ is its commutator subgroup which has order $8$.   A direct computation shows that there are no normal subgroup of $\SL_2(\ZZ/9\ZZ)$ whose image modulo $3$ is the group $\SL_2(\ZZ/3\ZZ)$ (however, there are such subgroups if we exclude the normal condition).   So $H$ is a closed subgroup of $\SL_2(\ZZ_3)$ whose image modulo $9$ equals $\SL_2(\ZZ/9\ZZ)$.  By Lemma~\ref{L:reduction original Serre alt}, we have $H=\SL_2(\ZZ_3)$.
\end{proof}

\section{Agreeable groups} \label{S:agreeable}

We say that subgroup $\calG$ of $\GL_2(\Zhat)$ is \defi{agreeable} if it is open in $\GL_2(\Zhat)$, satisfies $\det(\calG)=\Zhat^\times$, contains the scalar matrices $\Zhat^\times\cdot I$, and the levels of $\calG\subseteq\GL_2(\Zhat)$ and $\calG\cap \SL_2(\Zhat)\subseteq\SL_2(\Zhat)$ have the same odd prime divisors.

The following proposition, which we will prove in \S\ref{SS:key construction proof}, shows that every open subgroup $G$ of $\GL_2(\Zhat)$ with full determinant lies in a unique minimal agreeable group $\calG \subseteq\GL_2(\Zhat)$.   We call $\calG$ the \defi{agreeable closure} of $G$. 

\begin{prop} \label{P:agreeable closure}
Let $G$ be an open subgroup of $\GL_2(\Zhat)$ with $\det(G)=\Zhat^\times$.   Then there is a unique minimal  agreeable subgroup $\calG$ of $\GL_2(\Zhat)$, with respect to inclusion, that satisfies $G \subseteq \calG$.   We have $[G,G]=[\calG,\calG]$ and hence $G$ is a normal subgroup of $\calG$ with $\calG/G$ finite and  abelian.
\end{prop}

Recall that for our application to Serre's open image theorem, we will study the agreeable closure $\calG_E$ of $\rho_E^*(\Gal_\QQ)$ for non-CM elliptic curves $E/\QQ$.   Since $\rho_E^*$ is defined up to isomorphism, the group $\calG_E\subseteq \GL_2(\Zhat)$ is uniquely determined up to conjugacy in $\GL_2(\Zhat)$.

\subsection{Projection notation} \label{projection notation}

We now introduce notation that we will use through \S\ref{S:agreeable}.

Let $G$ be a subgroup of $\GL_2(\Zhat)$.  For each positive integer $n$, we let $G_n$ be the image of $G$ under the homomorphism $\GL_2(\Zhat)\to \GL_2(\ZZ_n)$ arising from the natural ring homomorphism $\Zhat=\ZZ_n \times \prod_{\ell\nmid n} \ZZ_\ell \to \ZZ_n$.    For example, the \emph{level} of an open subgroup $G$ of $\GL_2(\Zhat)$ is the smallest positive integer $n$ for which we have a natural identification $G=G_n \times \prod_{\ell \nmid n} \GL_2(\ZZ_\ell)$.

Let $G$ be a subgroup of $\GL_2(\ZZ_N)$ for an integer $N>1$.  For a positive integer $n$ that divides some power of $N$, we denote by $G_n$ the image of $G$ under the homomorphism $\GL_2(\ZZ_N)\to \GL_2(\ZZ_n)$ arising from the projection ring homomorphism $\ZZ_N=\ZZ_n \times \prod_{\ell|N, \, \ell\nmid n} \ZZ_\ell \to \ZZ_n$. 

Suppose $N=n_1 n_2$ with $n_1$ and $n_2$ positive integers that are relatively prime.   Then any subgroup $G$ of $\GL_2(\ZZ_N)$ can be identified with a subgroup of $G_{n_1}\times G_{n_2}$; the projections of $G$ onto the first and second coordinate are surjective.  Note that we are now in a setting where we can apply Goursat's lemma (Lemma~\ref{L:Goursat}).

\subsection{Finiteness of agreeable groups with given projections}

An advantage of working with agreeable groups is that there are only finitely many with given $\ell$-adic projections.

\begin{lemma} \label{L:agreeable finiteness}
For each prime $\ell$, fix an open subgroup $H_\ell$ of $\GL_2(\ZZ_\ell)$ satisfying $\det(H_\ell)=\ZZ_\ell^\times$ and $\ZZ_\ell^\times \cdot I \subseteq H_\ell$.   Then there are only finite many agreeable subgroups $G$ of $\GL_2(\Zhat)$ such that $G_\ell=H_\ell$ for all $\ell$.
\end{lemma}
\begin{proof}
Take any agreeable subgroup $G$ of $\GL_2(\Zhat)$ such that $G_\ell = H_\ell$ for all $\ell$.  We may assume that there are only finite many primes $\ell$ with $H_\ell \neq \GL_2(\ZZ_\ell)$ since otherwise it would contradict that $G$ is open in $\GL_2(\Zhat)$.   Let $N$ be the product of primes $\ell$ for which $\ell \leq 3$ or $H_\ell \neq \GL_2(\ZZ_\ell)$.   

For a group $H$ and an integer $n\geq 0$, let $H^{(n)}$ be the group obtained from $H$ by taking commutator subgroups $n$ times, i.e., $H^{(0)}:=H$ and $H^{(n)}:=[H^{(n-1)},H^{(n-1)}]$ for $n\geq 1$.    Define $B:=G^{(4)}$; it is an open subgroup of $\SL_2(\Zhat)$ since $G$ is open in $\GL_2(\Zhat)$, cf.~Lemma~\ref{L:openness of commutator}.   We have a natural inclusion
\[
B \subseteq \prod_\ell H_\ell^{(4)}
\]
so that the projection map $B\to B_\ell =(G_\ell)^{(4)}=H_\ell^{(4)}$ is surjective for all $\ell$.

Suppose that $B=\prod_\ell H_\ell^{(4)}$.  The primes dividing the level of $B\subseteq \SL_2(\Zhat)$ divide $N$ since $H_\ell^{(4)}=\SL_2(\ZZ_\ell)^{(4)}=\SL_2(\ZZ_\ell)$ for $\ell \nmid N$, where we have used Lemma~\ref{L:basics commutators}(\ref{L:basics commutators i}).   Since $B\subseteq [G,G]$ and $G$ is agreeable, we find that the primes dividing the level of $G$ divide $N$ as well.   Therefore, we have an inclusion
\[
W:= \prod_{\ell |N} (\ZZ_\ell^\times \cdot H_\ell^{(4)}) \times \prod_{\ell\nmid N}\GL_2(\ZZ_\ell) \subseteq G.
\]
Since $W$ is an open subgroup of $\GL_2(\Zhat)$, $G$ is one of the finitely many groups that lies between $W$ and $\GL_2(\Zhat)$.

Now suppose that $B\neq \prod_\ell H_\ell^{(4)}$.   To finish the proof, it suffices to obtain a contradiction.  By Goursat's lemma, we find that there is a finite simple group that is isomorphic to the quotient of $H_\ell^{(4)}$ for two distinct primes $\ell$.  For $\ell \nmid N$, we have $H_\ell^{(4)}=\SL_2(\ZZ_\ell)$ by Lemma~\ref{L:basics commutators}(\ref{L:basics commutators i}) and its only simple quotient is isomorphic to $\SL_2(\FF_\ell)/\{\pm I\}$.

For $\ell | N$, we claim that $H_{\ell}^{(4)}$ is pro-$\ell$ and hence every simple quotient is isomorphic to $\ZZ/\ell\ZZ$.   By Lemma~\ref{L:JH factors}, the group $H_\ell \subseteq \GL_2(\ZZ_\ell)$ is prosolvable.   So to verify the claim, it suffices to prove that $M^{(4)}=1$ for any solvable subgroup $M\subseteq \GL_2(\ZZ/\ell\ZZ)$.   For $\ell \in \{2,3\}$ this holds since a direct computation shows that $\GL_2(\ZZ/\ell\ZZ)^{(4)}=1$.   Now assume $\ell \geq 5$.  For a description of the subgroups of $\GL_2(\ZZ/\ell\ZZ)$, see \S2 of \cite{Serre-Inv72}.  If $M$ is contained in the normalizer of a Cartan subgroup or a Borel subgroups, we have $M^{(2)}=1$.  The remaining possibility is that the image of $M$ in $\GL_2(\ZZ/\ell\ZZ)/((\ZZ/\ell\ZZ)^\times\cdot I)$ is isomorphic to a subgroup of $\mathfrak{S}_4$.  So $M^{(4)}=1$ since $\mathfrak{S}_4^{(3)}=1$. So we have proved that there is no simple group that is isomorphic to quotients of $H_\ell^{(4)}$ for two distinct primes $\ell$; this is the desired contradiction. 
\end{proof}

\subsection{Agreeable closure} \label{SS:key construction proof}

Fix an open subgroup $G$ of $\GL_2(\Zhat)$ that satisfies $\det(G)=\Zhat^\times$.  Let $N$ be the product of primes that divide the level of $[G,G] \subseteq \SL_2(\Zhat)$.  Define the subgroup
\begin{align}\label{E:agreeable closure}
\calG:= (\ZZ_N^\times \cdot G_N) \times \prod_{\ell\nmid N} \GL_2(\ZZ_\ell)
\end{align}
of $\GL_2(\Zhat)$.    The rest of this section is devoted to showing that $\calG$ is the agreeable group satisfying the properties in Proposition~\ref{P:agreeable closure}.  \\

We clearly have $G\subseteq \calG$ and $\Zhat^\times \cdot I \subseteq \calG$.  The group $\calG$ is open in $\GL_2(\Zhat)$ with full determinant since $G$ has these properties.  The integer $N$ is even since the commutator subgroup of $\GL_2(\Zhat)$ has level $2$ in $\SL_2(\Zhat)$ by Lemma~\ref{L:basics commutators}(\ref{L:basics commutators vi}).  The commutator subgroups of $\ZZ_N^\times \cdot G_N$ and $G_N$ agree, and $\GL_2(\ZZ_\ell)$ has commutator subgroup $\SL_2(\ZZ_\ell)$ for all $\ell\nmid N$ by Lemma~\ref{L:basics commutators}(\ref{L:basics commutators iv}).    Therefore,
\[
[\calG,\calG] = [G_N,G_N] \times \prod_{\ell\nmid N} \SL_2(\ZZ_\ell) = [G,G],
\]
where the last equality uses that $N$ is the product of primes that divide the level of $[G,G]$.   In particular, $G$ is a normal subgroup of $\calG$ with $\calG/G$ abelian.  

\begin{lemma} \label{L:finally}
\begin{romanenum}
\item\label{L:finally a}
For an odd prime $\ell$, we have $G_\ell=\GL_2(\ZZ_\ell)$ if and only if $\calG_\ell=\GL_2(\ZZ_\ell)$.   
\item \label{L:finally b}
Suppose that $3$ divides $N$ and $\calG_3=\GL_2(\ZZ_3)$.  Then there is a surjective homomorphism $\psi\colon \calG_{N/3} \to \mathfrak{S}_3$ such that 
\[
\calG_{N} \subseteq \{ (a,b) \in \GL_2(\ZZ_3) \times \calG_{N/3} :  \varphi_3(a)=\psi(b) \},
\]
with $\varphi_3$ as in \S\ref{SS:3-adic computations}.
\item \label{L:finally c}
The levels of $[\calG,\calG]$ and $\calG \cap \SL_2(\Zhat)$ in $\SL_2(\Zhat)$ and the level of $\calG$ in $\GL_2(\Zhat)$ have the same odd prime divisors as $N$.
\end{romanenum}
\end{lemma}
\begin{proof}
We first prove (\ref{L:finally a}).    If $G_\ell=\GL_2(\ZZ_\ell)$, then we have $\calG_\ell=\GL_2(\ZZ_\ell)$ by the inclusion $G_\ell \subseteq \calG_\ell$.   We now may assume that $\calG_\ell=\GL_2(\ZZ_\ell)$.   Since $ \ZZ_\ell^\times \cdot G_\ell \supseteq \calG_\ell$, we have $G_\ell \supseteq [G_\ell,G_\ell] \supseteq [\calG_\ell,\calG_\ell] = [\GL_2(\ZZ_\ell),\GL_2(\ZZ_\ell)]$.   Since $\ell$ is odd, we have $G_\ell \supseteq \SL_2(\ZZ_\ell)$ by Lemma~\ref{L:basics commutators}(\ref{L:basics commutators iv}) and hence $G_\ell=\GL_2(\ZZ_\ell)$ since $G$ has full determinant.   \\

We now prove (\ref{L:finally b}).  Assume that $3$ divides $N$ and $\calG_3=\GL_2(\ZZ_3)$.  Set $m:=N/3$.  We can identify $\calG_N$ with a proper subgroup of $\GL_2(\ZZ_3) \times \calG_m$ so that the projection maps $\calG_N\to \GL_2(\ZZ_3)$ and $\calG_N\to \calG_m$ are surjective.   Let $B$ be the subgroup of $\GL_2(\ZZ_3)$ for which $B\times\{I\}$ is the kernel of $\calG\to \calG_m$.  The group $B$ is open and normal in $\GL_2(\ZZ_3)$.   By Goursat's lemma, there is a surjective homomorphism $\psi'\colon \calG_m \to \GL_2(\ZZ_3)/B$ such that 
\[
\calG_N = \{ (a,b) \in \GL_2(\ZZ_3)\times \calG_m :  a B=\psi'(b)\}.
\]

Suppose that $B\supseteq \SL_2(\ZZ_3)$.  Since $B$ contains the scalar matrices $\ZZ_3^\times\cdot I$, we deduce that $B \supseteq  \ZZ_3^\times \cdot\SL_2(\ZZ_3):=W$.  The group $W$ is the index $2$ subgroup of $\GL_2(\ZZ_3)$ consisting of matrices whose determinant is a square in $\ZZ_3^\times$.    Therefore, $\calG_3/B$ is cyclic of order $1$ or $2$.    Fix an element $c\in \calG_m$ that generates $\calG_m/\ker(\psi')$.      Take any $a,b\in \calG_3=\GL_2(\ZZ_3)$.  There are integers $i,j \in \{0,1\}$ such that $(a,c^i)$ and $(b,c^j)$ lie in $\calG_N \subseteq \calG_3 \times \calG_m$.   Taking the commutator of these two elements, we find that $(aba^{-1}b^{-1},I) \in [\calG_N,\calG_N] = [G_N,G_N]$.   Since the commutator subgroup of $\GL_2(\ZZ_3)$ is $\SL_2(\ZZ_3)$ by Lemma~\ref{L:basics commutators}(\ref{L:basics commutators iv}), we deduce that $[G_N,G_N]$ equals $\SL_2(\ZZ_3) \times [G_m,G_m]$.   However, this contradicts that $3$ divides $N$.  Therefore, $B\not\supseteq \SL_2(\ZZ_3 )$.

Since $B\not\supseteq \SL_2(\ZZ_3)$, Lemma~\ref{L:varphi3 condition} implies that $\varphi_3(B)=1$ and we thus have a surjective homomorphism $\varphi_3\colon \GL_2(\ZZ_3)/B \to \mathfrak{S}_3$.  Let $\psi\colon \calG_m \to \mathfrak{S}_3$ be the surjective homomorphism obtained by composing $\psi'$ and $\varphi_3$.   We have
\[
\calG_N \subseteq \{ (a,b)\in \GL_2(\ZZ_3) \times \calG_m : \varphi_3(a)=\psi(b) \}.
\]
This completes the proof of (\ref{L:finally b}).\\ 

Since $[G,G]=[\calG,\calG]$, part (\ref{L:finally c}) is immediate for the group $[\calG,\calG]$ from the definition of $N$.   From its construction, the level of $\calG$ cannot be divisible by primes $\ell\nmid N$.   Since the level of $\calG \cap \SL_2(\Zhat)$ divides the level of $\calG$, it suffices to prove (\ref{L:finally c}) for the group $\calG\cap \SL_2(\Zhat)$.

For a prime $\ell\geq 5$, Lemmas~\ref{L:prime divisors of level} and \ref{L:finally}(\ref{L:finally a}) imply that $\ell$ divides $N$ if and only if $\calG_\ell \neq \GL_2(\ZZ_\ell)$.   So the integer $N$ and the level of $\calG\cap \SL_2(\Zhat)$ have the same prime divisors $\ell \geq 5$ by Lemma~\ref{L:prime divisors of level}.     It thus remains to prove (\ref{L:finally c}) for the group $\calG\cap \SL_2(\Zhat)$ and the prime $\ell=3$.  

If $3\nmid N$, then from the definition of $\calG$ we find that the level of $\calG \cap \SL_2(\Zhat)$ is not divisible by $3$.   So we may assume that $3$ divides $N$; we need to prove that $3$ also divides the level of $\calG \cap \SL_2(\Zhat)$.     

Suppose that $\calG_3 \neq \GL_2(\ZZ_3)$.   We have $(\calG \cap \SL_2(\Zhat))_3 \subseteq \calG_3 \cap\ \SL_2(\ZZ_3) \subsetneq \SL_2(\ZZ_3)$ since $\calG_3$ has full determinant.   This shows that $\calG \cap \SL_2(\Zhat)$ has level divisible by $3$.  

Finally suppose that $3$ divides $N$ and $\calG_3=\GL_2(\ZZ_3)$.    So there is a surjective homomorphism $\psi\colon \calG_{N/3}\to\mathfrak{S}_3$ as in (\ref{L:finally b}).  We thus have 
\[
(\calG \cap \SL_2(\Zhat))_N = \calG_N \cap \SL_2(\ZZ_N) \subseteq \{ (a,b) \in \SL_2(\ZZ_3) \times (\calG_{N/3} \cap \SL_2(\ZZ_{N/3})) :  \varphi_3(a)=\psi(b) \}.
\]
Since $\varphi_3(\SL_2(\ZZ_3))=\mathfrak{A}_3$, we deduce that $\calG_N \cap \SL_2(\ZZ_N)$ is a proper subgroup of $\SL_2(\ZZ_N)$ whose level is divisible by $3$.  Therefore, $3$ divides the level of $\calG \cap \SL_2(\Zhat)$.
\end{proof}

\begin{proof}[Proof of Proposition~\ref{P:agreeable closure}]
We have already shown that the group $\calG$ defined by (\ref{E:agreeable closure}) contains $G$ and has the same commutator subgroup.   We have already observed that $\calG$ is open in $\GL_2(\Zhat)$, contains the scalars $\Zhat^\times\cdot I$ and satisfies $\det(\calG)=\Zhat^\times$.     Lemma~\ref{L:finally}(\ref{L:finally c}) says that the levels of $\calG$ and $[\calG,\calG]$ have the same odd prime divisors.  Therefore, $\calG$ is agreeable.

Now take any agreeable group $H\subseteq\GL_2(\Zhat)$ that satisfies $G\subseteq H$.  We have $[G,G] \subseteq [H,H]$, so the primes that divide the level of $[H,H]$ must also divide $N$.  Since $N$ is even and $H$ is agreeable, the primes dividing the level of $H$ all divide $N$. Therefore, $H=H_N \times \prod_{\ell \nmid N} \GL_2(\ZZ_\ell).$   Since $G$ and  $\Zhat^\times \cdot I$ are subgroups of $H$, we have $\ZZ_N^\times \cdot G_N \subseteq H_N$.  
 Therefore, $\calG\subseteq H$ from our definition  of $\calG$.   This proves that $\calG$ is the minimal agreeable group, with respect to inclusion, that contains $G$.
\end{proof}

\subsection{Maximal agreeable subgroups} \label{SS:maximal WAWA}

Fix an agreeable subgroup $\calG$ of $\GL_2(\Zhat)$.   Amongst the proper agreeable subgroups $G$ of $\calG$ that satisfy $G_\ell=\calG_\ell$ for all primes $\ell$, we define $\calM(\calG)$ to be the set of those that are maximal with respect to inclusion.  The set $\calM(\calG)$ is finite by Lemma~\ref{L:agreeable finiteness}.   

In this section, we give information about the groups in $\calM(\calG)$ and explain how they can be computed.  The results are straightforward but a little tedious due to the annoying nature of the primes $2$ and $3$.  Let $\varphi_3\colon \GL_2(\ZZ_3)\to \mathfrak{S}_3$ be the surjective homomorphism from \S\ref{SS:3-adic computations}.

\begin{lemma} \label{L:agreeable different primes}
Take any group $G\in \calM(\calG)$.  Let $M$ and $N$ be the product of the primes that divide the level of $G$ and $\calG$, respectively.
\begin{romanenum}
\item \label{L:calM description i}
The integers $M$ and $N$ have the same prime divisors $\ell \geq 5$ and $M$ is divisible by $N$.  In particular, we have $M \in \{N,2N,3N,6N\}$.
\item \label{L:calM description ii}
Suppose $M=2N$.   Then there are surjective homomorphisms $\alpha \colon \GL_2(\ZZ_2)\to \{\pm 1\}$ and $\beta\colon \calG_N \to \{\pm 1\}$ such that 
\[
G_M = \{(a,b) \in \GL_2(\ZZ_2)\times \calG_{N} : \alpha(a)=\beta(b)\}.
\]
The level of $G$ is $2$, $4$ or $8$ times the level of $\calG$.

\item \label{L:calM description iii}
Suppose $M=3N$.  Then there is a surjective homomorphism $\psi\colon \calG_N \to \mathfrak{S}_3$ such that 
\[
G_M = \{(a,b) \in \GL_2(\ZZ_3)\times \calG_{N} : \varphi_3(a)=\psi(b)\}.
\]
The level of $G$ is $3$ times the level of $\calG$.
\item \label{L:calM description iv}
Suppose $M=6N$.  Then there is a surjective homomorphism $\psi\colon \GL_2(\ZZ_2) \to \mathfrak{S}_3$ such that 
\[
G_M = \{(a,b) \in \GL_2(\ZZ_2)\times \GL_2(\ZZ_3) : \psi(a)=\varphi_3(b)\} \times \calG_N.
\]
The level of $G$ is $6$ times the level of $\calG$.
\end{romanenum}
\end{lemma}
\begin{proof}
Since $G_\ell=\calG_\ell$ for all primes $\ell$,  Lemma~\ref{L:prime divisors of level} implies that the levels of $G\cap \SL_2(\Zhat)$ and $\calG\cap\SL_2(\Zhat  )$ have the same prime divisors $\ell\geq 5$.   Since $G$ and $\calG$ are both agreeable, we deduce that the levels of $G$ and $\calG$ have the same prime divisors $\ell\geq 5$.   The inclusion $G\subseteq \calG$ implies that $N|M$.    We have $M\in \{N,2N,3N,6N\}$ since $M$ and $N$ are squarefree.   This proves (\ref{L:calM description i}). \\

Now suppose that $M =2N$ and hence $N$ is odd.
 
We claim that $G_N=\calG_N$.  
Define $W:=G_{N} \times \prod_{\ell\nmid N} \GL_2(\ZZ_\ell)$.   We have inclusions $G\subseteq W \subseteq \calG$ since $G\subseteq \calG$ and $N|M$. The group $W$ is open in $\GL_2(\Zhat)$ and satisfies $\det(W)=\Zhat^\times$ and $\Zhat^\times \cdot I \subseteq W$ since $G$ has these properties as well.  Intersecting with $\SL_2(\Zhat)$, we obtain inclusions $G\cap \SL_2(\Zhat) \subseteq W\cap\SL_2(\Zhat) \subseteq \calG \cap \SL_2(\Zhat)$.   Since $G$ and $\calG$ are agreeable group whose levels have the same odd primes divisors, these inclusions imply that $W$ is also agreeable.   For each prime $\ell$, we have $G_\ell=\calG_\ell$ since $G\in \calM(\calG)$ and hence $W_\ell=\calG_\ell$.  Since $G\subseteq W \subseteq \calG$ and $G\in \calM(\calG)$, we have $W=G$ or $W=\calG$.  We have $W=\calG$ since $2$ divides the level of $G$ but not the level of $W$.   Therefore, $G_N=W_N=\calG_N$ as claimed.

We have $G_N=\calG_N$ and $G_2=\calG_2=\GL_2(\ZZ_2)$, where the last equality uses that $2\nmid N$.  So we can identify $G_M$ with a proper subgroup of $\GL_2(\ZZ_2) \times \calG_{N}$ so that the projection maps $G\to \GL_2(\ZZ_2)$ and $G\to \calG_{N}$ are surjective.    Let $B$ be the subgroup of $\GL_2(\ZZ_2)$ for which $B\times\{I\}$ is the kernel of $G\to \calG_{N}$.  The group $B$ is open and normal in $\GL_2(\ZZ_2)$.   By Goursat's lemma, there is a surjective homomorphism $\psi'\colon \calG_{N} \to \GL_2(\ZZ_2)/B$ such that 
\[
G_M = \{ (a,b) \in \GL_2(\ZZ_2)\times \calG_{N} :  a B = \psi'(b)\}.
\]
We have $\GL_2(\ZZ_2)/B \neq 1$ since $2$ divides the level of $G$.  The maximal abelian quotient of $\GL_2(\ZZ_2)$ is pro-$2$ by Lemma~\ref{L:basics commutators}(\ref{L:basics commutators v}), so $\GL_2(\ZZ_2)/B$ has a quotient isomorphic to $\{\pm 1\}$.   So there are surjective homomorphisms $\alpha \colon \GL_2(\ZZ_2)\to \{\pm 1\}$ and $\beta\colon \calG_N \to \{\pm 1\}$ such that 
\[
G_M \subseteq \{(a,b) \in \GL_2(\ZZ_2)\times \calG_{N} : \alpha(a)=\beta(b)\}=:W_M.
\]
Define $W:=W_M \times \prod_{\ell\nmid M} \GL_2(\ZZ_\ell) \subseteq \GL_2(\Zhat)$.   We have inclusions $G\subseteq W \subsetneq \calG$, where $W\neq \calG$ since the level of $\calG$ is not divisible by $2$.  Since $G$ and $\calG$ are agreeable and $M=2N$, we find that $W$ is also agreeable.   We have $G=W$ since $G$ is in $\calM(\calG)$ and hence $G_M=W_M$. 

From our description of $G_M$, we find that the level of $G$ is equal to the product of the levels of $\ker \alpha$ and $\ker \beta$.   The level of $\ker \beta$ is the same as the level of $\calG_N$, and hence also $\calG$, since $N$ is odd.  So to complete the proof of (\ref{L:calM description ii}), it suffices to show that $\ker \alpha$ has level $2$, $4$ or $8$.  We have $\ZZ_2^\times \cdot I \subseteq \ker \alpha$ since $G$ is agreeable.  Therefore, $\ker\alpha \supseteq \ZZ_2^\times \cdot [\GL_2(\ZZ_2),\GL_2(\ZZ_2)]$ and this second group has level dividing $8$ by Lemmas~\ref{L:basics commutators}(\ref{L:basics commutators v}) and \ref{L:level by adjoining scalars}.  Since $\ker\alpha\neq 1$, it must has level $2$, $4$ or $8$.
\\

Now suppose that $M$ is $3N$ or $6N$, and hence $3\nmid N$.  Define $N_0 :=  M/3 \in \{N,2N\}$.
 
We claim that $G_{N_0}=\calG_{N_0}$.  
Define $W:=G_{N_0} \times \prod_{\ell\nmid N_0} \GL_2(\ZZ_\ell)$.   
  We have $G\subseteq W$ since $N_0|M$. The group $W$ is open in $\GL_2(\Zhat)$ and satisfies $\det(W)=\Zhat^\times$ and $\Zhat^\times \cdot I \subseteq W$ since $G$ has these properties as well.    Since $G\subseteq \calG$ and $N|N_0$,  we have inclusions $G \subseteq W \subseteq \calG$ and hence also $G \cap \SL_2(\Zhat) \subseteq W \cap \SL_2(\Zhat)$.   Since $G$ and $\calG$ are agreeable, we deduce that the levels of $W$ and $W\cap \SL_2(\Zhat)$ have the same prime divisors $\ell \geq 5$, i.e., the odd prime that divide $N$.   Since $3\nmid N_0$, the levels of $W$ and $W\cap \SL_2(\Zhat)$ are not divisible by $3$ and hence they have the same odd prime divisors.  Therefore, $W$ is agreeable.       For each prime $\ell$, we have $G_\ell=\calG_\ell$ since $G\in \calM(\calG)$ and hence $W_\ell=\calG_\ell$.  Since $G\subseteq W \subseteq \calG$ and $G\in \calM(\calG)$, we have $W=G$ or $W=\calG$.   We have $W=\calG$ since $3$ divides the level of $G$ but not the level of $W$.   So $G_{N_0}=W_{N_0}=\calG_{N_0}$ as claimed.

We have $G_3=\calG_3=\GL_2(\ZZ_3)$ and $G_{N_0}=\calG_{N_0}$.   So we can identify $G_M$ with a proper subgroup of $\GL_2(\ZZ_3) \times \calG_{N_0}$ so that the projection maps $G\to \GL_2(\ZZ_3)$ and $G\to \calG_{N_0}$ are surjective.   Let $B$ be the subgroup of $\GL_2(\ZZ_3)$ for which $B\times\{I\}$ is the kernel of $G\to \calG_{N_0}$.  The group $B$ is open and normal in $\GL_2(\ZZ_3)$.   By Goursat's lemma, there is a surjective homomorphism $\psi'\colon \calG_{N_0} \to \GL_2(\ZZ_3)/B$ such that 
\[
G_M = \{ (a,b) \in \GL_2(\ZZ_3)\times \calG_{N_0} :  a B = \psi'(b)\}.
\]
If $B \supseteq \SL_2(\ZZ_3)$, then $3$ does not divide the level of $G\cap\SL_2(\Zhat)$.  Since $3|M$ and $G$ is agreeable, the level of $G\cap\SL_2(\Zhat)$ must be divisible by $3$.  Therefore, $B \not \supseteq \SL_2(\ZZ_3)$.   By Lemma~\ref{L:varphi3 condition}, we have $\varphi_3(B)=1$.  Let $\psi\colon \calG_{N_0} \to \mathfrak{S}_3$ be the composition of $\psi'$ with $\varphi_3$.  Using that $\varphi_3(B)=1$, we deduce that 
\begin{align*}\label{E:GM proof new}
G_M \subseteq \{(a,b) \in \GL_2(\ZZ_3)\times \calG_{N_0} : \varphi_3(a)=\psi(b)\}=:W_M.
\end{align*}
Define the subgroup $W:=W_M \times \prod_{\ell\nmid M} \GL_2(\ZZ_\ell)$ of $\GL_2(\Zhat)$.  We have inclusions $G\subseteq W \subseteq \calG$ and hence also $G\cap\SL_2(\Zhat)\subseteq W \cap\SL_2(\Zhat)\subseteq \calG\cap\SL_2(\Zhat)$.   Since $G$ and $\calG$ are agreeable and their levels have the same prime divisors $\ell\geq 5$, the inclusions imply that $W$ and $W\cap\SL_2(\Zhat)$ have the same prime divisors $\ell\geq 5$.  Since the levels of $W$ and $W\cap\SL_2(\Zhat)$ are both divisible by $3$, we deduce that $W$ is agreeable.    Since $G\subseteq W \subseteq \calG$ and $G\in \calM(\calG)$, we have $W=G$ or $W=\calG$.   We have $W=G$ since $3$ divides the level of $W$ but not the level of $\calG$.   So 
\[
G_{M}=W_{M}= \{(a,b) \in \GL_2(\ZZ_3)\times \calG_{N_0} : \varphi_3(a)=\psi(b)\}.
\]

Consider the special case $M=3N$, equivalently $N_0=N$.  To prove part (\ref{L:calM description iii}), it remains to show that the level of $G$ is $3$ times the level of $\calG$.   From our description of $G_M$, it suffices to show that the level of $\ker \psi$ is equal to the level of $\calG_N$ (and hence also the level of $\calG$).  Let $B$ be the subgroup of $\calG_N$ consisting of matrices that are congruent modulo the level of $\calG$ to the identity matrix.    Since $B$ is a normal subgroup of $\calG_N$ that is the product of $p$-groups, $\psi(B)$ is a normal subgroup of $\mathfrak{S}_3$ whose $p$-Sylow subgroups are all normal.  Therefore, $\psi(B)=1$ and hence the level of $\ker \psi$ equals the level $\calG_N$.  \\  

Finally, we are left with the special case $M=6N$; equivalently, $N_0=2N$. Since $2$ does not divide the level of $\calG$, we have $\calG_{N_0} = \GL_2(\ZZ_2) \times \calG_N$.   Since $\psi$ is surjective, $\psi(\GL_2(\ZZ_2)\times\{I\})$ and $\psi(\{I\}\times \calG_N))$ are normal subgroup of $\mathfrak{S}_3$ that commute with each other and generate $\mathfrak{S}_3$; this can only happen if one of these groups is trivial and the other is $\mathfrak{S}_3$.  From our description of $G_M$, if $\psi$ is trivial on $\GL_2(\ZZ_2)\times\{I\}$, then the level of $G_M$ (and also $G$) is not divisible by $2$.  Since $2|M$, we deduce that $\psi$ is trivial on $\{I\}\times \calG_N$.      So there is a surjective homomorphism $\psi\colon \GL_2(\ZZ_2)\to \mathfrak{S}_3$ such that 
\[
G_6 = \{(a,b) \in \GL_2(\ZZ_3)\times \GL_2(\ZZ_2) : \varphi_3(a)=\psi(b)\}
\]
and $G_M=G_6 \times \calG_N$.  From this description of $G_M$, the level of $G$ is clearly $6$ times the level of $\calG$. This completes the proof of (\ref{L:calM description iv}).
\end{proof}

Lemma~\ref{L:agreeable different primes} gives a description of the groups $G\in \calM(\calG)$ for which the levels of $G$ and $\calG$ have different prime divisors (actually it gives a finite number of candidates for $G$ for which one can check directly if they are agreeable).   We now want to say something about the case where the prime divisors are the \emph{same}.  Let $N$ be the product the primes that divide the level of $\calG$.   

Take any proper divisor $1<d_1\leq \sqrt{N}$ of $N$.  Set $d_2=N/d_1$; it is relatively prime to $d_1$ and $d_1<d_2$.    We may identify $\calG_N$ with a subgroup of $\calG_{d_1} \times \calG_{d_2}$.   The kernel of the projection maps $\calG_N \to \calG_{d_2}$ and $\calG_N \to \calG_{d_1}$ are of the form $B_1 \times \{I\}$ and $\{I\}\times B_2$, respectively.

Let $\mathscr{C}_{d_1}$ be the set of pairs $(C_1,C_2)$ with open subgroups $C_1\subsetneq B_1$ and $C_2 \subsetneq B_2$ that satisfy the following conditions:
\begin{itemize}
\item
$C_i$ is a normal subgroup of $\calG_{d_i}$, 
\item
$C_i$ is maximal amongst the closed normal subgroup $H$ of $\calG_{d_i}$ that satisfy $H \subsetneq B_i$,
\item
$C_i$ contains the scalar matrices $\ZZ_{d_i}^\times \cdot I$,
\item
$B_1/C_1$ and $B_2/C_2$ are isomorphic abelian groups,
\item 
$\calG_{d_1}/C_1$ and $\calG_{d_2}/C_2$ are isomorphic groups.
\end{itemize}
For a pair $(C_1,C_2) \in \mathscr{C}_{d_1}$, let $\calA(C_1,C_2)$ denote an abelian group isomorphic to $B_1/C_1$ and $B_2/C_2$.

\begin{lemma}  \label{L:MG info}
Take any group $G$ in $\calM(\calG)$ and suppose that $N$ and the level of $G$ have the same prime divisors.  Then there is a proper divisor $1<d_1\leq \sqrt{N}$ of $N$ and a pair $(C_1,C_2)\in \mathscr{C}_{d_1}$ such that 
\[
 C_1 \times C_2\subseteq G_N \subsetneq \calG_N
 \]
 and $[\calG:G]=|\calA(C_1,C_2)|$.
\end{lemma}
\begin{proof}
We have $G=G_N \times \prod_{\ell \nmid N} \GL_2(\ZZ_\ell)$ and $\calG=\calG_N \times \prod_{\ell \nmid N} \GL_2(\ZZ_\ell)$ by our assumption on the level of $G$.   Therefore, $G_N \subsetneq \calG_N$ since $G\subsetneq \calG$ .  In particular, $[\calG:G]=[\calG_N:G_N]$.

Let $n\geq 1$ be the smallest divisor of $N$ for which $G_n \subsetneq \calG_n$.   The integer $n$ is composite since $G_\ell = \calG_\ell$ for all primes $\ell |N$.    Choose a proper divisor $1<d_1\leq \sqrt{n}$ of $n$ and set $d_2=N/d_1$.   By the minimality of $n$, we have $G_{d_1}=\calG_{d_1}$ and $G_{n/d_1}=\calG_{n/d_1}$.    Let $G'$ be the inverse image of $G_n$ under the projection map $\calG\to \calG_n$.     For any prime $\ell$, we have $G_\ell \subseteq G'_\ell \subseteq \calG_\ell$ and hence $G'_\ell=\calG_\ell$.    We have $G\subseteq G' \subsetneq \calG$ and hence the level of $G'$ has the same prime divisors as $N$.  Since $G$ and $\calG$ are agreeable, the inclusions $G\cap \SL_2(\Zhat)\subseteq G'\cap \SL_2(\Zhat) \subseteq \calG\cap \SL_2(\Zhat)$ imply that the odd primes dividing the level of $G'\cap \SL_2(\Zhat)$ are the same as those that divide $N$.  Therefore, $G'$ is agreeable.   We must have $G'=G$ since $G\in \calM(\calG)$.  Therefore, $G_{d_1}=\calG_{d_1}$ and $G_{d_2}=\calG_{d_2}$ where $N=d_1d_2$ with $1<d_1 \leq \sqrt{n}\leq \sqrt{N}$.

 We may identify $G_N$ and $\calG_N$ with subgroups of $\calG_{d_1} \times \calG_{d_2}$.    The kernel of the projection maps $G_N\to G_{d_2}=\calG_{d_2}$ and $G_N\to G_{d_1}=\calG_{d_1}$ are of the form $C_1 \times \{I\}$ and $\{I\}\times C_2$, respectively, where $C_i$ is a normal subgroup of $\calG_{d_i}$.    By Goursat's lemma, the image of the natural homomorphism
\[
G_N/(C_1\times C_2)\hookrightarrow (\calG_{d_1}\times \calG_{d_2})/(C_1\times C_2) =  \calG_{d_1}/C_1 \times \calG_{d_2}/C_2
\]
is the graph of an isomorphism $f\colon \calG_{d_1}/C_1 \xrightarrow{\sim} \calG_{d_2}/C_2$. 

The kernel of the projection maps $\calG_N\to \calG_{d_2}$ and $\calG_N\to \calG_{d_1}$ are of the form $B_1 \times \{I\}$ and $\{I\}\times B_2$, respectively, where $B_i$ is a normal subgroup of $\calG_{d_i}$.  By Goursat's lemma, the image of the natural homomorphism
\[
\calG_N/(B_1\times B_2)\hookrightarrow (\calG_{d_1}\times \calG_{d_2})/(B_1\times B_2) =  \calG_{d_1}/B_1 \times \calG_{d_2}/B_2
\]
is the graph of an isomorphism $\overline f\colon \calG_{d_1}/B_1 \xrightarrow{\sim} \calG_{d_2}/B_2$.    Using that $G_N\subseteq \calG_N$, we find that $C_i \subseteq B_i$ and that $f$ induces $\overline{f}$, i.e, composing $f$ with the quotient map $\calG_{d_2}/C_2 \to \calG_{d_2}/B_2$ gives a homomorphism that factors through $\overline{f}$.  In particular, $B_1/C_1$ and $B_2/C_2$ are isomorphic. 

For a fixed $i\in\{1,2\}$, take any normal subgroup $D_i$ of $\calG_{d_i}$ satisfying $C_i \subseteq D_i \subseteq B_i$.   Using the isomorphism $f$, we may assume that such a group $D_i$ exists for each $i\in \{1,2\}$ and that $f$ induces an isomorphism $D_1/C_1 \xrightarrow{\sim} D_2/C_2$.  Then $W:=\{(g_1,g_2) \in \calG_{d_1}\times \calG_{d_2} : f(g_1 D_1) = g_2 D_2\}$ satisfies $G_N\subseteq W \subseteq \calG_N$ and $[\calG_N:W]=[B_i:D_i]$ and $[W:G_N]=[D_i:C_i]$.   The subgroup $G_N$ of $\calG_N$ is maximal since $G$ is a maximal agreeable subgroup of $\calG$ and $G_N\neq \calG_N$, so $D_i=B_i$ or $D_i=C_i$.      Therefore, there are no normal subgroups $D_i$ of $\calG_{d_i}$ that satisfy $C_i \subsetneq D_i \subsetneq B_i$.  By considering the special case where $D_i=C_i$, we find that $[\calG:G]=[B_i:C_i]$.   We have $C_i\neq B_i$ since $G\neq \calG$.

Now suppose that the groups $B_i/C_i$ are nonabelian.   The group $[B_i,B_i]\cdot C_i$ is a normal subgroup of $\calG_{d_i}$ that lies between $C_i$ and $B_i$, so it is either $C_i$ or $B_i$.   Since $B_i/C_i$ is nonabelian by assumption, we deduce that $[B_i,B_i] \cdot C_i = B_i$ and hence $B_i/C_i$ is a nonabelian perfect group.   So there is a nonabelian simple group $S$ that is isomorphic to the quotient of an open normal subgroup of $\calG_{p_1}$ and $\calG_{p_2}$ for some prime $p_1|d_1$ and $p_2|d_2$.  Since $p_1\neq p_2$, Lemma~\ref{L:JH factors} implies that $\SL_2(\FF_{p_1})/\{\pm I\}$ and $\SL_2(\FF_{p_2})/\{\pm I\}$ are isomorphic which is impossible since they have different cardinalities.  So the groups $B_1/C_1$ and $B_2/C_2$ are both abelian

We have now verified that $(C_1,C_2)$ is in $\mathscr{C}_{d_1}$, $G_N\supseteq C_1\times C_2$, and $[\calG:G]=[\calG_N:G_N]=[B_i:C_i] = |\calA(C_1,C_2)|$.  
\end{proof}

\subsubsection{Computing $\calM(\calG)$} \label{SSS:computing MG}

Fix an agreeable subgroup $\calG$ of $\GL_2(\Zhat)$.  We now explain how to compute the groups in the set $\calM(\calG)$.   We assume that $\calG$ is given explicitly by its level and generators of its image modulo its level.  Let $N$ be the product of the primes that divide the level of $\calG$.  \\

Take any proper divisor $1<d_1\leq \sqrt{N}$ of $N$ and set $d_2=N/d_1$.    From $\calG$, we can compute the corresponding subgroups $B_1 \subseteq \calG_{d_1}$ and $B_2 \subseteq \calG_{d_2}$.  The group $\ZZ_{d_i}^\times \cdot [B_i,B_i]$ is open in $\GL_2(\ZZ_{d_i})$, normal in $\calG_{d_i}$ and is computable using \S\ref{SSS:computing commutator subgroups} and Lemma~\ref{L:level by adjoining scalars}.  For any pair $(C_1,C_2)\in \mathscr{C}_{d_1}$, we have 
\begin{align}\label{E:C inclusions}
\ZZ_{d_i}^\times \cdot [B_i,B_i]\subseteq C_i \subsetneq B_i 
\end{align}
for $i\in\{1,2\}$ since $B_i/C_i$ is abelian and $C_i$ contains the scalar matrices.   

We can compute the finite number of pairs of groups $(C_1,C_2)$ that satisfy (\ref{E:C inclusions}).  The group $C_i$ contains the scalars $\ZZ_{d_i}^\times$.  The group $C_i$ is normal in $B_i$ and $B_i/C_i$ is abelian since $[B_i,B_i] \subseteq C_i$.  For each pair, we can determine whether $(C_1,C_2)$ lies in $\mathscr{C}_{d_i}$.  So we can thus compute the set $\mathscr{C}_{d_i}$.     For each pair $(C_1,C_2) \in \mathscr{C}_{d_i}$, we can compute the groups $G$ satisfying $\calG \supseteq G \supseteq C_1\times C_2$ and $[\calG:G] = |\calA(C_1,C_2)|$, and determine which lie in $\calM(\calG)$.     

By varying over all proper divisors $1<d_1\leq \sqrt{N}$ of $N$,  Lemma~\ref{L:MG info} says that the above method will find all the groups $G$ in $\calM(\calG)$ for which the levels of $G$ and $\calG$ have the same prime divisors.

To compute the $G\in \calM(\calG)$ for which the levels of $G$ and $\calG$ have different prime divisors, we can check the finite number of possible groups $G$ arising from parts (\ref{L:calM description ii}), (\ref{L:calM description iii}) and (\ref{L:calM description iv}) of Lemma~\ref{L:agreeable different primes}.

\section{Constructing agreeable groups} \label{S:constructing agreeable}

The goal of this section is to prove Theorem~\ref{T:main agreeable} and to explain how to construct all the relevant groups and modular curves.  Set $\calL:=\{2,3,5,7,11,13,17,37\}$.

\subsection{$\ell$-adic case}  \label{SS:ell-adic case for scrA}

Fix any prime $\ell \in \calL$.   We want to construct an analogue of the set  $\scrA$ in Theorem~\ref{T:main agreeable} when we restrict to groups whose level is a power of $\ell$.  More precisely, we want a finite set $\scrS_\ell$ of agreeable subgroups that are pairwise non-conjugate in $\GL_2(\Zhat)$ and satisfy the following conditions:
\begin{itemize}
\item
For any group $\calG\in \scrS_\ell$, the level of $\calG$ is a power of $\ell$.
\item
Let $G$ be any agreeable group whose level is a power of $\ell$ and $X_G(\QQ)$ has a non-CM point.
\begin{itemize}
\item
If $X_{G}(\QQ)$ is infinite, then $G$ is conjugate in $\GL_2(\Zhat)$ to some group $\calG\in \scrS_\ell$.
\item
If $X_G(\QQ)$ is finite, then $G$ is conjugate in $\GL_2(\Zhat)$ to a subgroup of some $\calG\in \scrS_\ell$ with $X_\calG(\QQ)$ finite.
\end{itemize}
\item 
If $\calG\in \scrS_\ell$ is a group for which $X_{\calG}(\QQ)$ is finite, then $X_G(\QQ)$ is infinite for all agreeable groups $\calG\subsetneq G \subseteq \GL_2(\Zhat)$.
\item
If $\calG\in \scrS_\ell$ is a group for which $X_{\calG}(\QQ)$ has genus at most $1$, then $X_\calG(\QQ)$ has a non-CM point.
\end{itemize}

In \cite{MR3671434}, there is a classification of the open subgroups $\calG$ of $\GL_2(\Zhat)$ for which the level of $\calG$ is a power of $\ell$, $\det(\calG)=\Zhat^\times$, $-I \in \calG$, and $X_\calG(\QQ)$ is infinite.   In the recent work of Rouse, Sutherland and Zureick-Brown \cite{RSZ}, they consider the more general  problem of describing the $\ell$-adic images of elliptic curves over $\QQ$ and they give a complete description up to a few modular curves of high genus whose rational points they cannot determine.

By combining the results from \cite{MR3671434} and \cite{RSZ}, it is easy to produce a finite set $\scrS_\ell$ of agreeable groups that satisfy the desired properties.  Note that our groups need to be open in $\GL_2(\Zhat)$ with $\ell$-power level, have full determinant, and contain the scalars $\Zhat^\times$; the agreeable property is then immediate.

For the groups $\calG\in \scrS_\ell$ of genus at most $1$, we computed a model for $X_\calG$ and the morphism $\pi_\calG\colon X_\calG\to \PP^1_\QQ$ to the $j$-line using the methods outlined in \S\ref{SS:explicit low genus model}.   For these groups $\calG$ and a rational number $j\in \QQ$, it is a direct computation to check whether $j=\pi_\calG(P)$ for some $P\in X_\calG$ (it will reduce to finding roots of polynomials in $\QQ[x]$).

Now consider a group $\calG \in \scrS_\ell$ with $X_\calG$ having genus at least $2$.   In the cases where $X_\calG(\QQ)$ is known in \cite{RSZ}, they have computed the finite number of $j$-invariants of the non-CM points (if there are no non-CM points, then $\calG$ can be removed from the set $\scrA$).   In the few cases, where $X_\calG(\QQ)$ is not known, we can follow the method of \S11 in \cite{RSZ} to use Frobenius matrices to rule out rational points lying above any particular $j$-invariant $j\in \QQ$ of a non-CM elliptic curve (or in an incredibly unlikely case, you will find an unexpected rational point on their modular curves).   

\subsection{Constructions of $\scrA$}
\label{SS:construction of scrA}

For each prime $\ell\in \calL$, let $\scrS_\ell$ be a set of agreeable groups as in \S\ref{SS:ell-adic case for scrA}.   Let $\scrS$ be the (finite) set of subgroups of $\GL_2(\Zhat)$ that are of the form $\bigcap_{\ell\in \calL} H_\ell$ with $H_\ell\in \scrS_\ell$.   Every group $G \in \scrS$ is of the form $\prod_\ell G_\ell$ where $G_\ell$ is an open subgroup of $\GL_2(\ZZ_\ell)$ that satisfies $\det(G_\ell)=\ZZ_\ell^\times$ and $\ZZ_\ell^\times \cdot I \subseteq G_\ell$ for all $\ell$, and $G_\ell=\GL_2(\ZZ_\ell)$ for $\ell\notin \calL$.   So each group $G\in \scrS$ is agreeable and its level is not divisible by any prime $\ell\notin \calL$.

The groups in $\scrS$ are pairwise non-conjugate in $\GL_2(\Zhat)$.  We define a partial ordering on the set $\scrS$ by saying that $G\preceq \calG$ if $G$ is conjugate in $\GL_2(\Zhat)$ to a subgroup of $\calG$.  We now construct a subset $\scrB \subseteq \scrS$ by applying the following algorithm.   
\begin{itemize}
\item
Set $\scrB:=\emptyset$ and $\calS:=\scrS$.
\item
Choose a maximal element $G$ of $\calS$ with respect to $\preceq$ and remove it from the set $\calS$.   

When $X_G$ has genus at most $1$, we can determine whether the set $X_G(\QQ)$ is infinite and whether it has a non-CM point.  As in \S\ref{SS:explicit low genus model}, we can compute an explicit model of $X_G$ and compute its rational points.   When $X_G$ has a rational point, we can also compute the morphism $\pi_G$ down to the $j$-line and then determine if $X_G(\QQ)$ has a non-CM point.

When $X_G$ has genus at least $2$, then $X_G(\QQ)$ is finite by Faltings.

\item
If $X_G$ has genus at least $2$ or $X_G(\QQ)$ has a non-CM point, then adjoin $G$ to the set $\scrB$.

\item
If $X_G(\QQ)$ is finite, then remove from $\calS$ all the elements $\calG$ for which $\calG \preceq G$.    

\item 
If $\calS$ is nonempty, we go back to the second step where we chose another maximal element of $\calS$.  
\end{itemize}

The above process eventually halts since $\scrS$ is finite  and when it ends we will have our desired set $\scrB$.

\begin{lemma} \label{L:entanglement consistency}
Let $G$ be an agreeable subgroup of $\GL_2(\Zhat)$ such that $X_G(\QQ)$ has a non-CM point.  Assume further that $G=\prod_\ell G_\ell$, where $G_\ell$ is an open subgroup of $\GL_2(\ZZ_\ell)$ satisfying  $G_\ell=\GL_2(\ZZ_\ell)$ for all primes $\ell\notin \calL$.
\begin{romanenum}
\item \label{L:entanglement consistency i}
If $X_G(\QQ)$ is infinite, then $G$ is conjugate in $\GL_2(\Zhat)$ to some group $\calG\in \scrB$.
\item \label{L:entanglement consistency ii}
If $X_G(\QQ)$ is finite, then $G$ is conjugate in $\GL_2(\Zhat)$ to a subgroup of some $\calG\in \scrB$ with $X_\calG(\QQ)$ finite.
\end{romanenum}
\end{lemma}
\begin{proof}
For each $\ell \in \calL$, define $H_\ell := G_\ell \times \prod_{p \neq \ell} \GL_2(\ZZ_p)$.  The group $H_\ell$ is agreeable and $X_{H_\ell}(\QQ)$ has a non-CM point since $G\subseteq H_\ell$ and $X_G(\QQ)$ has a non-CM point. 

First suppose that $X_{H_\ell}(\QQ)$ is finite for some $\ell \in \calL$ and hence $X_G(\QQ)$ is finite.   By the properties of $\scrS_\ell$, there is a group $\calG \in \scrS_\ell$ such that $H_\ell$ is conjugate to a subgroup of $\calG$, $X_{\calG}(\QQ)$ is finite, and $X_{G'}(\QQ)$ is infinite for all $\calG \subsetneq G' \subseteq \GL_2(\Zhat)$.   We have $\calG \in \scrS$.    Since $X_{G'}(\QQ)$ is infinite for all $\calG \subsetneq G' \subseteq \GL_2(\Zhat)$, we find that $\calG$ is in the set $\scrB$. Therefore, $G$ is conjugate to a subgroup of $\calG \in \scrB$ and $X_{\calG}(\QQ)$ is finite.  

For the rest of the proof, we may assume that $X_{H_\ell}(\QQ)$ is infinite for all $\ell \in \calL$. After replacing $G$ by a conjugate in $\GL_2(\Zhat)$, we may assume by the properties of the sets $\scrS_\ell$ that $H_\ell \in \scrS_\ell$ for all $\ell \in \calL$.  Therefore, $G= \cap_{\ell \in \calL} H_\ell$ is an element of $\scrS$.

Suppose $X_G(\QQ)$ is infinite and hence $X_\calG(\QQ)$ is infinite for all groups $\calG\in \scrS$ with $G \preceq \calG$.  From our construction of $\scrB$ and $G\in \scrS$, we deduce that $G$ is an element of $\scrB$.  This completes the proof of part (\ref{L:entanglement consistency i}).

Finally suppose that $X_G(\QQ)$ is finite.    Let $\calG \in \scrS$ be a group that is maximal, with respect to $\preceq$, amongst the groups for which $X_\calG(\QQ)$ is finite and $G \preceq \calG$ (such a group exists since $G\in \scrS$ and $X_G(\QQ)$ is finite).   For every group $G' \in \scrS$ with $\calG \preceq G'$ and $\calG\neq G'$, $X_{G'}(\QQ)$ is infinite by the maximality of $\calG$.    From our construction of $\scrB$, we deduce that $\calG$ is an element of $\scrB$.    This completes the proof of part (\ref{L:entanglement consistency ii}).
\end{proof}

Take any group $H \in \scrB$; it is agreeable and satisfies $H=\prod_\ell H_\ell$.  From $H$, we now construct a finite set $\scrA_H$ of agreeable groups by applying the following algorithm.   
\begin{itemize}
\item
Set $\scrA_H:=\{H\}$ and $\calS:=\{H\}$.
\item
Choose a group $\calG$ in $\calS$ with minimal index $[\GL_2(\Zhat):\calG]$ and remove it from the set $\calS$.  
\item
We  compute the set $\calM(\calG)$, with notation as in \S\ref{SS:maximal WAWA}, using the method outlined in  \S\ref{SSS:computing MG}.   
\item 
We now let $G \subseteq \GL_2(\Zhat)$ vary over the elements of $\calM(\calG)$.

If $X_G$ has genus at least $2$ and $G$ is not conjugate in $\GL_2(\Zhat)$ to a group in $\scrA_H$, then we adjoin $G$ to the set $\scrA_H$.

Now suppose that $X_G$ has genus at most $1$  and $G$ is not conjugate in $\GL_2(\Zhat)$ to a group in $\scrA_H$.    As in \S\ref{SS:explicit low genus model}, we can compute an explicit model of $X_G$ and compute its rational points.   When $X_G$ has a rational point, we can also compute the morphism from $X_G$ to $X_H$ and then down to the $j$-line.  We can then determine if $X_G(\QQ)$ has a non-CM point.   If $X_G(\QQ)$ has infinitely many points, then we adjoin $G$ to the sets $\scrA_H$ and $\calS$.   If $X_G(\QQ)$ has finitely many points and a non-CM point, then we adjoin $G$ to the set $\scrA_H$. 

\item 
If $\calS$ is nonempty, we go back to the second step where we chose another group $\calG$ in $\calS$.
\end{itemize}
For each $\calG\in \scrA_H$, we will have $\calG_\ell=H_\ell$ for all primes $\ell$.  The above process halts by Lemma~\ref{L:agreeable finiteness}.  Finally, define the (finite) set 
\[
\scrA:=\bigcup_{H\in \scrB} \scrA_H
\]
of agreeable subgroups of $\GL_2(\Zhat)$.  From our descriptions of the sets $\scrB$ and $\scrA_H$, we have explained how the set $\scrA$ is computable.

For any group $\calG\in \scrA$, the level of $\calG$ is divisible only by primes in $\calL$.   For each $\calG\in \scrA$ with $X_\calG$ having genus at most $1$, $X_\calG(\QQ)$ has a non-CM point by construction; moreover, we will have found a model of $X_\calG$ and the morphism from $X_\calG$ to the $j$-line. 

\begin{lemma} \label{L:more entanglement consistency}
Take any agreeable group $G$ for which $X_G(\QQ)$ has a non-CM point.  Assume that the level of $G$ is not divisible by any prime $\ell\notin\calL$.
\begin{romanenum} 
\item \label{L:more entanglement consistency i}
If $X_G(\QQ)$ is infinite, 
then $G$ is conjugate in $\GL_2(\Zhat)$ to a group $\calG\in \scrA$.    
\item \label{L:more entanglement consistency ii}
If $X_G(\QQ)$ is finite, then $G$ is conjugate in $\GL_2(\Zhat)$ to a subgroup of some group $\calG \in \scrA$ with $X_\calG(\QQ)$ finite.
\end{romanenum}
\end{lemma}
\begin{proof}
 Define $H:=\prod_\ell G_\ell$, where $G_\ell \subseteq \GL_2(\ZZ_\ell)$ is the $\ell$-adic projection of $G$; it is an agreeable group whose level is only divisible by primes in $\calL$.    Since $G\subseteq H$, we deduce that $X_H(\QQ)$ has a non-CM point.  
 
Suppose that $X_H(\QQ)$ is finite.   By Lemma~\ref{L:entanglement consistency}(\ref{L:entanglement consistency ii}), $H$ is conjugate to a subgroup of some group $\calG \in \scrB$ with $X_\calG(\QQ)$ finite. In particular, $X_G(\QQ)$ is finite and $G$ is conjugate to a subgroup of $\calG$.   Part (\ref{L:more entanglement consistency ii}) in this case follows since $\calG \in \scrB$ and hence $\calG \in \scrA_\calG \subseteq \scrA$.
 
We may now assume that $X_H(\QQ)$ is infinite.   After first replacing $G$ by a conjugate in $\GL_2(\Zhat)$, we may assume by Lemma~\ref{L:entanglement consistency}(\ref{L:entanglement consistency i}) that $H \in \scrB$.    Let $\calG \in \scrA_H$ be a group with maximal index in $\GL_2(\Zhat)$ for which $G$ is conjugate in $\GL_2(\Zhat)$ to a subgroup of $\calG$ (it exists since $G\subseteq H$ and $H\in \scrA_H$).   If $G$ is conjugate to $\calG$ in $\GL_2(\Zhat)$, then the lemma is immediate since $\calG \in \scrA_H \subseteq \scrA$.  If $X_\calG(\QQ)$ is finite, and hence $X_G(\QQ)$ is finite as well, then the lemma also holds.

Finally, after replacing $G$ by a conjugate in $\GL_2(\Zhat)$, we are left to consider the case where $G$ is a proper subgroup of $\calG$ and $X_\calG(\QQ)$ is infinite.  Choose an agreeable group $M$ that is maximal amongst those that satisfy $G \subseteq M \subsetneq\calG$.   Since $\calG\in \scrA_H$, we have $G_\ell=\calG_\ell$ and hence $M_\ell =\calG_\ell$ for all primes $\ell$, where $\calG_\ell$ and $M_\ell$ are the $\ell$-adic projections of $\calG$ and $M$, respectively.    Since $M$ is a maximal agreeable subgroup of $\calG$, we have $M \in \calM(\calG)$.   Since $\calG\in \scrA_H$ and $X_\calG(\QQ)$ is infinite, we have $M\in \scrA_H \subseteq \scrA$.  Since $G\subseteq M \subsetneq \calG$, this contradicts the maximality in our choice of $\calG$.  Therefore, this final case does not occur.
\end{proof}

\subsection{Proof of Theorem~\ref{T:main agreeable}}

Let $\scrA$ be the finite set of agreeable groups constructed in \S\ref{SS:construction of scrA}.  

 By construction, our set $\scrA$ satisfies  (\ref{I:prop of goal a}).  The set $\scrA$ satisfies  (\ref{I:prop of goal b}) by Lemma~\ref{L:more entanglement consistency}.  The set $\scrA$ satisfies  (\ref{I:prop of goal d}) and (\ref{I:prop of goal e}) since in our construction we computed a model of any modular curve of genus at most $1$ that occurred and ignored those with no non-CM points.

For each group $\calG \in \scrA$ for which $X_\calG$ has genus at least $2$, we can check whether there is a group $G\in \scrA$ that is conjugate in $\GL_2(\Zhat)$ to a proper subgroup of $\calG$.  If any such group $G$ exists, we can remove it from the set $\scrA$.  Since condition (\ref{I:prop of goal b}) holds, our possibly smaller set $\scrA$ will satisfy condition (\ref{I:prop of goal c}).   

The set $\scrA$ we constructed may contain distinct subgroups of $\GL_2(\Zhat)$ that are conjugate.   To obtain our final set, we replace $\scrA$ by a maximal subset of groups that are non-conjugate in $\GL_2(\Zhat)$.  This does not affect the validity of conditions (\ref{I:prop of goal a})--(\ref{I:prop of goal e}).

It remains to explain how to implement (\ref{I:prop of goal f}).  Take any group $\calG \in \scrA$.   If $X_\calG$ has genus at most $1$, then from (\ref{I:prop of goal e}) we have an explicit model of $X_\calG$ and we have computed the map to the $j$-line.  So (\ref{I:prop of goal f}) can be done directly when $X_\calG$ has genus at most $1$.    We may now assume that $X_\calG$ has genus at least $2$.  

As mentioned in \S\ref{SS:ell-adic case for scrA}, \cite{RSZ} explains how to check whether a given $j\in \QQ-\{0,1728\}$ is in $\pi_G(X_G(\QQ))$ for an open subgroup $G$ of $\GL_2(\Zhat)$ with $\det(G)=\Zhat^\times$ and prime power level.  So we may further assume that the level of $\calG$ is not a prime power.  

For each prime $\ell$, let $\calG_\ell$ be the $\ell$-adic projection of $\calG$.   Define $H:=\prod_{\ell} \calG_\ell\subseteq \GL_2(\Zhat)$; it is agreeable and $\calG\subseteq H$.  Suppose that $\calG=H$, i.e., $\calG$ has no ``entanglements''.    Take any $j\in \QQ-\{0,1728\}$.   Using Proposition~\ref{P:revised moduli property} and $\calG=H$, we have  $j\in\pi_\calG(X_\calG(\QQ))$ if and only if $j\in \pi_{\calG'}(X_{\calG'}(\QQ))$ with $\calG':=\calG_\ell \times \prod_{p\neq \ell} \GL_2(\ZZ_p)$ for all $\ell$ dividing the level of $\calG$.   So in this case, (\ref{I:prop of goal f}) reduces to the prime power case covered in \cite{RSZ}.

We are left to consider the case where also  $\calG$ is a proper subgroup of $H=\prod_{\ell} \calG_\ell$.  Choose an agreeable group $\calG \subsetneq G \subseteq H$ with $[G:\calG]$ minimal.
 Since $X_\calG$ has genus at least $2$, and hence $X_\calG(\QQ)$ is finite by Faltings, (\ref{I:prop of goal c}) implies that $X_G(\QQ)$ is infinite.  By (\ref{I:prop of goal b}), $G$ is conjugate in $\GL_2(\Zhat)$ to a group in $\scrA$.  After replacing $\calG$ with a conjugate, we may assume without loss of generality that $G\in \scrA$.   
 
We are now in the setting of \S\ref{SS:alternate models}.  So we can compute a singular model of $X_\calG$ and, with respect to this model, we have the natural morphism $\pi\colon X_\calG \to X_G$.  Now take any $j\in \QQ-\{0,1728\}$.  By (\ref{I:prop of goal e}), we have a model of $X_G$ and can compute the morphism $\pi_G\colon X_G\to \PP^1_\QQ$.  So we can compute the finite set of $Q\in X_G(\QQ)$ with $\pi_G(Q)=j$.   For each such $Q$, we can then compute if there are any $P\in X_\calG(\QQ)$ with $\pi(P)=Q$.   This gives (\ref{I:prop of goal f}) since $\pi_G \circ \pi=\pi_\calG$.  (To deal with any singularities we can either resolve them or compute another model of the curve with different choices.)

\begin{remark}
In practice, one wants to conjugate the groups in $\scrA$ so that for each $\calG\in \scrA-\{\GL_2(\Zhat)\}$, there is another group $G\in \scrA$ with $\calG \subsetneq G$ and $[G: \calG]$ as small as possible.   We have $\pi_\calG=\pi_G \circ \pi$, where $\pi\colon X_\calG\to X_G$ is the natural morphism of degree $[G:\calG]$.  Repeating, we can hope to express $\pi_\calG$ as the composition of morphisms of relatively small degree.  Having such an expression makes it easier to compute the set of $P\in X_\calG(\QQ)$ with $\pi_\calG(P)$ equal to a fixed $j\in \QQ-\{0,1728\}$.
\end{remark}

\section{Finding the agreeable closure of the image of Galois} \label{S:finding the agreeable closure}

Fix a non-CM elliptic curve $E$ defined over $\QQ$.    We have defined a representation 
\[
\rho_E^*\colon \Gal_\QQ \to \GL_2(\Zhat).
\] 
From Serre (Theorem~\ref{T:Serre 1972}), we know that the image $G_E$ of $\rho_E^*$ is an open subgroup of $\GL_2(\Zhat)$.   We have $\det(G_E)=\Zhat^\times$.  By Proposition~\ref{P:agreeable closure}, there is a minimal agreeable subgroup $\calG_E$ of $\GL_2(\Zhat)$ for which $G_E\subseteq \calG_E$, i.e., the agreeable closure of $G_E$.  The group $G_E$, and hence also $\calG_E$, is uniquely determined up to conjugacy in $\GL_2(\Zhat)$.   In this section, we describe how to compute the group $\calG_E$.   

The hardest steps have already been completed by the proof of Theorem~\ref{T:main agreeable} in \S\ref{S:constructing agreeable}.   Let $\scrA$ be a finite set of agreeable subgroups of $\GL_2(\Zhat)$ as in Theorem~\ref{T:main agreeable}.  Note that the set $\scrA$ can be computed once and does not depend on the curve $E/\QQ$.      

For our algorithm, we want as input an elliptic curve $E/\QQ$ given by a Weierstrass model with output the level $N$ of $\calG_E$ along with a set of generators of the image of $\calG_E$ modulo $N$.  This algorithm has been implemented \cite{github} assuming that Conjecture~\ref{C:Serre question} holds for $E$ and that $E$ does not give rise to unknown rational points on a few explicit high genus modular curves (these conditions are verified during the algorithm).   For any remaining cases, which involve only a finite number of $j$-invariants or a counterexample to Conjecture~\ref{C:Serre question}, we can compute the group $\calG_E$ using ad hoc techniques like those in \S\ref{SS:closure in exceptional cases}.\\

Once we know $\calG_E$, we can compute the commutator subgroup $[\calG_E,\calG_E]$, cf.~\S\ref{SSS:computing commutator subgroups},  which is an open subgroup of $\SL_2(\Zhat)$.    By Proposition~\ref{P:agreeable closure} and Lemma~\ref{L:abelian quotients and same commutators}, we have  $G_E \cap\SL_2(\Zhat)=[\calG_E,\calG_E]$
and 
\[
[\GL_2(\Zhat): G_E] = [\SL_2(\Zhat): [\calG_E,\calG_E]].
\]
In particular, we have an algorithm to compute the group $G_E \cap\SL_2(\Zhat)$, up to conjugacy in $\GL_2(\Zhat)$, and to compute the index $[\GL_2(\Zhat):G_E] = [\GL_2(\Zhat):\rho_E(\Gal_\QQ)]$ occurring in Serre's open image theorem.

\begin{lemma}  \label{L:level of agreeable and primes}
Let $E/\QQ$ be a non-CM elliptic curve over $\QQ$ for which Conjecture~\ref{C:Serre question} holds.   Then the level of $\calG_E$ in $\GL_2(\Zhat)$ is not divisible by any prime $\ell \notin \calL:=\{2,3,5,7,11,13,17,37\}$.
\end{lemma}
\begin{proof}
By hypothesis on $E/\QQ$, we have $\rho_{E,\ell}(\Gal_\QQ)=\GL_2(\ZZ/\ell\ZZ)$ for all primes $\ell \notin \calL$.   By Lemma~\ref{L:modell to elladic}, we have $\rho_{E,\ell^\infty}(\Gal_\QQ)=\GL_2(\ZZ_\ell)$ for all primes $\ell \notin \calL$.   So for each prime $\ell\notin \calL$, we have $\GL_2(\ZZ_\ell)=(G_E)_\ell \subseteq (\calG_E)_\ell$ and hence $(\calG_E)_\ell=\GL_2(\ZZ_\ell)$, where $(G_E)_\ell$ and $(\calG_E)_\ell$ are the $\ell$-adic projections of $G_E$ and $\calG_E$, respectively.   Since $\calG_E$ is agreeable, its level is not divisible by any prime $\ell \notin \calL$ by Lemma~\ref{L:prime divisors of level}.
\end{proof}

\subsection{Finding the agreeable closure in most cases} 
\label{SS:closure in most cases}

Throughout \S\ref{SS:closure in most cases}, we assume that $j_E \notin \pi_\calG(X_\calG(\QQ))$ for all groups $\calG\in \scrA$ for which $X_{\calG}$ is finite.   This can be verified for the given $j_E$ using Theorem~\ref{T:main agreeable}(\ref{I:prop of goal f}).   This excludes a finite number of $j$-invariants from consideration that we will describe how to deal with in \S\ref{SS:closure in exceptional cases}.

We can apply the algorithm in \cite{surjectivityalgorithm} to compute the finite set of primes $\ell>13$ for which $\rho_{E,\ell}(\Gal_\QQ)\neq \GL_2(\ZZ/\ell\ZZ)$.   For the rest of \S\ref{SS:closure in most cases}, we shall further assume that Conjecture~\ref{C:Serre question} holds for $E$, i.e., $\rho_{E,\ell}(\Gal_\QQ)= \GL_2(\ZZ/\ell\ZZ)$ for all primes $\ell>13$ with   
\[
(\ell,j_E) \in \big\{\, (17, -17^2 \!\cdot\! 101^3/2), \,(17,-17\!\cdot\! 373^3/2^{17}),\, (37,-7\!\cdot\! 11^3),\, (37,-7\!\cdot\! 137^3\!\cdot\! 2083^3) \,\big\}.
\]
Any potential counterexample to Conjecture~\ref{C:Serre question} can be dealt with using the methods from \S\ref{SS:closure in exceptional cases}.\\

By Lemma~\ref{L:level of agreeable and primes}, the level of $\calG_E$ is not divisible by any primes $\ell\notin \calL:=\{2,3,5,7,11,13,17,37\}$.   Since $G_E \subseteq \calG_E$, we know that $X_{\calG_E}$ has a rational non-CM point; in particular, there is a $P\in X_{\calG_E}(\QQ)$ such that $\pi_{\calG_E}(P)=j_E$.

If $X_{\calG_E}(\QQ)$ is finite, then $\calG_E$ is conjugate in $\GL_2(\Zhat)$ to a subgroup of some $\calG \in \scrA$ for which $X_\calG(\QQ)$ is finite by Theorem~\ref{T:main agreeable}(\ref{I:prop of goal b}).   Therefore, $X_{\calG_E}(\QQ)$ is infinite by our assumption on $E$.  By Theorem~\ref{T:main agreeable}(\ref{I:prop of goal b}), $\calG_E$ is conjugate in $\GL_2(\Zhat)$ to a unique group $\calG\in \scrA$.   Since $\calG_E$ is the agreeable closure of $G_E$, we can characterize $\calG$ as the group in $\scrA$ with maximal index in $\GL_2(\Zhat)$ for which $j_E \in \pi_\calG(X_\calG(\QQ))$.  So by making use of Theorem~\ref{T:main agreeable}(\ref{I:prop of goal f}), we can find $\calG$ which gives the group $\calG_E$ up to conjugacy in $\GL_2(\Zhat)$.

\subsection{Finding the agreeable closure in exceptional cases} 
\label{SS:closure in exceptional cases}

Now suppose that $j_E \in \pi_\calG(X_\calG(\QQ))$ for some group $\calG\in \scrA$ with $X_\calG(\QQ)$ finite.     There are only finitely many $j$-invariants $j_E$ that can arise in this way but they are difficult to determine since finding rational points on high genus curves can be very challenging.   In \S\ref{SS:closure in exceptional cases}, we shall make use of the notation from \S\ref{projection notation}.

However, we can compute the agreeable closure $\calG_E$ for any such $E/\QQ$ that arises.   The group $\calG_E$, up to conjugacy in $\GL_2(\Zhat)$, depends only on $j_E$.  So far, we have found 81 exceptional $j$-invariants that needed to be considered specially.  Any others that may arise can be dealt with in a similar manner.   For simplicity, let us assume that Conjecture~\ref{C:Serre question} holds for $E/\QQ$; we can handle any counterexamples that occur by similar techniques.
\\

By assumption, there is an agreeable group $\calG\in \scrA$ such that $j_E \in \pi_{\calG}(X_{\calG}(\QQ))$ and $X_\calG$ has only finite many rational points.    After conjugating $\rho_E^*$, we may assume that $\calG_E \subseteq \calG$.

We claim that we can compute the $\ell$-adic projection $(\calG_E)_\ell$, up to conjugacy in $\GL_2(\ZZ_\ell)$, for all primes $\ell$.  Using the explicit definition (\ref{E:agreeable closure}) of the agreeable closure, this is equivalent to computing $\ZZ_\ell^\times \cdot \rho_{E,\ell^\infty}^*(\Gal_\QQ)$ for all primes $\ell$.  When $\rho_{E,\ell}(\Gal_\QQ)=\GL_2(\ZZ/\ell\ZZ)$ and $\ell\geq 5$, we have $\rho_{E,\ell^\infty}^*(\Gal_\QQ)=\GL_2(\ZZ_\ell)$ by Lemma~\ref{L:modell to elladic}, and hence $(\calG_E)_\ell=\GL_2(\ZZ_\ell)$.   For the finite number of $\ell$ with $\ell\leq 3$ or $\rho_{E,\ell}(\Gal_\QQ)\neq\GL_2(\ZZ/\ell\ZZ)$, we can compute $\rho_{E,\ell^\infty}(\Gal_\QQ)$, and hence also $(\calG_E)_\ell$, by using the results from \cite{RSZ}.

Using $\calG_E \subseteq \calG$,  we can then check whether or not $(\calG_E)_\ell=\calG_\ell$ for all primes $\ell$.    Suppose that $(\calG_E)_\ell\subsetneq\calG_\ell$ for some prime $\ell$.  This will produce an explicit proper agreeable group $\calG' \subsetneq \calG$ for which $j_E\in \pi_{\calG'}(X_{\calG'}(\QQ))$.    We can replace $\calG$ by $\calG'$ and then repeat the above process.   We will eventually end up with an explicit agreeable subgroup $\calG$ of $\GL_2(\Zhat)$ with $X_\calG(\QQ)$ finite, $j_E\in \pi_\calG(X_\calG(\QQ))$, and $\calG_\ell$ and $(\calG_E)_\ell$ conjugate in $\GL_2(\ZZ_\ell)$ for all primes $\ell$.\\

As in \S\ref{SS:maximal WAWA}, we can define $\calM(\calG)$ to be the set of maximal proper agreeable subgroups $G$ of $\calG$ that satisfy $G_\ell=\calG_\ell$ for all primes $\ell$.  The set $\calM(\calG)$ is finite and computable, cf.~\ref{SSS:computing MG}.   

Take any of the groups $G\in \calM(\calG)$ up to conjugacy in $\GL_2(\Zhat)$.  We want to know whether or not $j_E \in \pi_G(X_G(\QQ))$.  The direct approach is to check after first computing a model for the curve $X_G$ and the morphism $\pi_G$ to the $j$-line; this is doable using the techniques from \S\ref{S:computing modular forms}.   However, these computations seem excessive to deal with a single $j$-invariant $j_E$.  We now explain our ad hoc computations with traces of Frobenius, which can be found in \cite{github}, that allows to verify if $j_E \in \pi_G(X_G(\QQ))$ holds without computing any further modular curves.  

Let $N$ be the level of $G$ and let $M$ be the product of $N$ with the bad primes of $E/\QQ$.   Take any prime $p\nmid M$ and let $a_p(E)$ be the trace of Frobenius of the reduction of $E$ mod $p$, i.e., $a_p(E)=p+1-|E(\FF_p)|$, where we are using a model of $E$ with good reduction at $p$.  The representation $\rho_{E,N}^*$ has good reduction at $p$ and $\rho_{E,N}^*(\Frob_p)^{-1} \in\rho_{E,N}^*(\Gal_\QQ)\subseteq \GL_2(\ZZ/N\ZZ)$ has trace $a_p(E)$ and determinant $p$ modulo $N$.  Let $\xi_p$ be the pair $(a_p(E),p)$ modulo $N$.   Now suppose we found a prime $p\nmid M$ such that $(\tr(g),\det(g))\neq \xi_p$ for all $g$ in the image of $G$ modulo $N$.  In particular, the group $\rho_{E,N}^*(\Gal_\QQ)$ contains an element that is not conjugate in $\GL_2(\ZZ/N\ZZ)$ to any element of the image of $G$ modulo $N$.   Therefore, $\rho_{E}^*(\Gal_\QQ)$ is not conjugate in $\GL_2(\Zhat)$ to a subgroup of $G$ and thus $j_E\notin \pi_G(X_G(\QQ))$.   So by computing $\xi_p$ for many primes $p\nmid M$, we hope to be able to prove that $j_E\notin \pi_G(X_G(\QQ))$.  \\    

Now suppose that after computing $\xi_p$ for many primes $p\nmid M$, we are unable to conclude that $j_E\notin \pi_G(X_G(\QQ))$.   In all the exceptional cases we considered, we then had $[\calG:G]=2$.  The group $G$ is normal in $\calG$ and we obtain a quadratic character 
\[
\chi\colon \Gal_\QQ\xrightarrow{\rho_E^*} \calG_E \hookrightarrow \calG \to \calG/G \cong \{\pm 1\}.
\]
For $\sigma\in \Gal_\QQ$, $\chi(\sigma)$ depends only on $\rho_{E,N}^*(\sigma)$ since $G$ has level $N$.   Therefore, $\chi$ is unramified at all primes $p\nmid M$.    In particular, there are only finite many possible quadratic characters that could arise as $\chi$. Take any prime $p\nmid M$.  If $(\tr(g),\det(g))\neq \xi_p$ for all $g$ in the image of $\calG-G$ modulo $N$, then $\chi(\Frob_p)=1$.    

Suppose that we are able to show that $\chi(\Frob_p)=1$ for enough primes $p\nmid M$, to rule out all possibilities for the characters $\chi$ except $\chi=1$.   Since $\chi=1$, we deduce that $\rho_E^*(\Gal_\QQ)$ is conjugate in $\GL_2(\Zhat)$ to subgroup $G$.   Therefore, $j_E \in \pi_G(X_G(\QQ))$ and we can replace $\calG$ by $G$ and repeat the above process.\\

In the remaining cases, which occurred for three of our $j$-invariants, we are left with a unique $G\in \calM(\calG)$, up to conjugacy in $\GL_2(\Zhat)$, for which we were not yet able to determine whether or not $j_E$ lies in $\pi_G(X_G(\QQ))$. In these remaining cases, we found that for some prime $p\in \{3,5\}$, we have $\calG_{2p}=\calG_2\times \calG_p$ and $G_{2p} \subsetneq G_2 \times G_p = \calG_{2p}$.  We computed the division polynomials for $E$ at $2$ and $p$, factored them into irreducible polynomials over $\QQ$, and computing the discriminants of these polynomials.   From this information, we found a quadratic extension $K/\QQ$ with $K\subseteq \QQ(E[2])$, $K\subseteq \QQ(E[p])$ and $K$ not equal to $\QQ(\sqrt{(-1)^{(p-1)/2} p })$.  Since $\calG_{2p}=\calG_2\times \calG_p$, this proves that $\calG$ is not the agreeable closure of $\rho^*_E(\Gal_\QQ)$.   Therefore, $j_E \in \pi_G(X_G(\QQ))$ and we can replace $\calG$ by $G$ and repeat the above process.

In all our cases, the above arguments eventually lead to an explicit minimal agreeable group $\calG$ with $j_E \in \pi_\calG(X_\calG(\QQ))$ and hence we can take $\calG_E=\calG$.

\section{Abelian quotients} \label{S:finding image given agreeable closure}

Let $\calG$ be an open subgroup of $\GL_2(\Zhat)$ satisfying $\det(\calG)=\Zhat^\times$ and $-I\in \calG$.   Fix an open subgroup $G$ of $\calG$ satisfying $\det(G)=\Zhat^\times$ such that $G$ is a normal subgroup of $\calG$ with $\calG/G$ abelian.
From the openness, the abelian group $\calG/G$ is also finite.    Fix an integer $N\geq 3$ divisible by the level of $G$ and let $\bbar{G}$ and $\bbar{\calG}$ be the images of $G$ and $\calG$, respectively, in $\GL_2(\ZZ/N\ZZ)$.  Reduction modulo $N$ induces an isomorphism $\calG/G \xrightarrow{\sim} \bbar{\calG}/\bbar{G}$ that we will use as an identification. 

\subsection{Setup} \label{SS:general setup}

\subsubsection{Some representations}   \label{SSS:setup reps}

With notation as in \S\ref{SS:specializations}, we have a surjective homomorphism
\[
\varrho:=\varrho_{\scrE_\calG,N}^*\colon \pi_1(U_\calG,\bbar{\eta}) \to \bbar \calG.
\]
Recall that $\varrho$ depends on a choice of a nonzero modular form $f_0$ in $M_3(\Gamma(N),\QQ(\zeta_N))$ and a choice of $\beta$ in a field extension of $\calF_N$ that satisfies $\beta^2=j\cdot f_0^2/E_6$.   When $-I \notin G$, we shall further assume that $f_0$ is chosen to be a nonzero element of $M_{3,\bbar{G}}$; the existence of such an $f_0$ is a consequence of Lemma~\ref{L:weight 3 existence}.     By composing $\varrho$ with the quotient map $\bbar\calG \to \bbar\calG/\bbar G = \calG/G$, we obtain a surjective homomorphism
\[
\alpha \colon \pi_1(U_\calG,\bbar{\eta}) \to \calG/G.
\]   
Let $\phi\colon Y \to U_\calG$ be the \'etale cover corresponding to $\alpha$.   The cover $\phi$ is Galois with Galois group $\calG/G$.  When $-I \in G$, we will have $Y=U_G$ and $\phi$ will be the natural morphism.\\

Now consider the representation $\rho_{\calE,N}^* \colon \Gal_{\QQ(j)} \to \GL_2(\ZZ/N\ZZ)$ as in \S\ref{SS:representations modular explicit}.  By Lemma~\ref{L:FN revisited}, $\rho_{\calE,N}^*$ is surjective and factors through an isomorphism $\Gal(\calF_N(\beta)/\QQ(j)) \xrightarrow{\sim} \GL_2(\ZZ/N\ZZ)$.  We let $\GL_2(\ZZ/N\ZZ)$ act on the right of $\calF_N(\beta)$ via $f * \rho_{\scrE,N}^*(\sigma)^{-1} := \sigma(f)$ for all $f\in \calF_N(\beta)$.   By Lemma~\ref{L:FN revisited}(\ref{L:FN revisited iii}), this extends our earlier right action of $\GL_2(\ZZ/N\ZZ)$ on $\calF_N$.   Since $\beta\notin \calF_N$, we have $\beta * (-I) = -\beta$ by Lemma~\ref{L:FN revisited}(\ref{L:FN revisited vi}).  Define the subfield  $L:=\calF_N(\beta)^{\bbar{G}}$ of  $\calF_N(\beta)$.   The representations $\rho_{\calE,N}^*$ induces an isomorphism $\Gal(L/\QQ(X_\calG)) \xrightarrow{\sim} \bbar{\calG}/\bbar{G} = \calG/G$. \\
 
Moreover, the representation $\varrho$ in \S\ref{SS:specializations} is constructed so that the specialization at the generic point of $U_G$ gives the representation $\Gal_{\QQ(X_G)}\to \bbar\calG \subseteq \bbar\calG \subseteq \GL_2(\ZZ/N\ZZ)$ that is the restriction of the representation $\rho_{\calE,N}^*$.   We can thus identify $L$ with the function field of $Y$ and the field extension $L/\QQ(X_\calG)$ corresponds to the morphism $\phi\colon Y\to U_\calG$.   Note that the curve $Y$ is geometrically irreducible since $\QQ$ is algebraically closed in $L$ by Lemma~\ref{L:FN revisited}(\ref{L:FN revisited ii}) and our assumption $\det(G)=\Zhat^\times$. 

\subsubsection{Specializations} \label{SSS:general setup - specializations}
Take any point $u\in U_\calG(\QQ)$.  Specializing $\varrho$ at $u$ gives a homomorphism $\varrho_{u}\colon \Gal_\QQ\to \bbar\calG$.     The representations $ \Gal_\QQ\xrightarrow{\varrho_u} \bbar\calG \subseteq \GL_2(\ZZ/N\ZZ)$ and $\rho_{(\scrE_\calG)_u,N}^*\colon \Gal_\QQ \to \GL_2(\ZZ/N\ZZ)$ are isomorphic, where $(\scrE_\calG)_u$ is the elliptic curve over $\QQ$ defined by the Weierstrass equation $y^2 =  x^3 -  27  j (j-1728)  \cdot  x +54  j (j-1728)^2$ with $j:=\pi_\calG(u)$.   Let 
\[
\alpha_u\colon \Gal_\QQ\to \calG/ G
\] 
be the homomorphism that is the specialization of $\alpha$ at $u$.  Equivalently, $\alpha_u$ is the composition of $\varrho_u$ with the quotient map $\bbar\calG\to \bbar\calG/\bbar G= \calG/G$.     Since $\calG/ G$ is abelian, $\alpha_u$ is uniquely determined (while the specialization $\varrho_u$ is only uniquely determined up to conjugation by an element in $\bbar\calG$).   

Let $\phi^{-1}(u)\subseteq Y$ be the fiber of $\phi$ over $u$.   The action of $\calG/G$ on the $\Qbar$-points of $\phi^{-1}(u)$ is simply transitive since $\phi$ is \'etale.   The group $\Gal_\QQ$ acts on the $\Qbar$-points of $\phi^{-1}(u)$ since $\phi$ and $u$ are defined over $\QQ$.  These actions of $\calG/G$ and $\Gal_\QQ$ commute. For a fixed $y_0\in Y(\Qbar)$ with $\phi(y_0)=u$, we have
\begin{align} \label{E:characterization of alphau}
\sigma(y_0)= \alpha_u(\sigma) \cdot y_0
\end{align}
for all $\sigma\in \Gal_\QQ$; note that this does not depend on the choice of $y_0$ since $\calG/ G$ is abelian and it commutes with the Galois action.   In particular, the expression (\ref{E:characterization of alphau}) determines $\alpha_u(\sigma)$.

\subsection{Defining $\alpha_E$ using modular curves} \label{SS:finding image setup}

Consider a non-CM elliptic curve $E/\QQ$ for which $G_E=\rho_E^*(\Gal_\QQ)$ is conjugate in $\GL_2(\Zhat)$ to a subgroup of $\calG$.   By Proposition~\ref{P:intial moduli property}, there is a point $u\in U_\calG(\QQ)$ such that $\pi_\calG(u)=j_E$.   

Since any two non-CM elliptic curves with the same $j$-invariant are quadratic twists of each other, there is a unique squarefree integer $d$ such that $E/\QQ$ is isomorphic to the quadratic twist of $E':=(\scrE_\calG)_u$ by $d$, where $(\scrE_\calG)_u$ is the elliptic curve over $\QQ$ defined as in \S\ref{SSS:general setup - specializations}.   Let $\chi_d\colon \Gal_\QQ \to \{\pm 1\}$ be the homomorphism that factors through $\Gal(\QQ(\sqrt{d})/\QQ) \hookrightarrow \{\pm 1\}$.   Define the homomorphism 
\[
\alpha_E \colon \Gal_\QQ \to \calG/G
\]
by $\alpha_E=\chi_d \cdot \alpha_u$ with $\alpha_u$ as in \S\ref{SSS:general setup - specializations}.   The following lemma shows that this definition of $\alpha_E$ is consistent with our earlier definition in \S\ref{SS:finding Galois}.

\begin{lemma} \label{L:alphaE properties}
After replacing $\rho_E^*$ by an isomorphic representation, we will have $\rho_E^*(\Gal_\QQ) \subseteq \calG$ and the composition of $\rho_E^*\colon \Gal_\QQ\to \calG$ with the quotient map $\calG\to\calG/G$ is $\alpha_E$.
\end{lemma}
\begin{proof}
Since $E$ is a quadratic twist of $E'$ by $d$, we can choose bases so that $\rho_{E',N}^*=\varrho_u$ and $\rho_E^* = \chi_d\cdot \rho_{E'}^*$.  In particular, $\rho_{E,N}^*(\Gal_\QQ) \subseteq \pm \varrho_u(\Gal_\QQ) \subseteq \bbar\calG$ and hence $\rho_{E}^*(\Gal_\QQ)\subseteq\calG$ since the level of $\calG$ divides $N$.  Take any $\sigma\in \Gal_\QQ$.   We have $\rho_{E,N}^*(\sigma) \cdot \bbar G = \chi_d(\sigma)\cdot \varrho_u(\sigma) \cdot \bbar G$ and hence $\rho_E^*(\sigma) \cdot G=\chi_d(\sigma) \cdot \alpha_u(\sigma) = \alpha_E(\sigma)$.  
\end{proof}

\begin{lemma} \label{L:alphaE properties0}
The homomorphism $\alpha_E$ is unramified at all primes $p\nmid N$ for which $E$ has good reduction.
\end{lemma}
\begin{proof}
Using our isomorphism $\calG/G = \bbar{\calG}/\bbar{G}$, we could also view $\alpha_E$ as the composition of $\rho_{E,N}^*$ with the quotient map $\bbar{\calG}\to \bbar{\calG}/\bbar{G}=\calG/G$.   The lemma is immediate since $\rho_{E,N}^*$ is unramified at all primes $p\nmid N$ for which $E$ has good reduction.
\end{proof}

Take any prime $p\nmid 2Nd$ for which $E$ has good reduction.   The character $\chi_d$ is unramified at $p$ and $\chi_d(\Frob_p)=1$ if and only if $d$ is a square modulo $p$.     By Lemma~\ref{L:alphaE properties0}, we deduce that $\alpha_u$ is unramified at $p$ and that $\alpha_E(\Frob_p) = \chi_d(\Frob_p) \cdot \alpha_u(\Frob_p)$.

In \S\ref{SSS:low genus phi}, we will describe how to compute $\alpha_u(\Frob_p)\in \calG/G$ for all sufficiently large primes $p$ under the additional assumptions that $\calG/G$ is cyclic of prime power order, $X_\calG(\QQ)$ is infinite, and $u$ lies outside some explicit finite subset of $X_\calG(\QQ)$.

\subsection{The function field $L$ in a special case}  \label{SS:L special}

We shall now assume further that $\calG/G$ is a cyclic group of prime power order $p_0^e>1$.    In this section, we will describe a set of generators of the extension $L$ of $\QQ(X_\calG)$ with a simple and explicit action of $\calG/G$ on them.  \\

Fix a matrix $g_0 \in \bbar{\calG} \cap \SL_2(\ZZ/N\ZZ)$ so that $g_0 \bbar{G}$ generates the cyclic group $\bbar{\calG}/\bbar{G}=\calG/G$.   For an integer $k\geq 2$, $M_{k,\bbar{G}}$ has a right action by the group $\calG/ G$.       We can compute a basis of $M_{k,\bbar{G}}$ using the methods of \S\ref{SS:MkG} and we can compute the action of $g_0$ with respect to this basis by using \S\ref{S:explicit slash}.    We can choose $k\geq 2$ so that the action of $\calG/ G$ on $M_{k,\bbar{G}}$ is faithful.  When $-I \in G$, we further assume that $k$ is chosen to be even and large enough so that there is a nonzero $h \in M_{k,\bbar{\calG}}$.   

Suppose $-I \notin  G$. By Lemma~\ref{L:weight 3 existence}, there is a nonzero $f_1\in M_{3,\bbar{G}}$.  We claim that $\calG/ G$ acts faithfully on $M_{3,\bbar{G}}$.   Since $\calG/ G$ is cyclic of order $p_0^e>1$, it suffices to show that the minimal nontrivial subgroup of $\calG/ G$ acts faithfully.   Since $-I \notin  G$, this group is $(\pm  G)/ G$ and it acts faithfully on $M_{3,\bbar{G}}$ because this space is nonzero and $-I$ acts as multiplication by $-1$.  So when $-I \notin  G$, we may always take $k=3$.

Since $\calG/ G$ is cyclic of prime power order, there is a $\QQ$-subspace $V \subseteq M_{k,\bbar{G}}$ for which the right action of $\calG/ G$ is faithful and  irreducible.   Moreover, we can find an explicit basis $f_1,\ldots, f_m$ of $V$ such that 
\[
f_j * \,g_0 = \sum_{i=1}^m f_i  \cdot C_{i,j}
\] 
for all $1\leq j \leq m$, where $m=\phi(p_0^e)=p_0^{e-1}(p_0-1)$ and $C\in \GL_m(\QQ)$ is the companion matrix of the cyclotomic polynomial $\Phi_{p_0^e}(x)$.  Note that the matrix $C$ has order $p_0^e$ in $\GL_m(\QQ)$.

Let $\calR$ be the $\QQ$-subalgebra of $\QQ[x_1,\ldots,x_m]$ consisting of polynomials $F$ for which $F(x_1,\ldots, x_m) = F((x_1,\ldots, x_m) C)$.  Take any homogeneous polynomial $F\in \calR$ and denote its degree by $d$.  When $-I\notin  G$, a power of $C$ is $-I$ and hence $d$ is even.    The modular form $F(f_1,\ldots, f_m)$ has weight $dk$ and is fixed by $\bbar{G}$ and $g_0$.  Therefore,  $F(f_1,\ldots, f_m)$ is an element of $M_{dk,\overline{\calG}}$.  Define 
 \[
 c_F := \begin{cases}
      \dfrac{F(f_1,\ldots, f_m)}{h^{d}} & \text{if $-I \in \bbar G$,} \\
       \dfrac{F(f_1,\ldots, f_m)  \,j^{d/2}}{E_6^{d/2}}& \text{if $-I\notin \bbar G$}.
\end{cases}
 \]
Note that $c_F$ is in $\calF_N^{\bbar{\calG}} = \QQ(X_\calG)$.   The following lemma describes the extension $L/\QQ(X_\calG)$ in terms of the $c_F$.
  
\begin{lemma}  \label{L:fibers and L}
We have $L=\QQ(X_\calG)(y_1,\ldots, y_m)$, where the $y_j$ can be chosen such that: 
\begin{itemize}
\item $y_j * g_0  = \sum_{i=1}^m y_i\cdot C_{i,j}$ for all $1\leq m$, 
\item $F(y_1,\ldots, y_m)= c_F$ for all homogeneous polynomials $F\in \calR$.
\end{itemize}
\end{lemma} 
\begin{proof}
First suppose that $-I \in  G$.  We define $y_j:=f_j/h$ for $1\leq j \leq m$.   For a homogeneous $F\in \calR$ of degree $d$, we have $F(y_1,\ldots, y_m)=F(f_1,\ldots,f_m)/h^d=c_F$.   Take any $\sigma\in \Gal_{\QQ(X_\calG)}$ and set $A:=\rho_{\calE,N}^*(\sigma)^{-1} \in \bbar{\calG}$.  We have $\sigma(y_j)= y_j *A=(f_j *A)/(h*A)=(f_j * A)/h$, where the last equality uses our choice of $h$.  We have $y_j\in L$ since $\sigma(y_j)=y_j$ when $A\in \bbar{G}$.  Now suppose $\sigma$ is chosen such that $A= g_0$.  Then
 $y_j * g_0 = (f_j * g_0)/h =  (\sum_{i=1}^m f_i  \cdot C_{i,j})/h = \sum_{i=1}^m y_i  \cdot C_{i,j}$.
 
Now suppose that $-I \notin  G$ and hence $k=3$.  Define $y_i:=\beta \cdot f_i/f_0$ for $1\leq i \leq m$.  Take any homogeneous $F\in \calR$ of (even) degree $d$.  Since $\beta^2 = j\cdot {f_0^2}/{E_6}$, we have 
 \[
 F(y_1,\ldots, y_m)=\left(\frac{\beta^2}{ f_0^2}\right)^{d/2}  F( f_1,\ldots,  f_m) = \left(\frac{j}{E_6}\right)^{d/2}  F( f_1,\ldots,  f_m) = c_F. 
 \]
Take any $\sigma\in \Gal_{\QQ(X_\calG)}$ and set $A:=\rho_{\calE,N}^*(\sigma)^{-1}$.   By Lemma~\ref{L:FN revisited}(\ref{L:FN revisited vi}), we have $\sigma(\beta) = \tfrac{f_0 *A}{f_0}\, \beta$. Since $f_j/f_0 \in \calF_N$, we have 
 \[
 \sigma(y_j)= \sigma(\beta)\cdot \sigma(\tfrac{f_j}{f_0}) = \tfrac{f_0 *A}{f_0}\, \beta \cdot (\tfrac{f_j}{f_0}) *A =\tfrac{f_0 *A}{f_0}\, \beta \cdot \tfrac{f_j * A}{f_0*A} =\beta \cdot \tfrac{f_j * A}{f_0}.
\] 
If $A\in \bbar{G}$, then $\sigma(y_j) = \beta \cdot f_j/f_0 = y_j$.  Therefore, $y_1,\ldots, y_m$ all lie in $L$. Now suppose $\sigma$ is chosen such that $A= g_0$.  Therefore,
\[
y_j* g_0=
\sigma(y_j) = \beta  \cdot \tfrac{f_j *A}{f_0} = \beta  \tfrac{f_j * g_0}{f_0} = \beta (\sum_{i=1}^m f_i \cdot C_{i,j})/f_0 = \sum_{i=1}^m y_i \cdot C_{i,j}.
\]
 
In both cases, we have proved that $\QQ(X_\calG)(y_1,\ldots, y_m) \subseteq L$ and that $y_1,\ldots, y_m$ have the desired properties.   In both constructions, we have shown that there is a $\sigma \in \Gal_{\QQ(X_\calG)}$ whose action on $\oplus_{j=1}^m \QQ y_j \subseteq L$ is given by the matrix $C$.  Since the order of $C\in \GL_m(\QQ)$ is equal to $|\bbar\calG/\bbar G|=[L:\QQ(X_\calG)]$, we deduce that $L=\QQ(X_\calG)(y_1,\ldots,y_m)$.
\end{proof}

\subsubsection{Low genus setting} \label{SSS:low genus phi}

We now further assume that the curve $X_\calG$ has infinitely many rational points; in particular, $X_\calG$ has genus at most $1$ and a rational point.  As outlined in \S\ref{SS:explicit low genus model}, we can compute an explicit model for $X_\calG$.  In particular, the function field $\QQ(X_\calG)$ will be of the form $\QQ(f)$ or $\QQ(x,y)$ with $x$ and $y$ satisfying a Weierstrass equation of an elliptic curve over $\QQ$.

For any given homogeneous polynomial $F\in \calR$, we can express $c_F \in \QQ(X_\calG)$ in terms of the explicit generators of our function field $\QQ(X_\calG)$ using the methods from \S\ref{SSS:recognizing elements in function field}.\\
 
We now apply Lemma~\ref{L:fibers and L} to describe all but finite many fibers of $\phi$.   Since $\calR$ is a finitely generated $\QQ$-algebra, there are only finite many $u\in U_\calG(\QQ)$ for which $c_F$ has a pole at $u$ for some homogeneous $F\in \calR$.   For a point $u \in U_\calG(\QQ)$ for which $c_F$ never has a pole at $u$, we let $Z_u \subseteq \AA^m_\QQ$ be the subscheme defined by the equations 
\[
F(x_1,\ldots, x_m) = c_F(u)
\]
with $F \in \calR$ homogeneous.   The group $\bbar{\calG}/\bbar{G}=\calG/G$  acts on $Z_u(\Qbar)$ by $g_0\bbar{G} \cdot (a_1,\ldots, a_m) := (a_1,\ldots,a_m) \cdot C$.   For all but finitely many $u$, $Z_u$ is a reduced finite $\QQ$-scheme of degree $|\calG/ G|=p_0^e$ with $\calG/ G$ acting transitively on $Z_u(\Qbar)$; for such $u$, we have an isomorphism $\phi^{-1}(u)\cong Z_u$ with compatible $\calG/ G$-actions.

Choose homogeneous polynomials $F_1,\ldots, F_r \in \calR$ such that there is a point $u \in U_\calG(\QQ)$ for which the equations
\begin{align} \label{E:fiber model finite Fs}
F_1(x_1,\ldots, x_m)=c_{F_1}(u),\quad \ldots, \quad F_r(x_1,\ldots, x_m)=c_{F_r}(u)
\end{align}
define $Z_u$ as a variety (so it is finite of degree $p_0^e$ with a transitive $\calG/ G$-action on its $\Qbar$-points).   Recall that we can compute the functions $c_{F_j}\in \QQ(X_\calG)$.  The equations (\ref{E:fiber model finite Fs}) will also define $Z_u$ for $u\in U_\calG(\QQ)-\calS$, where $\calS$ is a finite set that can be computed.   So for any point $u\in U_\calG(\QQ)-\calS$, (\ref{E:fiber model finite Fs}) gives an explicit model for the fiber $\phi^{-1}(u)$ with the action of $ \calG/ G$ upon it.    So from (\ref{E:characterization of alphau}), we have 
\[
\sigma(z_0)=\alpha_u(\sigma)\cdot  z_0
\] 
for any fixed point $z_0\in Z_u(\Qbar)$ and $\sigma\in \Gal_\QQ$.

Fix a point $u\in U_\calG(\QQ)-\calS$.   For all primes $p\nmid Np_0$ large enough,  we can reduce  the equations (\ref{E:fiber model finite Fs}) modulo $p$ to obtain a variety $Z_{u,p} \subseteq \AA_{\FF_p}^m$ with an action of $ \calG/ G$ so that the action of $ \calG/ G$ on $Z_{u,p}(\FFbar_p)$ is simply transitive.  For such a prime $p$, $\alpha_u$ is unramified at $p$ and we have $(z_1^p,\ldots, z_m^p)=\alpha_u(\Frob_p)\cdot (z_1,\ldots, z_m)$ for any $(z_1,\ldots, z_m)\in Z_{u,p}(\FFbar_p)$.   In particular, for any large enough prime $p$, we can use this to compute $\alpha_u(\Frob_p)\in  \calG/G$.

\begin{remark}
There are a finite number of excluded points $u\in \calS \subseteq U_\calG(\QQ)$.  By changing variables, adding more equations of the form $F_i(x_1,\ldots, x_m) = c_{F_i}(u)$, or making a different choice of $h$ if relevant, we can often find a model for the fiber $\phi^{-1}(u)$ that allows us to compute $\alpha_u(\Frob_p)$ for any sufficiently large prime $p$.   

For our application to Serre's open image theorem, we work with a \emph{single} choice of $u\in U_\calG(\QQ)$ satisfying $\pi_\calG(u)=j$ for some fixed $j\in \QQ-\{0,1728\}$. So for our application we can exclude from consideration any $u\in \calS$ for which there is another point $u'\in U_\calG(\QQ)-\calS$ satisfying $\pi_\calG(u)=\pi_\calG(u')$.
\end{remark}

\subsection{Precomputations}  \label{SS:precomputation}

We now describe some one-time computations that will be required for our algorithm for computing the image of $\rho_E^*$, up to conjugacy, for any non-CM elliptic curve $E/\QQ$.

Consider any of the finite number of groups $\calG\in \scrA$, with $\scrA$ from Theorem~\ref{T:main agreeable}, that satisfy the following properties:
\begin{alphenum}
\item  \label{I:precomp a}
$X_\calG(\QQ)$ is infinite,
\item \label{I:precomp b}
if $\calG$ is conjugate in $\GL_2(\Zhat)$ to a proper subgroup of some $\calG' \in \scrA$, then $[\calG,\calG]$ and $[\calG',\calG']$ are not conjugate in $\GL_2(\Zhat)$.
\end{alphenum}
From Theorem~\ref{T:main agreeable}, we will already have computed a model for $X_\calG$ and the morphism $\pi_\calG$ to the $j$-line.  In particular, we have a model for $U_\calG$.\\

Choose an open subgroup $G_0$ of $\calG$ satisfying $\det(G_0)=\Zhat^\times$ and $G_0\cap \SL_2(\Zhat) =[\calG,\calG]$.  In our cases, we choose $G_0$ with minimal level; the levels of the groups $G_0$ and $\calG$ turn out to have the same odd prime divisors.   Note that $G_0$ is a normal subgroup of $\calG$ and that $\calG/G_0$ is finite and abelian.   So we can choose proper normal subgroups $G_1,\ldots, G_s$ of $\calG$ containing $G_0$ such that the quotient maps $\calG\to \calG/G_i$ induce an isomorphism
\begin{align} \label{E:calG/G decomp}
\calG/G_0 \xrightarrow{\sim} \calG/G_1\times \cdots \times \calG/G_s,
\end{align}
where the groups $\calG/G_i$ are all nontrivial and cyclic of prime power order.  Moreover, we choose our groups $G_i$, with $1\leq i \leq s$, so that at most one does not contain $-I$ and so that the levels of the groups are as small as possible.     Since $G_0\subseteq G_i \subseteq \calG$, the group $G_i$ is open in $\GL_2(\Zhat)$, $\det(G_i)=\Zhat^\times$ and $-I\in G_i$.  

With notation as in \S\ref{SS:general setup} and $G:=G_0$, we have a homomorphism 
\[
\alpha\colon \pi_1(U_\calG,\bbar\eta) \to \calG/G_0
\]
with specializations $\alpha_u\colon \Gal_\QQ \to \calG/G_0$ for $u\in U_\calG(\QQ)$.   Taking instead $G:=G_i$ with $1\leq i \leq s$, we obtain a homomorphism 
\[
\alpha_i \colon \pi_1(U_\calG,\bbar\eta) \to \calG/G_i
\]
with specializations $\alpha_{i,u} \colon \Gal_\QQ \to \calG/G_i$.  Note that $\alpha_i$ can also be obtained by composing $\alpha$ with the isomorphism (\ref{E:calG/G decomp}) and then projecting to the factor $\calG/G_i$.

Now take any $1\leq i \leq s$ and consider the setting of \S\ref{SS:L special} with $G:=G_i$. With notation as in \S\ref{SS:L special} and \S\ref{SSS:low genus phi}, we compute homogeneous polynomials $F_1,\ldots, F_r$ and rational functions $c_{F_1},\ldots, c_{F_r} \in \QQ(X_\calG)$ such that for all $u\in U_\calG(\QQ)$ away from an explicit finite set $\calS_i$, the equations (\ref{E:fiber model finite Fs}) define a $\QQ$-scheme $Z_u$ with a transitive $\calG/G$-action on $Z_u(\Qbar)$ such that 
\[
\sigma(z)=\alpha_{i,u}(\sigma) \cdot z
\]
for all $z\in Z_u(\Qbar)$ and $\sigma\in \Gal_\QQ$.   As explained in  \S\ref{SSS:low genus phi}, we can compute $\alpha_{i,u}(\Frob_p)\in \calG/G=\calG/G_i$ for any sufficiently large primes $p$ by considering the reduction modulo $p$.   

\begin{remark}
For $p$ large enough, the point $u$ modulo $p$ will be enough to determine $\alpha_{i,u}(\Frob_p)$.  So that we can reuse computations when dealing with many elliptic curves over $\QQ$, we have precomputed these values for several primes $p$ and $\FF_p$-points on a model of $U_\calG$.
\end{remark}

Define $\calS:=\calS_1\cup \cdots\cup \calS_s$; it is an explicit finite set.  Consider any $u\in U_\calG(\QQ)-\calS$.   For each $1\leq i \leq s$, we noted that one can compute $\alpha_{i,u}(\Frob_p)$ for all sufficiently large primes $p$.  By making use of the isomorphism (\ref{E:calG/G decomp}), we can verify that $\alpha_u$ is unramified at $p$ and compute $\alpha_u(\Frob_p) \in \calG/G_0$ for any sufficient large primes.

\section{Computing the image of $\rho_E$} \label{S:putting it together}

Take any non-CM elliptic curve $E/\QQ$.  We now combine the previous sections to explain how to compute the image of $\rho_E$ up to conjugacy in $\GL_2(\Zhat)$.   We assume that $E$ is given explicitly as a Weierstrass model.  Let $j_E\in \QQ$ be the $j$-invariant of $E$. 

\subsection{Agreeable closure}

As outlined in \S\ref{S:finding the agreeable closure}, we can compute the agreeable closure $\calG_E$ of $G_E:=\rho_E^*(\Gal_\QQ)$, up to conjugacy in $\GL_2(\Zhat)$, and determine whether $X_{\calG_E}(\QQ)$ is infinite or not.   If $X_{\calG_E}(\QQ)$ is infinite, we may choose $\calG_E$ so that it lies in our finite set $\scrA$ from Theorem~\ref{T:main agreeable}.  

As noted in \S\ref{S:finding the agreeable closure}, from $\calG_E$ we can already compute the index $[\GL_2(\Zhat): G_E]$ and the open subgroup $G_E \cap \SL_2(\Zhat) = [\calG_E,\calG_E]$ of $\SL_2(\Zhat)$ up to conjugacy in $\GL_2(\Zhat)$.

\subsection{Computing the image of Galois in most cases}  \label{SS:gammaE main case}

Fix a group $\calG \in \scrA$ with $X_\calG(\QQ)$ infinite for which $\calG_E$ is conjugate in $\GL_2(\Zhat)$ to a subgroup of $\calG$.   If $X_{\calG_E}(\QQ)$ is infinite, we shall further assume that $\calG$ is chosen so that $[\calG_E,\calG_E]$ and $[\calG,\calG]$ are conjugate in $\SL_2(\Zhat)$ (such a group exists in this case since $\calG_E$ is conjugate to an element of $\scrA$).    After possibly replacing $\calG$ by a different group in $\scrA$, we may further assume that condition (\ref{I:precomp b}) of \S\ref{SS:precomputation} holds; it already satisfies condition (\ref{I:precomp a}).  

 In \S\ref{SS:precomputation}, we chose (independent of $E$) an open subgroup $G:=G_0$ of $\calG$ such that $\det(G)=\Zhat^\times$ and $G\cap \SL_2(\Zhat) =[\calG,\calG]$.   In particular, $G$ is a normal subgroup of $\calG$ with $\calG/G$ finite and abelian.

Let $\calS\subseteq X_\calG(\QQ)$ be finite set from \S\ref{SS:precomputation}. We now make the additional assumption on $E/\QQ$ that there is a rational point $u\in U_\calG(\QQ)-\calS$ for which $j_E = \pi_\calG(u)$.   For the models from our computations, this assumption always holds; if not, it could be treated separately as we do in  \S\ref{SS:gammaE exceptional case}.  
As in \S\ref{SS:precomputation}, we have a homomorphism
\[
\alpha_u\colon \Gal_\QQ \to \calG/G
\]
and for all sufficiently large primes $p$, we can verify that $\alpha_u$ is unramified at $p$ and actually compute $\alpha_u(\Frob_p)\in \calG/G$.\\

Let $d$ be the unique squarefree integer for which $E$ is isomorphic to the quadratic twist of $E'$ by $d$, where $E'/\QQ$ is the elliptic curve defined by the Weierstrass equation $y^2 =  x^3 -  27  j_E (j_E-1728)  \cdot  x +54  j_E(j_E-1728)^2$.  As in \S\ref{SS:finding image setup}, we can define a homomorphism
\[
\alpha_E\colon \Gal_\QQ \to \calG/G,\quad \sigma\mapsto \chi_d(\sigma) \cdot \alpha_u(\sigma),
\]
where $\chi_d\colon \Gal_\QQ \to \{\pm 1\}$ is the character that factors through $\Gal(\QQ(\sqrt{d})/\QQ)\hookrightarrow \{\pm 1\}$.  For $p\nmid 2d$, $\chi_d$ is unramified at $p$ and $\chi_d(\Frob_p)=1$ if and only if $d$ is a square modulo $p$.   Let $M$ be the product of those primes that divide $N$ or for which $E$ has bad reduction.   The homomorphism $\alpha_E$ is unramified at all $p\nmid M$ by Lemma~\ref{L:alphaE properties0}.  We can thus compute $\alpha_E(\Frob_p) \in \calG/G$ for any sufficiently large primes $p\nmid M$.

By Lemma~\ref{L:alphaE properties}, we may assume that, after possibly replacing $\rho_E^*$ by an isomorphic representation, that $\rho_E^*(\Gal_\QQ) \subseteq \calG$ and that $\alpha_E$ is the composition of $\rho_E^*$ with the quotient map $\calG/G$.  Let
\[
\gamma_E \colon \Zhat^\times \to \calG/G
\] 
be the unique homomorphism for which $\gamma_E(\chi_\cyc(\sigma)^{-1})=\alpha_E(\sigma)$ for all $\sigma\in \Gal_\QQ$.  Since $\alpha_E$ is unramified at all primes $p\nmid M$, we find that $\gamma_E$ factors through a homomorphism
\[
\bbar\gamma_E \colon \ZZ_M^\times/(\ZZ_M^\times)^e \to \calG/G,
\]
where $e$ is the exponent of the group $\calG/G$, and $\bbar\gamma_E(p\cdot (\ZZ_M^\times)^e)=\alpha_E(\Frob_p)^{-1} \in \calG/G$ for all primes $p\nmid M$.   So we can find $\gamma_E$ by computing $\alpha_E(\Frob_p)$ for a finite set of primes $p\nmid M$ that generate the finite group $\ZZ_M^\times/(\ZZ_M^\times)^e$.

\begin{remark}
We have used the larger group $\calG$ instead of $\calG_E$ since it leads to fewer cases to consider in \S\ref{SS:precomputation}.   There are {$454$} groups $H \in \scrA$ for which $X_H(\QQ)$ is infinite, but only {$138$} of these groups $H$ will arise as a group $\calG$ like above.  
\end{remark}


Define the explicit subgroup
\[
\calH_E := \{ g \in \calG :  g\cdot G = \gamma_E(\det g)\} 
\]
of $\GL_2(\Zhat)$. In particular, note that $\calH_E$ is computable, cf.~Remark~\ref{R:let us make clear}.   The group $G_E:=\rho_E^*(\Gal_\QQ)$ is a subgroup of $\calH_E$ since  $\rho_E^*(\sigma) \cdot G =\alpha_E(\sigma)=\gamma_E(\chi_\cyc(\sigma)^{-1}) = \gamma_E( \det(\rho_E^*(\sigma)) )$ for all $\sigma\in \Gal_\QQ$. \\

Now consider the case where $[\calG_E,\calG_E]$ and $[\calG,\calG]$ are conjugate subgroups of $\SL_2(\Zhat)$; this condition is automatic when $X_{\calG_E}(\QQ)$ is infinite by our choice of $\calG$.   We have $\calG_E \subseteq \calG$ so $[\calG_E,\calG_E] \subseteq [\calG,\calG]$ and hence $[\calG_E,\calG_E] = [\calG,\calG]$ since they are conjugate open subgroups of $\SL_2(\Zhat)$.    In particular, $G_E \cap \SL_2(\Zhat) = [\calG,\calG]$ by Lemma~\ref{L:KW} and Proposition~\ref{P:agreeable closure}.  Therefore,
\[
\calH_E \cap \SL_2(\Zhat) = G \cap \SL_2(\Zhat) = [\calG,\calG] = G_E \cap \SL_2(\Zhat).
\]    
Since $G_E$ is a subgroup of $\calH_E$ with $G_E \cap \SL_2(\Zhat)= \calH_E \cap \SL_2(\Zhat)$ and $\det(G_E)=\Zhat^\times$, we deduce that $G_E=\calH_E$.

\subsection{Computing the image of Galois in the remaining cases}\label{SS:gammaE exceptional case}

We have already computed $\calG_E$, up to conjugacy in $\GL_2(\Zhat)$, and we know if $X_{\calG_E}(\QQ)$ is infinite or not.   If $X_{\calG_E}(\QQ)$ is infinite, then \S\ref{SS:gammaE main case} shows how to compute the group $\rho_E^*(\Gal_\QQ)$ up to conjugacy in $\GL_2(\Zhat)$.  

We now restrict our attention to the case when $X_{\calG_E}(\QQ)$ is finite.   If $E/\QQ$ is not a counterexample to Conjecture~\ref{C:Serre question}, then the $j$-invariant $j_E$ lies in the finite set
\[
\calJ:= \bigcup_{\calG \in \scrA,\, X_\calG(\QQ) \text{ finite}} \pi_\calG( U_{\calG}(\QQ)) \subseteq \QQ
\]
with $\scrA$ as in Theorem~\ref{T:main agreeable}.  We are aware of 81 rational numbers $j\in \calJ$ for which $j$ is the $j$-invariant of a non-CM elliptic curve.   We will explain how to compute $G_E:=\rho_E^*(\Gal_\QQ)$, up to conjugacy, in these cases.   Any other non-CM $j$-invariants in $\calJ$, or counterexamples to Conjecture~\ref{C:Serre question}, can be dealt with in a similar direct manner.

For the finite number of $j$-invariants under consideration, we need only consider a single elliptic curve with that $j$-invariant (from Lemma~\ref{SS:quadratic twist}, replacing $E$ by a quadratic twist changes $G_E$ in an explicit way).

\subsubsection{Case 1:  the previous approach works}

For our 81 exceptional $j$-invariants, the group $G_E$ can be computed for 28 of them using the method of \S\ref{SS:gammaE main case}.   In particular, we can find an agreeable groups $\calG\in \scrA$ so that $\calG_E$ is conjugate in $\GL_2(\Zhat)$ to a subgroup of $\calG$, $X_\calG(\QQ)$ is infinite, and $[\calG_E,\calG_E]$ is conjugate to a subgroup of $[\calG,\calG]$ in $\SL_2(\Zhat)$.

\subsubsection{Case 2: intersections with relatively prime levels}

Suppose that there are distinct primes $2=\ell_1<\ell_2< \cdots < \ell_s$ and agreeable subgroups $\calG_1,\ldots, \calG_s \in \scrA$ so that all the following hold:
\begin{itemize}
\item 
$\calG_E$ is conjugate in $\GL_2(\Zhat)$ to a subgroup of $\calG_i$ for all $1\leq i \leq s$,
\item
the level of $\calG_i$ divides a power of $\ell_i$ for all $1\leq i \leq s$,
\item
$X_{\calG_i}(\QQ)$ has infinitely many points for all $1\leq i \leq s$,
\item 
$[\calG_E,\calG_E]$ and $\bigcap_{i=1}^s [\calG_i,\calG_i]$ are open subgroups of $\SL_2(\Zhat)$ that are conjugate in $\GL_2(\Zhat)$.
\end{itemize}
Since the levels of the groups $\calG_i$ are pairwise relatively prime, we find that the group $\bigcap_{i=1}^s [\calG_i,\calG_i]$, up to conjugacy in $\GL_2(\Zhat)$, does not change if we replace any $\calG_i$ by a conjugate in $\GL_2(\Zhat)$.  Thus we may assume further that the $\calG_i$ are chosen so that they satisfy condition (\ref{I:precomp b}) of \S\ref{SS:precomputation}.   

For each $1\leq i \leq s$, we choose an open subgroup $G_i$ of $\calG_i$ with $\det(G_i)=\Zhat^\times$ such that the level of $G_i$ is a power of $\ell_i$ and $G_i \cap \SL_2(\Zhat)= [\calG_i,\calG_i]$. We have $j_E \in \pi_{\calG_i}(X_{\calG_i}(\QQ))$ since $\calG_E$ is conjugate to a subgroup of $\calG_i$.   As in \S\ref{SS:precomputation} and \S\ref{SS:gammaE main case}, we can compute an explicit homomorphism $\gamma_{E,i}\colon \Zhat^\times \to \calG_i/G_i$ so that $\rho_E^*(\Gal_\QQ)$ is conjugate in $\GL_2(\Zhat)$ to a subgroup of 
\[
\calH_i := \{ g\in \calG_i : g \cdot G_i = \gamma_{E,i}(\det(g)) \}.
\]
This previous step uses that $X_{\calG_i}(\QQ)$ is infinite.\\

Since the group $G_1,\ldots, G_s$ have pairwise relatively prime levels, we find that after replacing $\rho_E^*$ by an isomorphic representation we have $G_E:=\rho_E^*(\Gal_\QQ) \subseteq \calH_i$ for all $1\leq i \leq s$.    In particular, $G_E \subseteq \calH:= \bigcap_{i=1}^s \calH_i$.  

We claim that $G_E=\calH$; since $\calH$ has an explicit description this will conclude our description of how to compute $G_E$ up to conjugacy.
  We have
\begin{align} \label{E:case 2 exceptional image}
G_E \cap \SL_2(\Zhat) \subseteq \calH \cap \SL_2(\Zhat) \subseteq \bigcap_{i=1}^s (\calH_i\cap \SL_2(\Zhat))= \bigcap_{i=1}^s (G_i \cap \SL_2(\Zhat)) = \bigcap_{i=1}^s [\calG_i,\calG_i].
\end{align}
We have $G_E \cap \SL_2(\Zhat) = [\calG_E,\calG_E]$ by Lemma~\ref{L:KW} and Proposition~\ref{P:agreeable closure}.  So by assumption, $G_E \cap \SL_2(\Zhat)$ and $\bigcap_{i=1}^s [\calG_i,\calG_i]$ are open subgroups of $\SL_2(\Zhat)$ that are conjugate in $\GL_2(\Zhat)$.  From the inclusions (\ref{E:case 2 exceptional image}), we deduce that $G_E \cap \SL_2(\Zhat) = \calH \cap \SL_2(\Zhat)$.  Since $G_E$ is a subgroup of $\calH$ with full determinant, we conclude that $G_E=\calH$.

Of the 53 exceptional $j$-invariants not handled by Case 1, we use the method above to compute $G_E$, up to conjugacy, for an additional 24 $j$-invariants.   Of the $29$ remaining exceptional $j$-invariants, $20$ of them arise in \cite{RSZ}.

\subsubsection{Case 3: check directly}

We already know the agreeable closure $\calG_E$ of $G_E=\rho_E^*(\Gal_\QQ)$.   We can choose an open subgroup $G$ of $\calG_E$ with minimal level that satisfies $\det(G)=\Zhat^\times$ and $G\cap \SL_2(\Zhat)=[\calG_E,\calG_E]$.     Let $M$ be the product of the primes $p$ so that $p$ divides the level of $G$ or $E$ has bad reduction at $p$; this is an integer we can compute.  The homomorphism $\alpha_E\colon \Gal_\QQ\to \calG_E/G$ obtained by composing $\rho_E^*$ with the obvious quotient map will be unramified at all primes $p\nmid M$.    Using Lemma~\ref{L:intro HE}, we deduce that the every prime that divides the level of $G_E$ must also divide $M$.
\\

After replacing $\calG_E$ by a conjugate, we can find a group $\calG \in \scrA$ so that $X_\calG(\QQ)$ is infinite and $\calG_E \subseteq \calG$.  We choose $\calG$ so that $[\calG:\calG_E]$ is minimal.   Using \S\ref{SS:gammaE main case},  we can construct a computable open subgroup $\calH$ of $\GL_2(\Zhat)$ for which $\calH \cap \SL_2(\Zhat)=[\calG,\calG]$, $\det \calH =\Zhat^\times$ and $G_E$ is conjugate in $\GL_2(\Zhat)$ to a subgroup of $\calH$.     

So after possibly replacing $\rho_E^*$ by an isomorphic representation, we find that $G_E=\rho_E^*(\Gal_\QQ)$ is an open subgroup of $\calH$ with 
\[
[\calH:G_E]=[\calH \cap\SL_2(\Zhat):G_E \cap\SL_2(\Zhat)]  = [ [\calG,\calG], [\calG_E,\calG_E] ]=:m
\]
So $G_E$ is an index $m$ open subgroup of $\calH$ whose level in $\GL_2(\Zhat)$ divides some power of $M$.  However, there are only finitely many such groups, so we can compute them all up to conjugacy in $\GL_2(\Zhat)$.   It remains to check which of these explicit candidates is actually conjugate to $G_E$.   Looking at traces of Frobenius can be useful to rule out some possibilities and hope that one case remains.   In any remaining cases, one can directly compute division polynomials for the curve $E$ and study their Galois groups to determine $G_E$.   For example, \S\ref{SS:largest known index} gives one of the exceptional elliptic curves we dealt with directly using division polynomials.

\subsection{Finding the image} \label{SS:end games}

From \S\ref{SS:gammaE main case} or \S\ref{SS:gammaE exceptional case}, we have found the following:
\begin{itemize}
\item
an agreeable group $\calG$ and an open and normal subgroup $G$ of $\calG$ satisfying $\det(G)=\Zhat$ and $G\cap \SL_2(\Zhat) = [\calG,\calG]$,
\item
a homomorphism $\gamma_E\colon \Zhat^\times \to \calG/G$ such that, after replacing $\rho_E^*$ by an isomorphic representation, we have $\rho_E^*(\Gal_\QQ) \subseteq \calG$ and the homomorphism
\[
\alpha_E\colon \Gal_\QQ \to \calG/G
\]
obtained by composing $\rho_E^*$ with the quotient map $\calG\to \calG/G$ satisfies $\gamma_E(\chi_\cyc(\sigma)^{-1})=\alpha_E(\sigma)$ for all $\sigma\in \Gal_\QQ$.
\item
the commutator subgroups of the two groups $\rho_E^*(\Gal_\QQ)$ and $\calG$ are conjugate in $\GL_2(\Zhat)$.
\end{itemize}

We have $G_E:=\rho^*_E(\Gal_\QQ) \subseteq \calG$ and hence $[G_E,G_E]\subseteq [\calG,\calG]$.  We have $[G_E,G_E]= [\calG,\calG]$ since they are open subgroups of $\SL_2(\Zhat)$ that are conjugate in $\GL_2(\Zhat)$.  In particular, $G_E \cap \SL_2(\Zhat) = [\calG,\calG]$ by Lemma~\ref{L:KW}.    By Lemma~\ref{L:intro HE}, $\rho_E^*(\Gal_\QQ)$ is conjugate in $\GL_2(\Zhat)$ to the explicit group
\[
\calH_E := \{ g \in \calG :  g\cdot G = \gamma_E(\det g)\}
\]
which is computable, cf.~Remark \ref{R:let us make clear}.

Let $\calH_E^t\subseteq \GL_2(\Zhat)$ be the group obtained by taking the transpose of all elements in $\calH_E$.  The groups $\rho_E(\Gal_\QQ)$ and $\calH_E^t$ are then conjugate in $\GL_2(\Zhat)$.

\section{Universal elliptic curves} \label{S:universal elliptic curves}

Consider an open subgroup $G$ of $\GL_2(\Zhat)$ that satisfies $\det(G)=\Zhat^\times$ and $-I \notin G$.    Define the group $\calG:=\pm G$.    From our definition in \S\ref{S:first modular curve}, we have $X_\calG=X_G$.  Recall that $U_G=U_\calG$ is the open subvariety $X_G - \pi_G^{-1}(\{0,1728,\infty\})$ of $X_G$.   

We will say that an elliptic scheme $E\to U_G$ is a \defi{universal elliptic curve over $U_G$} if the following hold for any number field $K$:
\begin{itemize}
\item
for any point $u\in U_G(K)$, the $j$-invariant of $E_u/K$ is $\pi_G(u)$, where the elliptic curve $E_u$ is the fiber of $E\to U_G$ over $u$,
\item
for all $u\in U_G(K)$, $\rho_{E_u}^*(\Gal_K)$ is conjugate in $\GL_2(\Zhat)$ to a subgroup of $G$.
\end{itemize}
In this section, we sketch some methods for computing such a universal elliptic curve.  This will follows directly from other parts of the paper, but we state it here for convenient reference.  We will not use this elsewhere. 

\begin{remark}
In this paper, we have not taken a moduli point of view for modular curves.  However, such a viewpoint makes the existence of  a universal elliptic curve obvious;  the underlying moduli space is fine since $-I \notin G$ (note that  we are excluding elliptic curves with extra automorphisms by focusing on $U_G$).  Moreover, using the moduli approach, one can also show that if $E'$ is an elliptic curve over a number field $K$ with $j_{E'} \in K-\{0,1728\}$, then $\rho_{E'}^*(\Gal_K)$ is conjugate in $\GL_2(\Zhat)$ to a subgroup of $G$ if and only if $E'/K$ is isomorphic to $E_u$ for some $u\in U_G(K)$.
\end{remark}

Let $N$ be the level of $G$.   We have $N\geq 3$ since $-I\notin G$.  Let $\bbar{G}$ and $\bbar{\calG}=\pm \bbar{G}$ be the images of $G$ and $\calG$, respectively, in $\GL_2(\ZZ/N\ZZ)$.   By Lemma~\ref{L:weight 3 existence}, there is a nonzero modular form $f_0\in M_{3,\bbar{G}}$; we can construct such an $f_0$ by using Corollary~\ref{C:Eisenstein span 2}.  Define 
\[
  \delta :=j\cdot f_0^2/E_6;
\]   
it is a nonzero element of $\calF_N^{\bbar{G}}=\QQ(X_G)$.   All poles of $\delta$ lie above the points $0$, $1$ and $\infty$ on the $j$-line; recall that $E_6^2=(j-1728)\Delta$.    In particular, we can view $\delta$ and $j$ as morphisms $U_G\to \AA_\QQ^1$.  

Consider the Weierstrass equation:
\begin{align} 
\label{E:universal Weierstrass}
\delta\cdot y^2 =  x^3 -  27 \cdot j (j-1728)  \cdot  x +54 \cdot j (j-1728)^2.
\end{align}

Let $U'$ be the maximal open subvariety of $U_G$ such that the valuation of $\delta$ at $P$ is even for all closed points $P$ of $U_G$.    The equation  (\ref{E:universal Weierstrass}) defines an elliptic scheme $E\to U'$.  This is clear if we instead restrict to the smaller open subvariety of $U_G$ for which $\delta$ is nonzero.  For excluded points of $U'$, we can scale $y$ appropriately and change coordinates, using our assumption on valuations, to extend the model.

\begin{prop}
We have $U'=U_G$ and $E\to U'=U_G$ defined by (\ref{E:universal Weierstrass}) is a universal elliptic curve over $U_G$.
\end{prop}
\begin{proof}
With notation as in \S\ref{SS:specializations}, we have an elliptic scheme $\scrE_\calG\to U_\calG$ and a surjective homomorphism 
\[
\varrho_{\scrE_\calG,N}^*\colon \pi_1(U_\calG,\bbar\eta) \to  \bbar\calG.
\]
Recall that $\varrho_{\scrE_\calG,N}^*$ depends on a choice of a nonzero modular form $f_0$ in $M_3(\Gamma(N),\QQ(\zeta_N))$ and a choice of $\beta$ in a field extension of $\calF_N$ that satisfies $\beta^2=j\cdot f_0^2/E_6$.   By Lemma~\ref{L:weight 3 existence}, we can assume that $f_0 \in M_{3,\bbar{G}}$ since $-I \notin \bbar G$. Let $\alpha\colon \pi_1(U_\calG,\bbar\eta) \to \bbar \calG/\bbar G$ be the homomorphism obtained by composing $\varrho_{\scrE_\calG,N}^*$ with the quotient map $\bbar\calG\to \bbar\calG/\bbar G$.  Let $\chi\colon \pi_1(U_\calG,\bbar\eta) \to \{\pm 1\}$ be the character obtained by composing $\alpha$ with the isomorphism $\bbar\calG/\bbar G \cong \{\pm 1\}$.

Take any number field $K\subseteq \Qbar$ and point $u\in U_\calG(K)$.   The fiber $(\scrE_\calG)_u$ above $u$ is an elliptic curve over $K$ that is isomorphic to the curve given defined by (\ref{E:generic Weierstrass}) with $j=\pi_\calG(u) \in K-\{0,1728\}$.  In particular, $(\scrE_\calG)_u$ has $j$-invariant $\pi_\calG(u)$.   The specialization of $\varrho_{\scrE_\calG,N}^*$ at $u$ is a representation $\Gal_K \to \bbar{\calG} \subseteq \GL_2(\ZZ/N\ZZ)$ that is isomorphic to $\rho_{(\scrE_\calG)_u,N}^*$.   In particular, if $E'/K$ is the quadratic twist of $(\scrE_\calG)_u$ by the specialization $\chi_u\colon \Gal_K\to \{\pm 1\}$ of $\chi$ at $u$, then $E'$ has $j$-invariant $\pi_\calG(u)$ and $\rho_{E',N}^*(\Gal_K)$ is conjugate in $\GL_2(\ZZ/N\ZZ)$ to a subgroup of $\bbar G$.

Since $N$ is the level of $G$, we deduce that the elliptic scheme $E \to U_\calG=U_G$ obtained by taking the quadratic twist of the elliptic scheme $\scrE_\calG\to U_\calG$ by $\chi$.

The homomorphisms $\alpha$ and $\chi$ correspond to an \'etale cover $\phi\colon Y\to U_\calG$ of degree $2$.   With notation as in \S\ref{SSS:setup reps},  $\phi$ corresponds to a quadratic extension $L/\QQ(X_G)$ with $L\subseteq \calF_N(\beta)$.  Since $\beta^2=\delta \in \QQ(X_\calG)$, it suffices to prove that $L=\QQ(X_\calG)(\beta)$.  That $U'=U_\calG=U_G$ is a consequence of $L/\QQ(X_\calG)$ arising from an \'etale cover of $U_\calG$.

The group $\bbar\calG/\bbar G$ is cyclic of order $2$ and $-I \notin \bbar G$.   With notation and definitions as in \S\ref{SS:L special}, we can assume that $m=1$, $C=(-1)$, $g_0=-I$, $k=3$ and $f_1:=f_0$.   Moreover, with $F(x_1):=x_1^2$ we have $c_F= j\cdot f_0^2/ E_6 = \delta$.  By Lemma~\ref{L:fibers and L}, we find that  $L=\QQ(X_\calG)(y_1)$ where $y_1^2 = \delta$. Since $\beta^2=\delta$, we conclude that $L=\QQ(X_\calG)(\beta)$.
\end{proof}

\section{Families of modular curves} \label{S:families}

We now discuss a point of view that may be of use for further study of modular curves and Mazur's Program~B; these remarks will not be used elsewhere in the paper.

\subsection{Families and twists}

Let $\calG$ be an open subgroup of $\GL_2(\Zhat)$ satisfying $\det(\calG)=\Zhat^\times$ and $-I \in \calG$.    Fix a closed subgroup $B$ of $\calG$ satisfying $[\calG,\calG] \subseteq B$.

\begin{definition}
The \defi{family} of groups associated to the pair $(\calG,B)$ is the set $\scrF(\calG,B)$ of subgroups $H$ of $\calG$ that satisfy $\det(H)=\Zhat^\times$ and $H\cap \SL_2(\Zhat)=B \cap \SL_2(\Zhat)$.
\end{definition}

Suppose that $\scrF(\calG,B)\neq \emptyset$.  Fix a group $G \in \scrF(\calG,B)$.  Note that $\scrF(\calG,B)=\scrF(\calG,G)$.   Since $G  \supseteq [\calG,\calG]$ and $G$ is open, we find that $G$ is a normal subgroup of $\calG$ and that the group $\calG/G$ is finite and abelian. For each homomorphism $\gamma\colon \Zhat^\times \to \calG/G$, define
the subgroup
\[
G_\gamma:=\{g \in \calG :  g\cdot G = \gamma(\det g) \}
\]
of $\GL_2(\Zhat)$.  

\begin{lemma}
With notation as above, the set $\scrF(\calG,B)$ consists of the groups $G_\gamma$ with $\gamma\colon \Zhat^\times \to \calG/G$ a homomorphism.
\end{lemma}
\begin{proof}
First take any $\gamma$.  We have $G_\gamma \cap \SL_2(\Zhat) = G\cap \SL_2(\Zhat)=B$.    The natural map $(\calG \cap \SL_2(\Zhat))/B \to \calG/G$ is an isomorphism since $G\cap \SL_2(\Zhat) = B$ and $\det(G)=\Zhat^\times$.  Using this isomorphism, we find that $\det(G_\gamma)=\Zhat^\times$.   Therefore, $G_\gamma \in \scrF(\calG,B)$.

Conversely, take any $H \in \scrF(\calG,B)$.  The quotient map $H\to \calG/G$ induces a homomorphism $f\colon H/(H\cap \SL_2(\Zhat)) \to \calG/G$ since $H\cap \SL_2(\Zhat)=B=G\cap \SL_2(\Zhat)$.   Let $\gamma\colon \ZZ^\times \to \calG/G$ be the homomorphism obtained by composing the inverse of the determinant map $H/(H\cap \SL_2(\Zhat)) \xrightarrow{\sim} \Zhat^\times$ with $f$.  For each $h\in H$, we have $h \cdot G = \gamma(\det h)$.  Therefore, $H \subseteq G_\gamma$.  Since $H$ and $G_\gamma$ both have full determinant and have the same intersection with $\SL_2(\Zhat)$, we conclude that $H=G_\gamma$.
\end{proof}

\begin{remark} \label{R:generalized Serre curves}
We saw a family of groups in \S\ref{SS:Serre curves} when discussing Serre curves.   In fact, one can show that $E/\QQ$ is a Serre curve if and only if $\rho_E(\Gal_\QQ)$ is an element of $\scrF(\GL_2(\Zhat),[\GL_2(\Zhat),\GL_2(\Zhat)])$.   With terminology from \cite{MR3350106}, $E/\QQ$ is a ``$\calG$-Serre curve'' if and only if $\rho_E(\Gal_\QQ)$ is conjugate in $\GL_2(\Zhat)$ to some group in $\scrF(\calG,[\calG,\calG])$.
\end{remark}

Let us loosely reinterpret some our results from \S\ref{SS:agreeable intro} and \S\ref{SS:finding Galois} in terms of families.  

We have proved that there are finitely many pairs $\calG_1,\ldots, \calG_m$ such  that for any non-CM elliptic curve $E/\QQ$  for which  Conjecture~\ref{C:Serre question} holds, $\rho_E^*(\Gal_\QQ)$ is conjugate in $\GL_2(\Zhat)$ to a group in the family $\scrF(\calG_i,[\calG_i,\calG_i])$ for some $1\leq i \leq m$.      The agreeable closure of $\rho_E^*(\Gal_\QQ)$, computed as in \S\ref{SS:agreeable intro}, determines which of our explicit families $\scrF(\calG_i,[\calG_i,\calG_i])$ our group lies in.    Once we know the specific family of groups, the results from \S\ref{SS:finding Galois} allow us to identify $\rho_E^*(\Gal_\QQ)$ in the family by constructing the appropriate  homomorphism $
\gamma_E \colon \Zhat^\times \to \calG_i/G_i$, where $G_i$ is a fixed group in $\scrF(\calG_i,[\calG_i,\calG_i])$.\\

So that we can talk about modular curves and groups interchangeably, let us now consider the case a nonempty family   $\scrF(\calG,B)$ for which $-I \in B$.  As before, fix $G\in \scrF(\calG,B)$; we have $-I\in G$.
 
   Let $\pi\colon X_G\to X_\calG$ be the morphism of modular curves induced by the inclusion $G \subseteq \calG$.  Since $G$ is a normal subgroup of $\calG$, the group $\calG$ acts on the modular curve $X_G$ with $G$ acting trivially.  This induces an isomorphism $\calG/G \xrightarrow{\sim} \Aut(X_G/X_\calG)$, where $\Aut(X_G/X_\calG)$ is the group of automorphisms $f$ of the curve $X_G$ that satisfy $\pi \circ f = \pi$.

For a fixed homomorphism $\gamma\colon \Zhat^\times \to \calG/G$, we obtain a homomorphism 
\[
\xi := \gamma\circ \chi_\cyc^{-1} \colon \Gal_\QQ \to \calG/G \cong \Aut(X_G/X_\calG).
\]
In particular, we can view $\xi$ as a $1$-cocycle of $X_G$.  Twisting $X_G$ by $\xi$ gives a curve $(X_G)_\xi$ and a morphism $\pi_\xi\colon (X_G)_\xi \to X_\calG$ that are both defined over $\QQ$.   A straightforward computation shows that we can in fact take $(X_G)_\xi=X_{G_\gamma}$ with $\pi_\xi\colon X_{G_\gamma}\to X_\calG$ being the morphism induced by the inclusion $G_\gamma\subseteq \calG$.  So our family of groups $\scrF(\calG,B)=\scrF(\calG,G)$ corresponds to a family of twists $\{(X_G)_\xi\}_\xi$ as we vary over $1$-cocycles $\xi\colon \Gal_\QQ \to \Aut(X_G/X_\calG)$.

Note that the modular curve $X_{G_\gamma}$ need not have a rational non-CM point for every $\gamma$ (moreover, there are families where $X_{G_\gamma}(\QQ)=\emptyset$ for all $\gamma$).

Consider two pairs $(C_1,\pi_1)$ and $(C_2,\pi_2)$, where $C_i$ is a curve over $\QQ$ and $\pi_i\colon C_i \to \PP^1_\QQ$ is a morphism.  We say that the pairs $(C_1,\pi_1)$ and $(C_2,\pi_2)$ are \defi{isomorphic} if there is an isomorphism $f\colon C_1\to C_2$ defined over $\QQ$ so that $\pi_2 \circ f = \pi_1$.  Rakvi \cite{Rakvi} has recently classified the pairs $(X_G,\pi_G)$, up to isomorphism, for which $G$ is an open subgroup of $\GL_2(\Zhat)$ satisfying $\det(G)=\Zhat^\times$, $-I\in G$, and $X_G\cong  \PP^1_\QQ$.   She accomplishes this by showing that all such group $G$ lie in a finite number of families and then identifying which curves arising from these families are isomorphic to $\PP^1_\QQ$.

\subsection{A conjecture on non-CM points of high genus modular curves} \label{SS:conjecture 2023}

The modular curve $X_0(37)$ has exactly two rational points that are not cusps, cf.~\cite{MR482230}.   These rational points of $X_0(37)$ map to the values $-7\!\cdot\! 11^3$ and $-7\!\cdot\! 137^3\!\cdot\! 2083^3$ in the $j$-line.  Let $G_1$ and $G_2$ be the group $\pm \rho_E^*(\Gal_\QQ)$, where $E/\QQ$ is an elliptic curve with $j$-invariant $-7\!\cdot\! 11^3$ and $-7\!\cdot\! 137^3\!\cdot\! 2083^3$, respectively.  Note that the subgroups $G_1$ and $G_2$ of $\GL_2(\Zhat)$ are uniquely defined up to conjugacy.     From \S\ref{SS:largest known index}, we find that $G_1$ is conjugate to the open subgroup of $\GL_2(\Zhat)$ whose level divides $5180$ and whose image modulo $5180$ is generated by $-I$ and the matrices (\ref{E:exceptional gen 2023}).  One can check that $G_2$ is conjugate in $\GL_2(\Zhat)$ to the group $G_1^t$ obtained by taking the transpose of the elements of $G_1$.   

One can show that the modular curves $X_{G_1}$ and $X_{G_2}$ both have genus $97$ and clearly $X_{G_1}(\QQ)$ and $X_{G_2}(\QQ)$ both have a non-CM point. The following conjecture predicts  that $X_{G_1}$ and $X_{G_2}$ are the highest genus modular curves with a rational non-CM point; we will motivate it in the next section.

\begin{conj} \label{C:2023}
Let $G$ be an open subgroup  of $\GL_2(\Zhat)$ with $\det(G)=\Zhat^\times$ and $-I \in G$.  Assume that $X_G$ has genus at least $54$ and that $G$ is not conjugate in $\GL_2(\Zhat)$ to $G_1$ or $G_2$.   Then $X_G(\QQ)$ has no non-CM points.
\end{conj}

\begin{remark}
It can be shown that there are \emph{infinitely many} groups $G\subseteq \GL_2(\Zhat)$ with $\det(G)=\Zhat^\times$ and $-I\in G$ such that $X_G$ has genus at least $53$ and $X_G(\QQ)$ has a non-CM point.  So the value $54$ in Conjecture~\ref{C:2023} would be best possible.
\end{remark}

\subsection{Motivation for our conjectures}

In this section, we give some brief motivation behind Conjectures~\ref{C:brave} and \ref{C:2023}.  We shall assume throughout that that Conjecture~\ref{C:Serre question} holds.    Let 
\[
\calG_1,\ldots, \calG_m
\]
be the subgroups of $\GL_2(\Zhat)$, up to conjugacy, that are the agreeable closures of $\rho_E^*(\Gal_\QQ)$ for some non-CM $E/\QQ$.    Using  Theorem~\ref{T:main agreeable} (with Lemma~\ref{L:level of agreeable and primes}), we find that there are indeed only finitely many groups $\calG_i$.  

Take any $1\leq i \leq m$.  Consider any group $G \in \scrF(\calG_i,[\calG_i,\calG_i])$.   Let $g_i$ be the genus of the curve $X_{G}$ and let $n_i$ be the index of $G\cap \SL_2(\Zhat)$ in $\SL_2(\Zhat)$.    Since $g_i$ and $n_i$ depend only on $G  \cap \SL_2(\Zhat) = [\calG_i,\calG_i]$, we find that $g_i$ and $n_i$ are independent of the choice of $G$.   

We now prove a version of Conjecture~\ref{C:brave} (assuming Conjecture~\ref{C:Serre question}).  We define $\calI$ to be the set of integers $n_i$ with $1\leq i \leq m$.  

\begin{lemma} \label{L:brave}
If $E/\QQ$ is a non-CM elliptic curve, then $[\GL_2(\Zhat):\rho_E(\Gal_\QQ)]$ lies in the set $\calI$.
\end{lemma}
\begin{proof}
Take any non-CM elliptic curve $E/\QQ$ and let $\calG$ be the agreeable closure of $G_E:=\rho_E^*(\Gal_\QQ)$.   After conjugating in $\GL_2(\Zhat)$, we may assume that $G_E\subseteq \calG=\calG_i$ for some $1\leq i \leq m$.   We have $G_E \cap \SL_2(\Zhat) = [\calG,\calG]=[\calG_i,\calG_i]$ and hence $G_E$ is in the family $\scrF(\calG_i,[\calG_i,\calG_i])$.  Therefore, $[\GL_2(\Zhat): G_E]=[\SL_2(\Zhat): G_E \cap \SL_2(\Zhat)] =n_i$.  In particular, $[\GL_2(\Zhat): G_E] \in \calI$.   
\end{proof}

We now prove a version of Conjecture~\ref{C:2023} (assuming Conjecture~\ref{C:Serre question}).  We define $\beta$ to be the maximum value of $g_i$ as we vary over all $1\leq i \leq m$ for which the image of $\calG_i$ modulo $37$ is not conjugate in $\GL_2(\ZZ/37\ZZ)$ to a group of upper triangular matrices.    

\begin{lemma} \label{L:2023}
Let $G$ be an open subgroup  of $\GL_2(\Zhat)$ with $\det(G)=\Zhat^\times$ and $-I \in G$.  Assume that $X_G$ has genus strictly greater than $\max\{\beta,49\}$ and that $G$ is not conjugate in $\GL_2(\Zhat)$ to $G_1$ or $G_2$.   Then $X_G(\QQ)$ has no non-CM points.
\end{lemma}
\begin{proof}
Let $G$ be an open subgroup  of $\GL_2(\Zhat)$ with $\det(G)=\Zhat^\times$ and $-I \in G$.  Assume that the genus $g$ of $X_G$ is strictly greater than $\max\{\beta,49\}$ and that $X_G(\QQ)$ has a non-CM point.  We need to show that $G$ is conjugate to $G_1$ or $G_2$.

Since $X_G(\QQ)$ has a non-CM point, there is a non-CM elliptic curve $E/\QQ$ so that, after conjugating, we have an inclusion $G_E:=\rho_E^*(\Gal_\QQ)\subseteq G$.   After conjugating our groups appropriately, we may further assume that $\calG_i$ is the agreeable closure of $G_E$ for some $1\leq i \leq m$.    We have $G_E \in \scrF(\calG_i, [\calG_i,\calG_i])$ since $G_E \cap \SL_2(\Zhat) = [G_E,G_E]=[\calG_i,\calG_i]$.   Therefore, $X_{G_E}$ has genus $g_i$.  The inclusion $G_E\subseteq G$ implies that $g_i\geq g$.    Since $g_i\geq g>\beta$, we deduce that $\calG_i$ modulo $37$ is conjugate to a group of upper triangular matrices.  From the inclusions $G_E \subseteq G \subseteq \calG_i$, we deduce that the modular curve $X_0(37)$ has a rational non-CM point arising from $E/\QQ$.   

From the beginning of \S\ref{SS:conjecture 2023}, we find that $\pm G_E$ is conjugate to $G_1$ or $G_2$ in $\GL_2(\Zhat)$.  So after conjugating $G$ appropriately, we may now assume that $G_j \subseteq G$ for some $j\in \{1,2\}$.  A computation shows that $X_{G_j}$ has genus $97$ and that $X_{G'}$ has genus at most $49$ for every group $G_j \subsetneq G' \subseteq \GL_2(\Zhat)$.   Since $X_G$ has genus $g>49$, we deduce that $G=G_j$ (after conjugation).
\end{proof}

We now explain how we made a conjecture for the list of groups $\calG_1,\ldots, \calG_m$.    By Theorem~\ref{T:main agreeable}, we can explicitly determine all the groups $\calG_i$ for which $X_{\calG_i}(\QQ)$ is infinite.

Let $\calJ$ be the set  (\ref{E:calJ new}) and let $\calJ'$ be the subset of $\calJ$ consisting of rational numbers that are the $j$-invariants of non-CM elliptic curves.  Equivalently, $\calJ'$ is the set of $j$-invariants of non-CM elliptic curves $E/\QQ$ for which $X_\calG(\QQ)$ is finite, where $\calG$ is the agreeable closure of $\rho_E^*(\Gal_\QQ)$.  Using Lemma~\ref{L:level of agreeable and primes} and Theorem~\ref{T:main agreeable}, the groups $\calG_i$ for which $X_{\calG_i}(\QQ)$ is finite are all obtained by taking the agreeable closure of $\rho_E^*(\Gal_\QQ)$ with $E/\QQ$ an elliptic curve whose $j$-invariant lies in $\calJ'$.  

So to find the groups $\calG_1,\ldots, \calG_m$, it would suffice to first compute $\calJ'$.  Note that the set $\calJ'$ is finite by Faltings theorem (which is ineffective).   Unfortunately,  $\calJ'$ is extremely difficult to compute.    As noted in \S\ref{SS:agreeable intro}, we have found $81$ elements of $\calJ'$; this was done by searching for low height rational points on models of the relevant modular curves.   We conjecture that $\calJ'$ actually has cardinality $81$ and hence we conjecturally know $\calJ'$.      In particular, this conjecture would allow us to compute our sequence of agreeable groups $\calG_1,\ldots, \calG_m$.    With this explicit list of groups, direct computations lead to the predictions that $\calI$ is the set from the statement of Conjecture~\ref{C:brave} and that $\beta$ is $53$.  These conjectural values of $\calI$ and $\beta$ along with Lemmas~\ref{L:brave} and \ref{L:2023} give rise to Conjectures~\ref{C:brave} and \ref{C:2023}, respectively.

\end{document}